\documentclass[intlim,righttag,10pt]{amsart}
\usepackage{amscd}
\usepackage{amssymb}
\usepackage[all]{xy}
\RequirePackage{mathrsfs}
\usepackage[pdftex, pdfborder= 0 0 0, citecolor=black, urlcolor=black,
  linkcolor=black, colorlinks=true, bookmarksopen=true]{hyperref}

\def\jcdot{\scriptscriptstyle\bullet}
\def\twomapr#1#2{\smash{\mathop{\longrightarrow}\limits^{#1}_{#2}}}
\def\invlim{\mathop{\vtop{\ialign{##\crcr$\hfill{\lim}\hfil$\crcr
\noalign{\kern1pt\nointerlineskip}\leftarrowfill\crcr\noalign
{\kern -3pt}}}}\limits}
\def\dirlim{\mathop{\vtop{\ialign{##\crcr$\hfill{\lim}\hfil$\crcr
\noalign{\kern1pt\nointerlineskip}\rightarrowfill\crcr\noalign
{\kern -3pt}}}}\limits}
\def\lomapr#1{\smash{\mathop{\relbar\joinrel\longrightarrow}\limits^{#1}}}
 \def\verylomapr#1{\smash{\mathop{\relbar\joinrel\relbar\joinrel\relbar\joinrel\longrightarrow}\limits^{#1}}}
\def\veryverylomapr#1{\smash{\mathop{\relbar\joinrel\relbar\joinrel\relbar
\joinrel\relbar\joinrel\relbar\joinrel\longrightarrow}\limits^{#1}}}
\def\phi{\varphi}
\def\epsilon{\varepsilon}
\let\mathcal\mathscr

\newtheorem{thmx}{Theorem}
\newtheorem{theorem}{Theorem}[section]
 \newtheorem{lemma}[theorem]{Lemma}
 \newtheorem{proposition}[theorem]{Proposition}
 \newtheorem{corollary}[theorem]{Corollary}

\theoremstyle{definition}
\newtheorem{definition}[theorem]{Definition}
\newtheorem{remark}[theorem]{Remark}

\newtheorem{example}[theorem]{Example}
\newtheorem*{acknowledgments}{Acknowledgments}

\newcommand{\hk}{\mathrm{HK}}
\newcommand{\dr}{\mathrm{dR}}
\newcommand{\ove}{\overline}
\newcommand{\what}{\widehat}
\newcommand{\Lie}{\operatorname{Lie}}
\newcommand{\Qp}{\mathbf{Q}_p}
\newcommand{\Zp}{\mathbf{Z}_p}
\renewcommand{\phi}{\varphi}

\newcommand{\ddr}{\mathrm{D}_{\mathrm{dR}}}
\newcommand{\dfont}{\mathrm{D}}

\newcommand{\dcroc}[1]{[\![ #1 ]\!]}
\newcommand{\nc}{nc}
\newcommand{\eff}{eff}
\newcommand{\pcr}{pcr}
\newcommand{\nd}{nd}
\newcommand{\ovu}{\overline{U}}
\newcommand{\R}{\mathrm {R} }
\newcommand{\LL}{\mathrm {L} }
\newcommand{\pst}{\mathrm{pst}}

\newcommand{\pri}{^{\prime}}
\newcommand{\bq}{{\mathbf Q}}
\newcommand{\bz}{{\mathbf Z}}

\newcommand{\norm}{\operatorname{norm} }

\newcommand{\ovk}{\overline{K} }
 \newcommand{\ovv}{\overline{V} }

\newcommand{\ord}{\operatorname{ord} }

\newcommand{\gp}{\operatorname{gp} }
\newcommand{\ad}{\operatorname{ad} }

\newcommand{\pr}{\operatorname{pr} }

\newcommand{\Pic}{\operatorname{Pic} }

 \newcommand{\holim}{\operatorname{holim} }

 \newcommand{\cont}{\operatorname{cont}  }
  \newcommand{\proeet}{\operatorname{pro\acute{e}t}  }

 \newcommand{\eet}{\operatorname{\acute{e}t} }

 \newcommand{\dlog}{\operatorname{dlog} }

 \newcommand{\ZAR}{\operatorname{ZAR} }

 \newcommand{\nr}{\operatorname{nr} }
 
 \newcommand{\Spec}{\operatorname{Spec} }

 \newcommand{\Spf}{\operatorname{Spf} }
 
 \newcommand{\Hom}{\operatorname{Hom} }

 \newcommand{\Ext}{\operatorname{Ext} }
 \newcommand{\Rep}{\operatorname{Rep} }

 \newcommand{\Gal}{\operatorname{Gal} }
 \newcommand{\tr}{ \operatorname{tr} }
 \newcommand{\can}{ \operatorname{can} }
 
 \newcommand{\id}{ \operatorname{Id} }
\newcommand{\synt}{ \operatorname{syn} }
 
 \newcommand{\Cone}{\operatorname{Cone} }

\newcommand{\st}{\operatorname{st} }
\newcommand{\bk}{\operatorname{BK} }

 \newcommand{\kker}{\operatorname{Ker} }
 \newcommand{\crr}{\operatorname{cr} }
 \newcommand{\gr}{\operatorname{gr} }
 \newcommand{\im}{\operatorname{Im} }

 \newcommand{\kr}{^{\cdot }}

 \newcommand{\sff}{{\mathcal{F}}}

 \newcommand{\sh}{{\mathcal{H}}}

 \newcommand{\sv}{{\mathcal{V}}}

 \newcommand{\sk}{{\mathcal{K}}}
 \newcommand{\sll}{{\mathcal{L}}}
 
 \newcommand{\so}{{\mathcal O}}
 \newcommand{\sj}{{\mathcal J}}
 \newcommand{\se}{{\mathcal{E}}}
 \newcommand{\sa}{{\mathcal{A}}}
 
 \newcommand{\sx}{{\mathcal{X}}}
 \newcommand{\sss}{{\mathcal{S}}}
\newcommand{\sd}{{\mathcal{D}}}
\newcommand{\sm}{{\mathcal{M}}}
\newcommand{\spp}{{\mathcal{P}}}

 \newcommand{\wh}{\widehat}

 \newcommand{\E}{\mathbb{E}}
\DeclareMathOperator{\PSh}{PSh}
\DeclareMathOperator{\Sp}{Sp_h}
\newcommand{\DMe}{DM^{eff}}
\newcommand{\DM}{DM}
\newcommand{\Af}{\mathbb A^1}

\makeindex
 
\begin{document}
 \title[Syntomic cohomology and  $p$-adic regulators for varieties over $p$-adic fields]
 {Syntomic cohomology and $p$-adic regulators for varieties over $p$-adic fields}
 \author{Jan Nekov\'a\v{r}, Wies{\l}awa Nizio{\l}}
 \date{\today}
\thanks{The authors' research was supported in part by the grant ANR-BLAN-0114 and by  the NSF grant DMS0703696, respectively.}
 \email{jan.nekovar@imj-prg.fr, wieslawa.niziol@ens-lyon.fr}
 \begin{abstract}
 We show that the logarithmic version of the syntomic cohomology of Fontaine and Messing for semistable varieties over $p$-adic rings  extends uniquely to a cohomology theory for varieties over $p$-adic fields that satisfies $h$-descent. This new cohomology -  syntomic cohomology -  is a  Bloch-Ogus cohomology theory, admits period map to \'etale cohomology, and has a syntomic descent spectral sequence (from an algebraic closure of the given field to the field itself) that is compatible with the Hochschild-Serre spectral sequence on the \'etale side and is related to the Bloch-Kato exponential map. In  relative dimension zero  we recover the potentially semistable  Selmer groups and,
as an application, we prove that Soul\'e's \'etale regulators land in the potentially semistable  Selmer groups.

  Our construction of syntomic cohomology is based on new ideas and techniques developed by Beilinson and Bhatt in their recent work on $p$-adic comparison theorems.
 \end{abstract}
 \maketitle
 \tableofcontents
 \section{Introduction}In this article we define syntomic cohomology for varieties over $p$-adic fields, relate it to the Bloch-Kato exponential map,  and use it to study the images of Soul\'e's \'etale regulators. Contrary to all the previous constructions of syntomic cohomology (see below for a brief review) we do not restrict ourselves to varieties coming with a nice model over the integers. Hence our syntomic regulators make no integrality assumptions on the $K$-theory classes in the domain.
\subsection{Statement of the main result} 
 Recall that, for varieties proper and smooth over a $p$-adic ring of mixed characteristic,
 syntomic cohomology (or its non-proper variant:  syntomic-\'etale cohomology) was introduced by Fontaine and Messing \cite{FM} in their proof of the Crystalline Comparison Theorem as a natural bridge between crystalline cohomology and \'etale cohomology. It was generalized  to log-syntomic cohomology for semistable varieties by Kato \cite{Kas}. For a log-smooth scheme $\sx$ over a complete discrete valuation ring $V$ of mixed characteristic $(0,p)$ and a perfect residue field, and for any $r\geq 0$, rational log-syntomic cohomology of $\sx$  can be defined as the "filtered Frobenius eigenspace" in log-crystalline cohomology, i.e., as the following mapping fiber
\begin{equation}
\label{first}
\R\Gamma_{\synt}(\sx,r):=\Cone\big(\R\Gamma_{\crr}(\sx,\sj^{[r]})\lomapr{1-\phi_r}\R\Gamma_{\crr}(\sx)\big)[-1],
\end{equation}
where $\R\Gamma_{\crr}(\cdot,\sj^{[r]})$ denotes the  absolute  rational log-crystalline cohomology (i.e., over ${\mathbf Z}_p)$ of the $r$'th Hodge filtration sheaf $\sj^{[r]}$ and $\phi_r$ is the crystalline Frobenius divided by $p^r$.
 This definition  suggested that the log-syntomic cohomology could be the sought for  $p$-adic  analog  of Deligne-Beilinson cohomology. Recall that, for a complex manifold $X$,  the latter can be defined as the cohomology $\R\Gamma(X,{\mathbf Z}(r)_{\sd})$ of Deligne complex ${\mathbf Z}(r)_{\sd}$:
$$0\to {\mathbf Z}(r)\to \Omega^1_X\to\Omega^2_X\to\ldots \to \Omega^{r-1}_X\to 0
$$
 And, indeed, since its introduction, log-syntomic cohomology  has been used with some success in the study   of special values of $p$-adic $L$-functions and in formulating $p$-adic Beilinson conjectures (cf. \cite{BE} for a review).

  The syntomic cohomology theory with
$\Qp$-coefficients $R\Gamma_{\synt}(X_h,r)$ ($r\geq 0$) for arbitrary
varieties -- more generally, for arbitrary essentially finite diagrams
of varieties -- over the $p$-adic field $K$ (the fraction field of $V$) that we construct in this article is a generalization of Fontaine-Messing(-Kato) log-syntomic cohomology.  That is,  for a semistable scheme \footnote{Throughout the Introduction, the divisors at infinity of semistable schemes have no multiplicities.} $\sx$ over $V$ we have $\R\Gamma_{\synt}(\sx,r)\simeq\R\Gamma_{\synt}(X_h,r)$, where $X$ is the largest subvariety of $\sx_K$ with trivial log-structure. An analogous theory $R\Gamma_{\synt}(X_{\ovk,h},r)$ ($r\geq 0$) exists
for (diagrams of) varieties over $\ovk$, where $\ovk$ is an algebraic closure of $K$.

   Our main result can be stated as follows.
\begin{thmx}
\label{main1}
For any variety $X$ over $K$, there is a canonical  graded commutative dg $\Qp$-algebra $\R\Gamma_{\synt}(X_h,*)$ such that
\begin{enumerate}
\item it is the unique extension of  log-syntomic cohomology to varieties over $K$ that satisfies $h$-descent, i.e., for any hypercovering $\pi: Y_{\jcdot}\to X$ in $h$-topology, we have a quasi-isomorphism 
$$\pi^*:\R\Gamma_{\synt}(X_h,*)\stackrel{\sim}{\to}\R\Gamma_{\synt}(Y_{\jcdot,h},*).
$$
\item it is a Bloch-Ogus cohomology theory \cite{BO}.
\item for $X=\Spec(K)$, $H^*_{\synt}(X_h,r)\simeq H^*_{\st}(G_K,\bq_p(r))$, where $H^i_{\st}(G_K, -)$ denotes the $\Ext$-group $\Ext^i(\Qp, -)$ in the
category of (potentially) semistable representations of $G_K=\Gal(\ovk/K)$.
\item
There are functorial syntomic period morphisms
\[
\rho_{\synt}: R\Gamma_{\synt}(X_h,r)\to R\Gamma(X_{\eet},{\mathbf Q}_p(r)),\qquad
\rho_{\synt}: R\Gamma_{\synt}(X_{\ovk,h},r) \to R\Gamma(X_{\ovk,\eet},{\mathbf Q}_p(r))
\]
compatible with products which induce quasi-isomorphisms
\[
\tau_{\leq r} R\Gamma_{\synt}(X_h,r) \stackrel{\sim}{\to}
\tau_{\leq r} R\Gamma(X_{\eet},{\mathbf Q}_p(r)),\qquad
\tau_{\leq r} R\Gamma_{\synt}(X_{\ovk,h},r) \stackrel{\sim}{\to}
\tau_{\leq r} R\Gamma(X_{\ovk,\eet},{\mathbf Q}_p(r)).
\]
\item
 The Hochschild-Serre spectral sequence for \'etale cohomology
\[
^{\eet}E^{i,j}_2 = H^i(G_K,H^j(X_{\ovk,\eet},{\mathbf Q}_p(r)))
\Longrightarrow H^{i+j}(X_{\eet},{\mathbf Q}_p(r))
\]
has a syntomic analog
\[
^{\synt}E^{i,j}_2 = H^i_{\st}(G_K,H^j(X_{\ovk,\eet},{\mathbf Q}_p(r)))
\Longrightarrow H^{i+j}_{\synt}(X_{h},r).
\]
\item
 There is a canonical morphism of spectral sequences
${}^{\synt}E_t \to {}^{\eet}E_t$ compatible with the syntomic period map. 
\item
 There are syntomic Chern classes
\[
{c}_{i,j}^{\synt} \colon K_j(X) \to H^{2i-j}_{\synt}(X_{h},i)
\]
compatible with \'etale Chern classes via the syntomic period map.
\end{enumerate}
\end{thmx}

  As is shown in \cite{DN}, syntomic cohomology $R\Gamma_{\synt}(X_h,*)$  can be interpreted as an absolute $p$-adic Hodge cohomology. That is, it is a derived $\Hom$ in the category of admissible $(\phi, N, G_K)$-modules between the trivial module and a complex of such modules canonically associated to a variety. Alternatively, it is a derived $\Hom$ in the category of potentially semistable representations between the trivial representation  and a complex of such representations  canonically associated to a variety. A particularly simple construction of such a complex, using Beilinson's Basic Lemma, was proposed by Beilinson (and is presented in \cite{DN}).  The category of modules over the syntomic cohomology algebra $R\Gamma_{\synt}(X_h,*)$ (taken in a  motivic sense) yields a category of $p$-adic Galois representations that better approximates the category of geometric representations than the category of potentially semistable representations \cite{DN}. For further applications of syntomic cohomology algebra we refer the interested reader to Op. cit. 
  
  Similarly, as is shown in \cite{HT}, geometric syntomic cohomology $R\Gamma_{\synt}(X_{\ovk,h},*)$   is a derived $\Hom$ in the category of effective $\phi$-gauges (with one paw) \cite{LF1} between the trivial gauge  and a complex of such gauges canonically associated to a variety. In particular, geometric syntomic cohomology group is a  finite dimensional Banach-Colmez Space \cite{CB}, hence has a very rigid structure.
  
   The syntomic descent spectral sequence and its compatibility with the Hochschild-Serre spectral sequence in \'etale cohomology imply the following proposition.
\begin{proposition}Let $i\geq 0$. The composition
\begin{align*}
H^{i-1}_{\dr}(X)/F^r \stackrel{\partial}{\to} H^i_{\synt}(X_h,r)\stackrel{\rho_{\synt}}{\longrightarrow} H^i_{\eet}(X,\Qp(r))\to H^i_{\eet}(X_{\ovk},\Qp(r))
\end{align*}
is the zero map. The induced (from the syntomic descent spectral sequence) map 
$$
H^{i-1}_{\dr}(X)/F^{r}\to H^1(G_K,H^{i-1}_{\eet}(X_{\ovk},\Qp(r)))
$$
 is equal to the  Bloch-Kato exponential associated with the Galois representation $H^{i-1}_{\eet}(X_{\ovk},\Qp(r))$.
\end{proposition}
This yields a comparison between $p$-adic \'etale regulators, syntomic
regulators, and the Bloch-Kato exponential (which was proved in the
good reduction case in \cite{JS} and \cite[Thm. 5.2]{NS}\footnote{The Bloch-Kato exponential  is called $l$ there.}) that is of fundamental importance
for theory of special values of $L$-functions, both complex valued
and $p$-adic. The point is that syntomic regulators  can be thought of as an abstract $p$-adic integration theory. 
The comparison results stated
above then relate certain  $p$-adic integrals to the values of the 
$p$-adic \'etale regulator via the Bloch-Kato exponential map. A modification of syntomic cohomology  developed in \cite{BS}
in the good reduction case (resp. in \cite{BLZ} -- using the techniques
of the present article -- in the case of arbitrary reduction) can be used to perform explicit computations.
For example, the formulas from \cite[\S{3}]{BLZ} were applied  to 
a calculation of certain $p$-adic regulators in \cite{BDR} and in \cite{DR}.

\subsection{Construction of syntomic cohomology}
  We will now sketch  the proof of Theorem \ref{main1}. Recall first that a little bit after log-syntomic cohomology had appeared on the scene, Selmer groups of Galois representations -- describing extensions in certain categories of Galois representations -- were introduced by Bloch and Kato \cite{BK} and linked to special values of $L$-functions.  And a syntomic cohomology (in the good reduction case), a priori different than that of Fontaine-Messing,  was defined in  \cite{NS} and by Besser in  \cite{BS} as a higher dimensional analog of the complexes computing these groups. The guiding idea here was that just as Selmer groups classify extensions in certain categories of "geometric" Galois representations their higher dimensional analogs -- syntomic cohomology groups -- should classify extensions in a category of "$p$-adic motivic sheaves". This was shown to be the case for $H^1$ by Bannai \cite{Ban} who has also shown that Besser's (rigid) syntomic cohomology is a $p$-adic analog of Beilinson's absolute Hodge cohomology \cite{BE0}.

   Complexes computing the semistable and potentially semistable Selmer groups were introduced in \cite{JH} and  \cite{FPR}. For a semistable scheme $\sx$ over $V$, their higher dimensional analog can be written as the following homotopy limit\footnote{See section 1.5 for an explanation of the notation we use for certain homotopy limits.}
\begin{equation}
\label{first2}
\R\Gamma^{\prime}_{\synt}(\sx,r):=\left[
\begin{aligned}\xymatrix@=40pt{
  \R\Gamma_{\hk}(\sx_0)\ar[d]^{N}\ar[r]^-{(1-\phi_r,\iota_{\dr})} &
\R\Gamma_{\hk}(\sx_0)\oplus \R\Gamma_{\dr}(\sx_K)/F^r\ar[d]^{(N,0)}\\
  \R\Gamma_{\hk}(\sx_0)\ar[r]^-{1-\phi_{r-1}} & \R\Gamma_{\hk}(\sx_0)
 }\end{aligned}\right]
\end{equation}
where $\sx_0$ is the special fiber of $\sx$, $\R\Gamma_{\hk}(\cdot)$ is the Hyodo-Kato cohomology,  $N$ denotes the Hyodo-Kato monodromy, and $\R\Gamma_{\dr}(\cdot)$ is the logarithmic de Rham cohomology. The map $\iota_{\dr}$ is the Hyodo-Kato morphism that induces a quasi-isomorphism $\iota_{\dr}:\R\Gamma_{\hk}(\sx_0)\otimes_{K_0}K\stackrel{\sim}{\to}\R\Gamma_{\dr}(\sx_K)$, for $K_0$ - the fraction field of Witt vectors of the residue field of $V$.

  Using Dwork's trick, we prove (cf. Proposition \ref{reduction1}) that the two definitions of log-syntomic cohomology are the same, i.e., that there is  a  quasi-isomorphism $$\alpha_{\synt}: \R\Gamma_{\synt}(\sx,r)\stackrel{\sim}{\to}\R\Gamma_{\synt}^{\prime}(\sx,r).
$$ It follows that log-syntomic cohomology groups vanish in degrees strictly higher than $2\dim X_K +2$ and that, if $\sx=\Spec(V)$, then $H^i\R\Gamma_{\synt}(\sx,r)\simeq H^i_{\st}(G_K,\bq_p(r))$.

  The  syntomic cohomology for varieties over $p$-adic fields that we introduce in this article is a generalization of the log-syntomic cohomology of Fontaine and Messing. Observe that it is clear how one can try to use log-syntomic cohomology to define syntomic cohomology for varieties over fields that satisfies $h$-descent. Namely, for a variety $X$ over $K$, consider the $h$-topology of $X$ and recall that (using alterations) one can show  that it  has a basis consisting of semistable models over  finite extensions of $V$ \cite{BE1}. By $h$-sheafifying the complexes $Y\mapsto \R\Gamma_{\synt}(Y,r)$  (for a semistable model $Y$)  we get  syntomic complexes $\sss(r)$. We define the ({\it arithmetic}) syntomic cohomology as 
$$
\R\Gamma_{\synt}(X_h,r):=\R\Gamma(X_h,\sss(r)).
$$

   A priori it is not clear  that the so defined syntomic  cohomology   behaves well: the finite  ramified field extensions introduced by alterations  are in general a problem for log-crystalline cohomology. For example, the related complexes $\R\Gamma_{\crr}(X_h,\sj^{[r]})$ are huge. However, taking Frobenius eigenspaces cuts off the "noise" and the resulting syntomic complexes do indeed behave well. To get an idea why this is the case,   $h$-sheafify the complexes $Y\mapsto \R\Gamma^{\prime}_{\synt}(Y,r)$ and imagine that you can sheafify the maps $\alpha_{\synt}$ as well. We get sheaves  $\sss^{\prime}(r)$ and  quasi-isomorphisms  $\alpha_{\synt}:\sss(r)\stackrel{\sim}{\to}\sss^{\prime}(r)$. Setting  $\R\Gamma^{\prime}_{\synt}(X_h,r):=\R\Gamma(X_h,\sss^{\prime}(r))$ we  obtain the following quasi-isomorphisms
\begin{equation}
\label{first3}
\R\Gamma_{\synt}(X_h,r)\simeq \R\Gamma^{\prime}_{\synt}(X_h,r)\simeq \left[
\begin{aligned}\xymatrix@=40pt{
  \R\Gamma_{\hk}(X_h)\ar[d]^{N}\ar[r]^-{(1-\phi_r,\iota_{\dr})} &
\R\Gamma_{\hk}(X_h)\oplus \R\Gamma_{\dr}(X_K)/F^r\ar[d]^{(N,0)}\\
  \R\Gamma_{\hk}(X_h)\ar[r]^-{1-\phi_{r-1}} & \R\Gamma_{\hk}(X_h)
 }\end{aligned}\right]
\end{equation}
where $\R\Gamma_{\hk}(X_h)$ denotes the Hyodo-Kato cohomology (defined  as $h$-cohomology of the  presheaf:   $Y\mapsto \R\Gamma_{\hk}(Y_0)$) and $\R\Gamma_{\dr}(\cdot)$ is Deligne's de Rham cohomology \cite{De}. The Hyodo-Kato map $\iota_{\dr}$ is the $h$-sheafification of the logarithmic Hyodo-Kato map. It is well-known that Deligne's de Rham cohomology  groups are finite rank $K$-vector spaces; it turns out that the Hyodo-Kato cohomology groups are finite rank  $K_0$-vector spaces: we have a quasi-isomorphism 
$\R\Gamma_{\hk}(X_h)\stackrel{\sim}{\to}\R\Gamma_{\hk}(X_{\ovk,h})^{G_K}$ and the geometric Hyodo-Kato groups $H^*\R\Gamma_{\hk}(X_{\ovk,h})$ are finite rank  $K_0^{\nr}$-vector spaces, where $K_0^{\nr}$ is the maximal unramified extension of $K_0$ (see (\ref{trivialization})  below). 

It follows that syntomic cohomology groups  vanish in degrees higher than $2\dim X_K+2$ and that syntomic cohomology is, in fact, a generalization of the classical log-syntomic cohomology, i.e.,  for a semistable scheme $\sx$ over $V$ we have $\R\Gamma_{\synt}(\sx,r)\simeq\R\Gamma_{\synt}(X_h,r)$, where $X$ is the largest subvariety of $\sx_K$ with trivial log-structure. This follows from the quasi-isomorphism $\alpha_{\synt}$:  logarithmic Hyodo-Kato and de Rham cohomologies (over a fixed base) satisfy proper descent and the finite fields extensions that appear as the "noise" in  alterations do not destroy anything  since logarithmic Hyodo-Kato and de Rham cohomologies satisfy finite Galois descent.

 Alas, we were not able to sheafify the map $\alpha_{\synt}$. The reason for that is that the construction of $\alpha_{\synt}$ uses a twist by a high power of Frobenius  -- a power depending on the field $K$. And alterations are going to introduce a finite extension of $K$ -- hence a need for higher and higher powers of Frobenius. So instead we construct directly the map 
$\alpha_{\synt}:\R\Gamma_{\synt}(X_h,r)\to\R\Gamma^{\prime}_{\synt}(X_h,r) $. To do that we show first that the syntomic cohomological dimension of $X$ is finite. Then we take a semistable $h$-hypercovering of $X$, truncate it at an appropriate level, extend the base field $K$ to $K^{\prime}$, and base-change everything to $K^{\prime}$. There we can work with one field and use the map $\alpha_{\synt}$ defined earlier. Finally, we show that we can descend. 

\subsection{Syntomic period maps} 
We pass now to the construction of the period maps from  syntomic to \'etale cohomology that appear in Theorem \ref{main1}. They are easier to define over $\ovk$, i.e., from the {\it geometric} syntomic cohomology. In this setting, things go smoother with $h$-sheafification  since
going all the way up to $\ovk$ before completing kills a lot of "noise" in log-crystalline cohomology. More precisely, 
 for a semistable scheme $\sx$ over $V$, we have the following canonical quasi-isomorphisms \cite{BE2}
\begin{equation}
\label{trivialization}
\iota_{\crr}: \R\Gamma_{\hk}(\sx_{\ovv})^{\tau}_{B^+_{\crr}}\stackrel{\sim}{\to}\R\Gamma_{\crr}(\sx_{\ovv}),\quad \iota_{\dr}: \R\Gamma_{\hk}(\sx_{\ovv})^{\tau}_{\ovk}\stackrel{\sim}{\to} \R\Gamma_{\dr}(\sx_{\ovk}),
\end{equation}
  where $\ovv$ is the integral closure of $V$ in $\ovk$,  $B^+_{\crr}$ is the crystalline  period ring,  and $\tau$ denotes certain twist. These quasi-isomorphisms $h$-sheafify well: for a variety $X$over $K$, they induce the following  quasi-isomorphisms \cite{BE2}
\begin{equation}
\label{trivialization1}
\iota_{\crr}: \R\Gamma_{\hk}(X_{\ovk,h})^{\tau}_{B^+_{\crr}}\stackrel{\sim}{\to}\R\Gamma_{\crr}(X_{\ovk,h}),\quad \iota_{\dr}: \R\Gamma_{\hk}(X_{\ovk,h})^{\tau}_{\ovk}\stackrel{\sim}{\to} \R\Gamma_{\dr}(X_{\ovk}),
\end{equation}
where the terms have obvious meaning. Since
 Deligne's de Rham cohomology has proper descent (by definition),  it follows that  $h$-crystalline cohomology behaves well. That is, if
we define crystalline  sheaves
 $\sj^{[r]}_{\crr}$ and $\sa_{\crr}$ on $X_{\ovk,h}$ by $h$-sheafifying the complexes $Y\mapsto \R\Gamma_{\crr}(Y,\sj^{[r]})$ and $Y\mapsto \R\Gamma_{\crr}(Y)$, respectively, for $Y$ which are a base change to $\ovv$ of a semistable scheme over a finite extension of $V$ (such schemes $Y$ form a basis of $X_{\ovk,h}$) then the complexes $\R\Gamma(X_{\ovk,h},\sj^{[r]})$ and $\R\Gamma_{\crr}(X_{\ovv,h}):=\R\Gamma(X_{\ovk,h},\sa_{\crr})$ generalize log-crystalline cohomology (in the sense described above) and the latter one is a perfect complex of $B^+_{\crr}$-modules.

  We obtain syntomic complexes
 $\sss(r)$ on $X_{\ovk,h}$ by $h$-sheafifying the complexes $Y\mapsto \R\Gamma_{\synt}(Y,r)$ and (geometric) syntomic cohomology by setting $\R\Gamma_{\synt}(X_{\ovk,h},r):=\R\Gamma(X_{\ovk,h},\sss(r))$. They fit into  an analog of the exact sequence (\ref{first}) and, by the above,  
 generalize  log-syntomic cohomology.

 To construct the syntomic period maps
 \begin{equation}
 \label{periodss}
 \rho_{\synt}: \R\Gamma_{\synt}(X_{\ovk,h},r)\to \R\Gamma(X_{\ovk,\eet},\bq_p(r)),\quad  \rho_{\synt}: \R\Gamma_{\synt}(X_{h},r)\to \R\Gamma(X_{\eet},\bq_p(r))
 \end{equation}
 consider the syntomic complexes $\sss_n(r)$:
 the mod-$p^n$ version of the syntomic complexes $\sss(r)$ on $X_{\ovk,h}$. We have the distinguished triangle
$$\sss_n(r)\to \sj_{\crr,n}^{[r]}\lomapr{p^r-\phi} \sa_{\crr,n}
$$
Recall that the filtered Poincar\'e Lemma of Beilinson and Bhatt \cite{BE2}, \cite{BH} yields a quasi-isomorphism $\rho_{\crr}: J_{\crr,n}^{[r]}\stackrel{\sim}{\to} \sj_{\crr,n}^{[r]}$,
where $J^{[r]}_{\crr}\subset A_{\crr}$ is the $r$'th filtration level of the period ring $A_{\crr}$. Using the  fundamental sequence of $p$-adic Hodge Theory $$0\to {\mathbf Z}/p^n(r)^{\prime}\to J^{<r>}_{\crr,n}\lomapr{1-\phi_r} A_{\crr,n}\to 0,$$
where ${\mathbf Z}/p^n(r)^{\prime}:=(1/(p^aa!){\mathbf Z}_p(r))\otimes {\mathbf Z}/p^n$ and $a$ denotes the largest integer $\leq r/(p-1)$,  we obtain the syntomic period map $\rho_{\synt}:\sss_n(r)\to {\mathbf Z}/p^n(r)^{\prime}$. It is a quasi-isomorphism modulo a universal constant. It induces the geometric syntomic period map in (\ref{periodss}),  and, by Galois descent, its arithmetic analog.

       To study the descent spectral sequences from Theorem \ref{main1}, we need to consider the other version of syntomic cohomology, i.e., the complexes
\begin{equation}
\label{first31}
\R\Gamma^{\prime}_{\synt}(X_{\ovk,h},r):=\left[\begin{aligned}\xymatrix@C=40pt{  \R\Gamma_{\hk}(X_{\ovk,h})\otimes_{K_0^{\nr}} B^+_{\st}\ar[r]^-{(1-\phi_r,\iota_{\dr})}\ar[d]^N & \R\Gamma_{\hk}(X_{\ovk,h})\otimes_{K_0^{\nr}} B^+_{\st}\oplus (\R\Gamma_{\dr}(X_{\ovk})\otimes _{\ovk}B^+_{\dr})/F^r\ar[d]^{(N,0)}\\
  \R\Gamma_{\hk}(X_{\ovk,h})\otimes_{K_0^{\nr}} B^+_{\st}\ar[r]^{1-\phi_{r-1}} &  \R\Gamma_{\hk}(X_{\ovk,h})\otimes _{K_0^{\nr}}B^+_{\st}
  }\end{aligned}\right]
\end{equation}
where $B^+_{\st}$ and $B^+_{\dr}$ are the semistable and de Rham $p$-adic period rings, respectively. 
We deduce a quasi-isomorphism 
$\R\Gamma_{\synt}(X_{\ovk,h},r)\stackrel{\sim}{\to}\R\Gamma^{\prime}_{\synt}(X_{\ovk,h},r)$.
\begin{remark}
This quasi-isomorphism yields, for a semistable scheme $\sx$ over $V$, the following exact sequence
$$
\to H^i_{\synt}(\sx_{\ovk},r)\to (H^i_{\hk}(\sx)_{\bq}\otimes_{K_0} B_{\st}^+)^{\phi=p^r,N=0}\to (H^i_{\dr}(\sx_K)\otimes_K B_{\dr}^+)/F^r\to H^{i+1}_{\synt}(\sx_{\ovk},r)\to
$$
It is a sequence of finite dimensional Banach-Colmez Spaces \cite{CB} and as such is a key in the proof of  semistable comparison theorem for formal schemes in \cite{CN}. 
\end{remark}

    We also have a syntomic period map 
 \begin{equation}
 \label{first4}
 \rho_{\synt}^{\prime}:\R\Gamma^{\prime}_{\synt}(X_{\ovk,h},r)\to \R\Gamma(X_{\ovk,\eet},\bq_p(r))
 \end{equation}
  that is compatible with the map $\rho_{\synt}$ via $\alpha_{\synt}$. To describe how it is constructed, recall that
   the crystalline period map of Beilinson induces compatible Hyodo-Kato and de Rham period maps \cite{BE2}
\begin{equation}
\label{first5}
\rho_{\hk}:\R\Gamma_{\hk}(X_{\ovk,h})\otimes _{K_0^{\nr}}B^+_{\st}{\to} \R\Gamma(X_{\ovk,\eet},\bq_p)\otimes B_{\st}^+,\quad \rho_{\dr}: \R\Gamma_{\dr}(X_K)\otimes_K B_{\dr}^+{\to}\R\Gamma(X_{\ovk,\eet},\bq_p)\otimes B_{\dr}^+
\end{equation}
Applying them to the above homotopy limit, removing all the pluses from the period rings, reduces the homotopy limit to the complex 
\begin{equation}
\label{first32}
\left[\begin{aligned}\xymatrix@C=40pt{  \R\Gamma(X_{\ovk,\eet},\Qp(r))\otimes B_{\st}\ar[r]^-{(1-\phi_r,\iota_{\dr})}\ar[d]^N & \R\Gamma(X_{\ovk,\eet},\Qp(r))\otimes B_{\st}\oplus (\R\Gamma(X_{\ovk,\eet},\Qp(r))\otimes B_{\dr})/F^r\ar[d]^{(N,0)}\\
  \R\Gamma(X_{\ovk,\eet},\Qp(r))\otimes B_{\st}\ar[r]^{1-\phi_{r-1}} &  \R\Gamma(X_{\ovk,\eet},\Qp(r))\otimes B_{\st}
  }\end{aligned}\right]
\end{equation}
By the familiar fundamental exact sequence 
$$
0\to \bq_p(r)\to B_{\st}\verylomapr{(N,1-\phi_r,\iota)}  B_{\st}\oplus B_{\st}\oplus B_{\dr}/F^r\veryverylomapr{(1-\phi_{r-1})-N}B_{\st}\to 0
$$
the above complex is quasi-isomorphic to $\R\Gamma(X_{\ovk,\eet},\Qp(r))$. 
This yields the syntomic period morphism from (\ref{first4}). We like to think of geometric syntomic cohomology as being represented by the complex from (\ref{first31}) and of geometric \'etale cohomology as represented by the complex (\ref{first32}).

  From the above constructions we derive several of the  properties mentioned in Theorem \ref{main1}.
 The quasi-isomorphisms (\ref{first5}) give that
 $$H^i_{\hk}(X_{\ovk,h})\simeq D_{\pst}(H^i(X_{\ovk,\eet},\Qp(r))),\quad H^i_{\hk}(X_h)\simeq D_{\st}(H^i(X_{\ovk,\eet},\Qp(r))),$$
where $D_{\pst}$ and $D_{\st}$ are the functors from \cite{FPR}.
 This combined with the diagram (\ref{first3}) immediately yields the spectral sequence ${}^{\synt}E_t$ since  the cohomology groups of the total complex of
$$ \left[
\begin{aligned}\xymatrix@=40pt{
  H^j_{\hk}(X_h)\ar[d]^{N}\ar[r]^-{(1-\phi_r,\iota_{\dr})} &
H^j_{\hk}(X_h)\oplus H^j_{\dr}(X_K)/F^r\ar[d]^{(N,0)}\\
  H^j_{\hk}(X_h)\ar[r]^-{1-\phi_{r-1}} & H^j_{\hk}(X_h)
 }\end{aligned}\right]
$$
 are equal to $H^*_{\st}(G_K,H^j(X_{\ovk,\eet},\Qp(r)))$. Moreover,
 the sequence of natural  maps of diagrams $(\ref{first3})\to (\ref{first31})\stackrel{\rho_{\synt}}{\to} (\ref{first32})$ yields a compatibility of the syntomic descent spectral sequence with the Hochschild-Serre spectral sequence in \'etale cohomology (via the period maps). We remark that, in the case of proper varieties with semistable reduction, this fact was announced in \cite{JB}.
  
  Looking again at the period map $\rho_{\synt}: (\ref{first31}){\to} (\ref{first32})$ we see that truncating all the complexes  at level $r$ will allow us to drop $+$ from the first diagram. Hence   we have
$$
\rho_{\synt}:\tau_{\leq r}\R\Gamma_{\synt}(X_{\ovk,h},r)\stackrel{\sim}{\to} \tau_{\leq r}\R\Gamma(X_{\ovk,\eet},\bq_p(r))
$$   To conclude  that we have
$$
\rho_{\synt}:\tau_{\leq r}\R\Gamma_{\synt}(X_h,r)\stackrel{\sim}{\to} \tau_{\leq r}\R\Gamma(X_{\eet},\bq_p(r))
$$ 
as well, 
we look at the map of spectral sequences ${}^{\synt}E\to {}^{\eet}E$ and observe that, in the stated ranges of the Hodge-Tate filtration we have $H^*_{\st}(G_K,\cdot)=H^*(G_K,\cdot)$ (a fact that follows, for example, from the work of Berger \cite{BER}).
\subsection{$p$-adic regulators}

  As an application of Theorem \ref{main1}, we look at the question of the image of Soul\'e's \'etale regulators 
$$
r^{\eet}_{r,i}: K_{2r-i-1}(X)_0\to H^1(G_K,H^i(X_{\ovk,\eet},\bq_p(r))),
$$
where $K_{2r-i-1}(X)_0:=\ker(c^{\eet}_{r,i+1}:K_{2r-i-1}(X)\to H^{i+1}(X_{\ovk,\eet},\bq_p(r)))$,
inside the Galois cohomology group. We prove that 
\begin{thmx}
The regulators $r^{\eet}_{r,i}$ factor through the group $H^1_{\st}(G_K,H^i(X_{\ovk,\eet},\bq_p(r)))$.
\end{thmx}
As we explain in the article, this fact is known to follow from the work of Scholl \cite{Sc} on "geometric" extensions associated to $K$-theory classes. In our approach, this is a simple consequence of good properties of  syntomic cohomology and the existence of the syntomic descent spectral sequence. Namely, as can be easily derived from the presentation (\ref{first3}), syntomic cohomology has projective space theorem and homotopy property \footnote{As explained in Appendix ~B, it follows that it is a Bloch-Ogus cohomology theory.} hence admits Chern classes from higher $K$-theory. It can be easily shown  that they are compatible with  the \'etale Chern classes via the syntomic period maps. The factorization we want in the above theorem follows then from the compatibility of the two descent spectral sequences.

\subsection{Notation and Conventions}
Let $V$ be a complete discrete valuation ring with fraction field
$K$  of characteristic 0, with perfect
residue field $k$ of characteristic $p$, and with maximal ideal ${\mathfrak m}_K$. Let $v$ be the valuation on $K$ normalized so that $v(p)=1$. Let $\ovk$ be an algebraic closure of $K$ and let $\overline{V}$ denote the integral closure of $V$ in $\ovk$. Let
$W(k)$ be the ring of Witt vectors of $k$ with
 fraction field $K_0$ and denote by $K_0^{\nr}$ the maximal unramified extension of $K_0$. Denote by $e_K$ the absolute ramification index of $K$, i.e., the degree of $K$ over $K_0$. Set $G_K=\Gal(\overline {K}/K)$ and let $I_K$ denote its inertia subgroup.  Let $\phi$ be the absolute
Frobenius on $W(\overline {k})$.
 We will denote by $V$,
$V^{\times}$, and $V^0$ the scheme $\Spec (V)$ with the trivial, canonical
(i.e., associated to the closed point), and $({\mathbf N}\to V, 1\mapsto 0)$
log-structure respectively. For a log-scheme $X$ over $\so_K$, $X_n$ will denote its reduction mod $p^n$, $X_0$ will denote its special fiber.

   Unless otherwise stated, we work in the category of integral quasi-coherent  log-schemes. In general, we will not distinguish between simplicial abelian groups and complexes of abelian groups. 

  Let $A$ be an abelian category with enough projective objects. In this paper $A$ will be the category of abelian groups or ${\mathbf Z}_p$-, ${\mathbf Z}/p^n$-, or $\bq_p$-modules. Unless otherwise stated, we work in the (stable) $\infty$-category $\sd(A)$, i.e., stable $\infty$-category whose objects are (left-bounded) chain complexes of projective objects of $A$. For a readable introduction to such categories the reader may consult \cite{MG}, \cite[1]{Lu2}. The $\infty$-derived category is essential to us for two reasons: first, it allows us to work simply with the Beilinson-Hyodo-Kato complexes; second, it supplies functorial homotopy limits.
 
  Many of our constructions will involve sheaves of objects from $\sd(A)$. The reader may consult the notes of Illusie \cite{IL} and Zheng \cite{Zhe} for a brief introduction to the subject and \cite{Lu1}, \cite{Lu2} for a thorough treatment.
  
  We will use a shorthand for certain homotopy limits. Namely,  if $f:C\to C'$ is a map  in the dg derived category of abelian groups, we set
$$[\xymatrix{C\ar[r]^f&C'}]:=\holim(C\to C^{\prime}\leftarrow 0).$$ 
We also set
$$
\left[\begin{aligned}
\xymatrix{C_1\ar[d]\ar[r]^f & C_2\ar[d]\\
C_3\ar[r]^g & C_4
}\end{aligned}\right]
:=[[C_1\stackrel{f}{\to} C_2]\to [C_3\stackrel{g}{\to} C_4]],
$$ 
 where the diagram in the brackets is a commutative diagram in the dg derived category.
\begin{acknowledgments}
Parts of this article were written during our visits to the Fields Institute in Toronto in spring 2012. The second author worked on this article also at BICMR, Beijing, 
and at the University of Padova. We would like to thank these institutions for  their support and hospitality.
 
This article was inspired by the work of Beilinson on $p$-adic comparison
theorems. We would like to thank him for discussions related to his work. Luc Illusie and Weizhe Zheng helped us understand  the $\infty$-category theory involved in Beilinson's work  and made their notes \cite{IL}, \cite{Zhe} available to us -- we would like to thank them for that. 
We have also profited from conversations with Laurent Berger, Amnon Besser,
Bharghav Bhatt, Bruno Chiarellotto, Pierre Colmez,  Fr\'ed\'eric D\'eglise, Luc Illusie, Tony Scholl, and Weizhe Zheng -
we are grateful for these exchanges. Moreover, we would like to thank Pierre
Colmez for reading and correcting parts of this article. Special thanks go to
Laurent Berger and Fr\'ed\'eric D\'eglise for writing the Appendices. 
 
\end{acknowledgments}
\section{Preliminaries}
In this section we will do some preparation. In the first part, we will collect some relevant facts from the literature concerning period rings, derived log de Rham complex, and $h$-topology. In the second part, we will prove vanishing results in Galois cohomology and a criterion comparing two spectral sequences that we will need to compare the syntomic descent spectral sequence with the \'etale Hochschild-Serre spectral sequence. 
\subsection{The rings of periods}
Let us recall briefly the definitions of the rings of periods $B_{\crr}$, $B_{\dr}$, $B_{\st}$
of Fontaine
 \cite{F1}. Let $A_{\crr}$ denote Fontaine's ring of crystalline periods \cite[2.2,2.3]{F1}. This is a $p$-adically complete ring such that $A_{\crr,n}:=A_{\crr}/p^n$ is a universal PD-thickening of $\overline{V}_n$ over $W_n(k)$. Let $J_{\crr,n}$ denote its PD-ideal, $A_{\crr,n}/J_{\crr,n}=\overline{V}_n$.
We have
$$
A_{\crr,n}=H^0_{\crr}(\Spec(\overline{V}_n)/W_n(k)),
\quad B^+_{\crr}:=A_{\crr}[1/p],\quad
B_{\crr}:=B^+_{\crr}[t^{-1}],
$$
where
$t$ is a certain element of $B^+_{\crr}$
(see \cite{F1} for a precise definition
of $t$). The ring $B^+_{\crr}$ is a topological $K_0$-module equipped
with a Frobenius
$\phi$ coming from the crystalline cohomology and a natural $G_K$-action.
We have that $\phi (t)=pt$ and that $G_K$ acts on $t$ via the cyclotomic
character.

   Let
$$
B^+_{\dr}:=\invlim_r({\bold Q}\otimes \invlim_n A_{\crr,n}/
J_{\crr,n}^{[r]}),\quad B_{\dr}:=B^+_{\dr}[t^{-1}].
$$
The ring $B^+_{\dr}$ has a discrete valuation given by the powers of $t$.
Its quotient
field is $B_{\dr}$. We set $F^nB_{\dr}=t^nB^+_{\dr}$. This defines a  descending filtration  on
$B_{\dr}$.

  The period ring $B_{\st}$ lies between $B_{\crr}$ and $B_{\dr}$ \cite[3.1]{F1}. To define it, choose a sequence of elements $s=(s_n)_{n\geq 0}$
  of $\overline{V}$ such that $s_0=p$ and $s_{n+1}^p=s_n$. Fontaine associates to it an element $u_{s }$
of $B^+_{\dr}$ that is transcendental over $B^+_{\crr}$. Let $B^+_{\st}$ denote the subring of $B_{\dr}$ generated by
$B^+_{\crr}$ and $u_{s}$.  It is a polynomial algebra
in one variable over $B^+_{\crr}$. The ring $B^+_{\st}$ does not depend on the choice of
 $s$ (because for another sequence $s^{\prime}=(s^{\prime}_n)_{n\geq 0}$ we have $u_s-u_{s^\prime}\in {\mathbf Z}_pt\subset B_{\crr}^+$). The action of $G_K$ on $B^+_{\dr}$ restricts well to $B^+_{\st}$.
The Frobenius $\phi $
extends to $B^+_{\st}$ by $\phi (u_{s})=pu_{s}$ and one defines the monodromy operator
$N:B^+_{\st}\to B^+_{\st} $
as the unique $B^+_{\crr}$-derivation such that $Nu_{s}=-1$.
We have $N\phi=p\phi N$  and the short exact sequence
\begin{equation}
\label{kwak11}
0\to B^+_{\crr}\to B^+_{\st}\stackrel{N}{\to}B^+_{\st}\to 0
\end{equation}
 Let $B_{\st }=B_{\crr}[u_{s }]$.
We denote by $\iota$
 the injection
$\iota:B^+_{\st}\hookrightarrow B^+_{\dr}$. The topology on $B_{\st}$ is the one induced by $B_{\crr}$ and the inductive topology; the map $\iota$ is continuous (though the topology on $B_{\st}$ is not the one induced from $B_{\dr}$).
\subsection{Derived log de Rham complex}
In this subsection we collect a few facts about  the relationship between  crystalline cohomology and de Rham cohomology.

  Let $S$ be a log-PD-scheme on which $p$ is nilpotent. For a log-scheme $Z$ over $S$, let $\LL\Omega^{\scriptscriptstyle\bullet}_{Z/S}$ denote the derived log de Rham complex (see \cite[3.1]{BE1} for a review). This is a commutative dg $\so_S$-algebra on $Z_{\eet}$ equipped with a Hodge filtration $F^m$. There is a natural morphism of filtered commutative dg $\so_S$-algebras \cite[1.9.1]{BE2}
\begin{equation}
\label{kappamap}
\kappa:\quad \LL\Omega^{\scriptscriptstyle\bullet}_{Z/S}\to \R u_{Z/S*}(\so_{Z/S}),
\end{equation}
where $u_{Z/S}: Z_{\crr}\to Z_{\eet}$ is the projection from the log-crystalline to the \'etale topos.
The following theorem was proved by Beilinson \cite[1.9.2]{BE2} by direct computations of both sides.
\begin{theorem}
\label{beilinson}
Suppose that $Z,S$ are fine and $f:Z\to S$ is an integral, locally complete intersection morphism.  Then (\ref{kappamap}) yields quasi-isomorphisms
$$\kappa_m:\quad \LL\Omega^{\scriptscriptstyle\bullet}_{Z/S}/F^m\stackrel{\sim}{\to} \R u_{Z/S*}(\so_{Z/S}/\sj^{[m]}_{Z/S}).
$$
\end{theorem}

 Recall \cite[Def. 7.20]{BH} that a log-scheme  is called G-log-syntomic if it is log-syntomic and the local log-smooth models can be chosen to be of Cartier type. The next theorem, finer than Theorem \ref{beilinson},  was proved by Bhatt \cite[Theorem 7.22]{BH} by looking at the conjugate filtration of the l.h.s.
\begin{theorem}
\label{bhatt}
Suppose that $f:Z\to S$ is G-log-syntomic. Then we have a quasi-isomorphism
$$\kappa:\quad \LL\Omega^{\scriptscriptstyle\bullet}_{Z/S}\stackrel{\sim}{\to} \R u_{Z/S*}(\so_{Z/S}).
$$
\end{theorem}
Combining the two theorems above, we get a filtered version:
\begin{corollary}
Suppose that $f:Z\to S$ is G-log-syntomic. Then we have a quasi-isomorphism
$$F^m\LL\Omega^{\scriptscriptstyle\bullet}_{Z/S}\stackrel{\sim}{\to} \R u_{Z/S*}(\sj^{[m]}_{Z/S}).
$$
\end{corollary}
\begin{proof}
 Consider the following commutative diagram with exact rows
$$
\begin{CD}
F^m\LL\Omega^{\scriptscriptstyle\bullet}_{Z/S} @>>>  \LL\Omega^{\scriptscriptstyle\bullet}_{Z/S} @>>>\LL\Omega^{\scriptscriptstyle\bullet}_{Z/S}/F^m \\
@VVV @VV\wr V @VV\wr V\\
\R u_{Z/S*}(\sj^{[m]}_{Z/S})@>>>  \R u_{Z/S*}(\so_{Z/S})
 @>>> \R u_{Z/S*}(\so_{Z/S}/\sj^{[m]}_{Z/S}).
\end{CD}
$$
and use the above theorems of Bhatt and Beilinson.
\end{proof}

 Let $X$ be a fine, proper,  log-smooth scheme over $V^{\times}$.  Set
\begin{align*}
\R\Gamma(X_{\eet},\LL\Omega^{\scriptscriptstyle\bullet,\wedge}_{X/W(k)})\wh{\otimes}{\mathbf Q}_p :=(\holim_n\R\Gamma(X_{\eet},\LL\Omega^{\scriptscriptstyle\bullet,\wedge}_{X_n/W_n(k)}))\otimes {\mathbf Q}
\end{align*}
and similarly for complexes over $V^{\times}$. 
Here the hat over derived log de Rham complex refers to the completion with respect to the Hodge filtration (in the sense of prosystems). For $r\geq 0$, consider the following sequence of maps 
\begin{equation}
\label{composition1}
\begin{aligned}
  \R\Gamma_{\dr}(X_K)/F^r & \stackrel{\sim}{\leftarrow}\R\Gamma(X,\LL\Omega^{\scriptscriptstyle\bullet}_{X/V^{\times}}/F^r)_\bq \stackrel{\sim}{\to} \R\Gamma(X_{\eet},\LL\Omega^{\scriptscriptstyle\bullet}_{X/V^{\times}}/F^r)\wh{\otimes}\Qp\\
 &  \stackrel{\sim}{\to}\R\Gamma_{\crr}(X,\so_{X/V^{\times}}/\sj^{[r]}_{X/V^{\times}})_{\mathbf Q}
  \leftarrow \R\Gamma_{\crr}(X,\so_{X/W(k)}/\sj^{[r]}_{X/W(k)})_{\mathbf Q} 
\end{aligned}
 \end{equation}
  The first quasi-isomorphism follows from the fact that the natural map $\LL\Omega_{X_K/K_0}^{\scriptscriptstyle\bullet}/F^r\stackrel{\sim}{\to} \Omega_{X_K/K_0}^{\scriptscriptstyle\bullet}/F^r$ is a quasi-isomorphism since $X_K$ is log-smooth over $K_0$. The second quasi-isomorphism follows from $X$ being proper and log-smooth over $V^{\times}$, the third one from Theorem \ref{beilinson}. Define the map 
$$
 \gamma_r^{-1}:\quad  \R\Gamma_{\crr}(X,\so_{X/W(k)}/\sj^{[r]}_{X/W(k)})_{\mathbf Q}\to \R\Gamma_{\dr}(X_K)/F^r 
$$ 
as the composition (\ref{composition1}). 

\begin{corollary}
\label{Langer}
Let $X$ be a fine, proper,  log-smooth scheme over $V^{\times}$. Let $r\geq 0$. There exists a canonical  quasi-isomorphism
$$\gamma_r:\quad \R\Gamma_{\dr}(X_K)/F^r \stackrel{\sim}{\to} \R\Gamma_{\crr}(X,\so_{X/W(k)}/\sj^{[r]}_{X/W(k)})_{\mathbf Q}
$$
\end{corollary}
\begin{proof}
 It suffices to show that the last map in the composition (\ref{composition1}) is also a quasi-isomorphism. By  Theorem \ref{beilinson}, this map  is quasi-isomorphic to the map
$$
  (\R\Gamma(X_{\eet},\LL\Omega^{\scriptscriptstyle\bullet,\wedge}_{X/W(k)})\wh{\otimes}\Qp)/F^r \to (\R\Gamma (X_{\eet},\LL\Omega^{\scriptscriptstyle\bullet,\wedge}_{X/V^{\times}})\wh{\otimes}\Qp)/F^r $$
  Hence it suffices to show that the natural map 
$$
  \gr^i_{F}\R\Gamma(X_{\eet},\LL\Omega^{\scriptscriptstyle\bullet,\wedge}_{X/W(k)})\wh{\otimes}\Qp \to \gr^i_F\R\Gamma (X_{\eet},\LL\Omega^{\scriptscriptstyle\bullet,\wedge}_{X/V^{\times}})\wh{\otimes}\Qp
$$
  is a quasi-isomorphism for all $i\geq 0$. 
  
  Fix $n\geq 1$ and $i\geq 0$ and recall \cite[1.2]{BE1} that we have a natural identification
$$\gr^i_F\LL\Omega^{\scriptscriptstyle\bullet}_{X_n/W_n(k)}\stackrel{\sim}{\to}L\Lambda^i_X(L_{X_n/W_n(k)})[-i],\quad \gr^i_F\LL\Omega^{\scriptscriptstyle\bullet}_{X_n/V^{\times}_n}\stackrel{\sim}{\to}L\Lambda^i_X(L_{X_n/V^{\times}_n})[-i],
$$ 
where $L_{Y/S}$ denotes the relative log cotangent complex \cite[3.1]{BE1} and $L\Lambda_X({\scriptscriptstyle\bullet })$ is the nonabelian left derived functor of the exterior power functor.
The distinguished triangle
$$
\so_X\otimes_VL_{V^{\times}_n/W_n(k)}\to L_{X_n/W_n(k)}\to L_{X_n/V^{\times}_n}
$$
yields a distinguished triangle
$$
L\Lambda_X^i(\so_X\otimes_VL_{V^{\times}_n/W_n(k)})[-i]\to \gr^i_F\LL\Omega^{\scriptscriptstyle\bullet}_{X_n/W_n(k)}\to  \gr^i_F\LL\Omega^{\scriptscriptstyle\bullet}_{X_n/V^{\times}_n}
$$
Hence we have a  distinguished triangle
$$ \holim_n \R\Gamma(X_{\eet},L\Lambda_X^i(\so_X\otimes_VL_{V^{\times}_n/W_n(k)}))\otimes{\mathbf Q}[-i] \to \gr^i_{F}\R\Gamma(X_{\eet},\LL\Omega^{\scriptscriptstyle\bullet,\wedge}_{X/W(k)})\wh{\otimes}\Qp \to \gr^i_F\R\Gamma (X_{\eet},\LL\Omega^{\scriptscriptstyle\bullet,\wedge}_{X/V^{\times}})\wh{\otimes}\Qp
$$

   It suffices to show that the term on the left is zero. But this will follow as soon as we show that the cohomology groups of $L_{V^{\times}_n/W_n(k)}$ are annihilated by $p^c$, where $c$ is a constant independent of $n$. To show this recall that $V$ is  a log complete intersection over $W(k)$. If $\pi$ is a generator of $V/W(k)$, $f(t)$ its minimal polynomial then (cf. \cite[6.9]{Ol}) $L_{V^{\times}/W(k)}$ is quasi-isomorphic to the cone of the multiplication by $f^{\prime}(\pi)$ map on $V$. Hence $L_{V^{\times}/W(k)}$ is acyclic in non-zero degrees, $H^0L_{V^{\times}/W(k)}=\Omega_{V^{\times}/W(k)}$ is a cyclic $V$-module and we have a short exact sequence
$$
0\to \Omega_{V/W(k)}\to \Omega_{V^{\times}/W(k)}\to V/{\mathfrak m}_K\to 0
$$
Since $\Omega_{V/W(k)}\simeq V/\sd_{K/K_0}$, where $\sd_{K/K_0}$ is the different, we get that $p^cH^0L_{V^{\times}/W(k)}=0$ for a constant $c$ independent of $n$. Since
$L_{V^{\times}/W(k)}\simeq L_{V^{\times}/W(k)}\otimes^{L}_VV_n$,
we are done.
\end{proof}
\begin{remark}
Versions of the above corollary appear in various degrees of generality in the proofs of the $p$-adic comparison theorems (cf. \cite[Lemma 4.5]{KM}, \cite[Lemma 2.7]{Ln}). They are proved using computations in crystalline cohomology. We find the above argument based on Beilinson comparison Theorem \ref{beilinson} particularly conceptual and pleasing.
\end{remark}
\subsection{$h$-topology} In this subsection we review terminology connected with $h$-topology from Beilinson papers \cite{BE1}, \cite{BE2}, \cite{BH}; we will use it freely.   Let $\mathcal{V}ar_K$ be the category of varieties (i.e., reduced and separated schemes of finite type) over a field $K$.
 An {\em arithmetic pair} over $K$ is an open embedding $j:U\hookrightarrow \overline{U}$ with dense image of a $K$-variety $U$ into a reduced proper flat $V$-scheme $\overline{U}$. A morphism $(U,\overline{U})\to (T,\overline{T})$ of pairs is a map $\overline{U}\to \overline{T}$ which sends $U$ to $T$. In the case that the pairs represent log-regular schemes this is the same as a map of log-schemes.
 For a pair $(U,\overline{U})$, we set $V_U:=\Gamma(\overline{U},\so_{\overline{U}})$ and $K_U:=\Gamma(\overline{U}_K,\so_{\overline{U}})$. $K_U$  is a product of several finite extensions of $K$ (labeled by the connected components of $\overline{U}$) and, if $\overline{U}$ is normal, $V_U$ is the product of the corresponding rings of integers.  We will denote by $\spp_K^{ar}$ the category of arithmetic pairs over $K$.
 A  {\em semistable pair} ({\em ss-pair}) over $K$ \cite[2.2]{BE1} is a pair of schemes $(U,\overline{U})$ over $(K,V)$ such that
 (i) $\overline{U}$ is regular and proper over $V$, (ii) $\overline{U}\setminus U$ is a divisor with normal crossings on $\overline{U}$, and (iii) the closed fiber $\overline{U}_0$ of $\overline{U}$ is reduced.
Closed fiber is taken over the closed points of $V_U$.  We will think of ss-pairs as log-schemes  equipped  with log-structure given by the divisor $\overline{U}\setminus U$. The closed fiber $\overline{U}_0$ has the induced log-structure. We will say that the log-scheme $(U,\overline{U})$ is {\em split} over $V_U$. We will denote by $\spp_K^{ss}$ the category of ss-pairs over $K$.  A semistable pair is called {\em strict} if the irreducible components of the closed fiber are regular.  We will often work with the larger category $\spp_K^{\log}$ of log-schemes $(U,\overline{U})\in \spp^{ar}_K$ log-smooth over $V_U^{\times}$. 

  A {\em semistable pair} ({\em ss-pair}) over $\ovk$ \cite[2.2]{BE1} is a pair of connected schemes $(T,\overline{T})$ over $(\ovk,\overline{V})$ such that there exists an ss-pair $(U,\overline{U})$ over $K$ and a $\ovk$-point $\alpha: K_U\to \ovk$ such that $(T,\overline{T})$ is isomorphic to the base change $(U_{\ovk},\overline{U}_{\overline{V}})$.
  We will denote by $\spp_{\ovk}^{ss}$ the category of ss-pairs over $\ovk$.
  
  A {\em geometric pair} over $K$ is a pair $(U,\overline{U})$ of varieties over $K$ such that $\overline{U}$ is proper and $U\subset \overline{U}$ is open and dense. We say that the pair $(U,\ovu)$ is a {\em nc-pair}  if $\overline{U}$ is regular and $\ovu\setminus U$ is a divisor with normal crossings in $\ovu$; it is {strict nc-pair} if the irreducible components of $U\setminus \ovu$ are regular.  A morphism of pairs $f:(U_1,\ovu_1)\to (U,\ovu)$ is a map $\ovu_1\to \ovu$ that sends $U_1$ to $U$. We denote the category of nc-pairs over $K$ by $\spp_K^{\nc}$.

  For a field $K$, the $h$-topology (cf. \cite{SV},\cite[2.3]{BE1}) on $\mathcal{V}ar_K$ is the coarsest topology finer than the Zariski and proper topologies.
  \footnote{The latter is generated by a pretopology whose coverings are proper surjective maps.} It is stronger than the \'etale and proper topologies. It is generated
  by the pretopology whose coverings are finite families of maps $\{Y_i\to X\}$ such that $Y:=\coprod Y_i\to X$ is a universal topological epimorphism
  (i.e., a subset of $X$ is Zariski open if and only if its preimage in $Y$ is open). We denote by $\mathcal{V}ar_{K,h},X_h$ the corresponding $h$-sites. For any of the categories  $\spp$ mentioned above let $\gamma: (U,\ovu)\to U$ denote the forgetful functor.  Beilinson proved \cite[2.5]{BE1} that the categories $\spp^{\nc}$, $(\spp_K^{ar},\gamma)$ and $(\spp_K^{ss},\gamma)$ form a base for $\mathcal{V}ar_{K,h}$. One can easily modify his argument to conclude the same about the categories $(\spp_K^{\log},\gamma)$.
\subsection{Galois cohomology} In this subsection we review the definition of (higher) semistable Selmer groups and prove that in stable ranges they are the same as Galois cohomology groups.
Our main references  are \cite{F2}, \cite{F3}, \cite{CF}, \cite{BK}, \cite{FPR}, \cite{JH}.
Recall \cite{F2}, \cite{F3} that a $p$-adic representation $V$ of $G_K$ (i.e., a finite dimensional continuous ${\bq}_p$-vector space representation) is called {\em semistable} (over $K$) if
$\dim_{K_0} (B_{\st}\otimes_{{\mathbf Q}_p} V )^{G_K} = \dim_{{\mathbf Q}_p} (V )$. It
 is called {\em potentially semistable} if there exists a
finite extension $K^{\prime}$ of
$K$ such that $V|G_{K^{\prime}}$
 is semistable over $K^{\prime}$. We denote by $\Rep_{\st}(G_K)$ and $\Rep_{\pst}(G_K)$
the categories of semistable and   potentially semistable representations of
$G_K$, respectively.

 As in \cite[4.2]{F2} a $\phi$-module over $K_0$ is a pair $(D,\phi)$, where $D$ is a finite dimensional $K_0$-vector space, $\phi=\phi_D$ is a $\phi$-semilinear automorphism of $D$; a $(\phi,N)$-module is a triple $(D,\phi,N)$, where $(D,\phi)$ is a $\phi$-module and $N=N_V$ is a $K_0$-linear endomorphism of $D$ such that $N\phi=p\phi N$ (hence $N$ is nilpotent). A filtered $(\phi,N)$-module  is a 
 $(D,\phi,N,F^{\scriptscriptstyle\bullet})$, where $(D,\phi,N)$ is a $(\phi, N)$-module and $F^{\scriptscriptstyle\bullet} $  is a decreasing finite filtration of $D_K$ by $K$-vector
spaces. There is a notion of a {\em (weakly) admissible} filtered $(\phi, N)$-module \cite{CF}. Denote by $MF^{\ad}_K(\phi,N)\subset MF_K(\phi,N)$ the categories of admissible filtered $(\phi, N)$-modules and  filtered $(\phi, N)$-modules, respectively. We know \cite{CF} that the pair of functors $D_{\st}(V)=(B_{\st}\otimes_{{\mathbf Q}_p}V)^{G_K}$, $V_{\st}(D)=(B_{\st}\otimes_{K_0}D)^{\phi=\id,N=0}\cap F^0(B_{\dr}\otimes_{K}D_K)$  defines an equivalence of categories $MF_K^{\ad}(\phi,N)\simeq \Rep_{\st}(G_K)$.

 For $D\in MF_K(\phi,N)$, set
$$
C_{\st}(D):=
\left[\begin{aligned}\xymatrix@C=36pt{D\ar[d]^N \ar[r]^-{(1-\phi,\can)} & D\oplus D_K/F^0\ar[d]^{(N,0)}\\
D\ar[r]^{1-p\phi} & D
}\end{aligned}\right]
$$
Here the brackets denote the total complex of the double complex inside the brackets.
Consider also the following complex
$$
C^+(D):=
\left[\begin{aligned}\xymatrix@C=40pt{D\otimes_{K_0}B^+_{\st}\ar[d]^N \ar[r]^-{(1-\phi,\can\otimes \iota)} & D\otimes_{K_0}B^+_{\st}\oplus (D_K\otimes_{K}B^+_{\dr})/F^0\ar[d]^{(N,0)}\\
D\otimes_{K_0}B^+_{\st}\ar[r]^{1-p\phi} & D\otimes_{K_0}B^+_{\st}}\end{aligned}\right]
$$
Define $C(D)$ by omitting the superscript $+$ in the above diagram. We have $C_{\st}(D)=C(D)^{G_K}$.
\begin{remark}
Recall \cite[1.19]{JH}, \cite[3.3]{FPR} that to every   $p$-adic representation $V$ of $G_K$ we can associate a complex
$$C_{\st}(V):\quad  D_{\st}(V)\veryverylomapr{(N,1-\phi,\iota)}D_{\st}(V)\oplus D_{\st}(V)\oplus t_V\veryverylomapr{(1-p\phi)-N} D_{\st}(V)\to 0\cdots
$$
where $t_V:=(V\otimes_{{\mathbf Q}_p} (B_{\dr}/B^+_{\dr}))^{G_K}$ \cite[I.2.2.1]{FPR}. The cohomology of this complex is called $H^*_{\st}(G_K,V)$. If $V$ is semistable then $C_{\st}(V)=C_{\st}(D_{\st}(V))$ hence $H^*(C_{\st}(D_{\st}(V)))=H^*_{\st}(G_K,V)$. If $V$ is potentially semistable the groups $H^*_{\st}(G_K,V)$ compute Yoneda extensions of ${\mathbf Q}_p$ by $V$ in the category of potentially  semistable representations \cite[I.3.3.8]{FPR}. In general \cite[I.3.3.7]{FPR}, $H^0_{\st}(G_K,V)\stackrel{\sim}{\to} H^0(G_K,V)$ and $H^1_{\st}(G_K,V)\hookrightarrow H^1(G_K,V)$ computes $\st$-extensions\footnote{Extension $0\to V_1\to V_2\to V_3\to 0$ is called $\st$ if the sequence $0\to D_{\st}(V_1)\to D_{\st}(V_2)\to D_{\st}(V_3)\to 0$ is exact.} of ${\mathbf Q}_p$ by $V$.
\end{remark}
\begin{remark}
\label{basics}
 Let $D\in MF_K(\phi, N)$. Note that
\begin{enumerate}
\item $H^0(C(D))=V_{\st}(D)$;
\item for $i\geq 2$, $H^i(C^+(D))=H^i(C(D))=0$ (because $N$ is surjective on $B^+_{\st}$ and $B_{\st}$);
\item if $F^1D_K=0$ then $F^0(D_K\otimes _KB^+_{\dr})=F^0(D_K\otimes _KB_{\dr})$ (in particular, the map of complexes $C^+(D)\to C(D)$ is an injection);
\item if $D=D_{\st}(V)$ is admissible  then we have quasi-isomorphisms
$$
C(D)\stackrel{\sim}{\leftarrow}V\otimes_{{\mathbf Q}_p}[B_{\crr}\verylomapr{(1-\phi,\can)}B_{\crr}\oplus B_{\dr}/F^0]\stackrel{\sim}{\leftarrow}V\otimes_{{\mathbf Q}_p}(B_{\crr}^{\phi=1}\cap F^0)=V
$$
and the  map of complexes $C_{\st}(D)\to C(D)$ represents  the canonical map $H^i_{\st}(G_K,V)\to H^i(G_K,V)$.
\end{enumerate}
\end{remark}
\begin{lemma}(\cite[Theorem II.5.3]{F1})
If $X\subset B_{\crr}\cap B_{\dr}^+$ and $\phi(X)\subset X$ then $\phi^2(X)\subset B_{\crr}^+$.
\end{lemma}
\begin{proposition}If $D\in MF_K(\phi, N)$ and $F^1D_K=0$ then $H^0(C(D)/C^+(D))=0$.
\end{proposition}
\begin{proof}
We will argue by induction on $m$ such that $N^m=0$.
Assume first that $m=1$ (hence $N=0$). We have
\begin{align*}
C(D)/C^+(D) & =\left[\begin{aligned}\xymatrix@C=40pt{D\otimes_{K_0}(B_{\st}/B^+_{\st})\ar[r]^-{(1-\phi,\can\otimes\iota)}\ar[d]^{1\otimes N} & D\otimes_{K_0}(B_{\st}/B^+_{\st})\oplus D_K\otimes_{K}(B_{\dr}/B^+_{\dr})\ar[d]^{(1\otimes N,0)}\\
D\otimes_{K_0}(B_{\st}/B^+_{\st})\ar[r]^{1-p\phi} & D\otimes_{K_0}(B_{\st}/B^+_{\st})
}\end{aligned}\right]\\
 & \stackrel{\sim}{\leftarrow}
[D\otimes_{K_0}(B_{\crr}/B^+_{\crr})\verylomapr{(1-\phi,\can)} D\otimes_{K_0}(B_{\crr}/B^+_{\crr})\oplus D_K\otimes_{K}(B_{\dr}/B^+_{\dr})]
\end{align*}
Write $D=\oplus_{i=1}^{r}K_0d_i$ and, for $1\leq i\leq r$, consider the following maps
$$
p_i:H^0(C(D)/C^+(D))=(D\otimes_{K_0}((B_{\crr}\cap B_{\dr}^+)/B^+_{\crr}))^{\phi=1}\subset \oplus_{i=1}^rd_i\otimes ((B_{\crr}\cap B_{\dr}^+)/B^+_{\crr})\stackrel{\pr_i}{\to}(B_{\crr}\cap B_{\dr}^+)/B^+_{\crr}
$$
Let $Y_a$, $a\in H^0(C(D)/C^+(D))$, denote the $K_0$-subspace of $(B_{\crr}\cap B_{\dr}^+)/B^+_{\crr}$ spanned by $p_1(a),\ldots,p_r(a)$.
We have $(p_1(a),\ldots ,p_r(a))^T=M\phi(p_1(a),\ldots ,p_r(a))^T$, for $M\in GL_r(K_0)$. Hence $\phi(Y_a)\subset Y_a$. Let $X_a\subset B_{\crr}\cap B^+_{\dr}$ be the inverse image of $Y_a$ under the projection $B_{\crr}\cap B^+_{\dr}\to (B_{\crr}\cap B^+_{\dr})/B_{\crr}^+$ (naturally $B^+_{\crr}\subset X_a$). Then $\phi(X_a)\subset X_a + B^+_{\crr}=X_a.$ By the above lemma $\phi^2(X_a)\subset B^+_{\crr}$. Hence $\phi^2(Y_a)=0$ and (applying $M^{-2}$) $Y_a=0$. This implies that $a=0$ and $H^0(C(D)/C^+(D))=0$, as wanted.

  For general $m>0$, consider the filtration $D_1\subset D$, where $D_1:=\ker(N)$ with induced structures. Set $D_2:= D/D_1$ with induced structures. Then $D_1,D_2\in MF_{K}(\phi,N)$; $N^i$ is trivial on $D_1$ for $i=1$ and on $D_2$ for $i=m-1$. Clearly $F^1D_{1,K}=F^1D_{2,K}=0$. Hence, by Remark \ref{basics}.3, we have a short exact sequence
$$
0\to C(D_1)/C^+(D_1)\to C(D)/C^+(D)\to C(D_2)/C^+(D_2)\to 0
$$
By the inductive assumption $H^0(C(D_1)/C^+(D_1))=H^0(C(D_2)/C^+(D_2))=0$. Hence $H^0(C(D)/C^+(D))=0$, as wanted.
\end{proof}
\begin{corollary}
If $D\in MF_K(\phi, N)$ and $F^1D_K=0$ then $H^0(C^+(D))=H^0(C(D))=V_{\st}(D)$ ($\subset D\otimes_{K_0}B^+_{\st}$) and $H^1(C^+(D))\hookrightarrow H^1(C(D))$.
\end{corollary}
\begin{corollary}
\label{MF1}
If $D\in MF^{\ad}_K(\phi, N)$ and $F^1D_K=0$ then
$$H^i(C^+(D))=H^i(C(D))=
\begin{cases}
V_{\st}(D)  &i=0\\
0  & i\neq 0
\end{cases}
$$
(i.e., $C^+(D)\stackrel{\sim}{\to} C(D)$).
\end{corollary}

 A filtered $(\phi,N,G_K)$-module is a  tuple $(D,\phi,N,\rho, F^{\scriptscriptstyle\bullet})$, where
\begin{enumerate}
\item $D$ is a  finite dimensional $K_0^{\nr}$-vector
space;
\item  $\phi : D \to D$ is a Frobenius  map;
\item $N : D \to D$ is  a $K_0^{\nr}$-linear monodromy map such that $N\phi = p\phi N$;
\item  $\rho$ is a $K_0^{\nr}$-semilinear $G_K$-action on $D$  (hence $\rho|I_K$ is linear) that is smooth, i.e., all vectors have open stabilizers, 
and that commutes with $\phi$ and $N$;
 \item  $F^{\scriptscriptstyle\bullet}$ is a decreasing finite filtration  of $D_K:=(D\otimes _{K_0^{\nr}}\ovk)^{G_K}$ by $K$-vector
spaces.

\end{enumerate}
Morphisms between filtered $(\phi,N, G_K)$-modules are $K_0^{\nr}$-linear maps preserving all
structures. There is a notion of a {\em (weakly) admissible} filtered $(\phi, N,G_K)$-module \cite{CF}, \cite{F3}. Denote by $MF_K^{\ad}(\phi,N,G_K)\subset MF_K^{}(\phi,N,G_K)$ the categories of admissible filtered $(\phi, N,G_K)$-modules and  filtered $(\phi, N,G_K)$-modules, respectively. We know \cite{CF} that the pair of functors $D_{\pst}(V)=\injlim_H(B_{\st}\otimes_{{\mathbf Q}_p}V)^{H}$, $H\subset G_K$ - an open subgroup, $V_{\pst}(D)=(B_{\st}\otimes_{K_0^{\nr}}D)^{\phi=\id,N=0}\cap F^0(B_{\dr}\otimes_{K}D_K)$  define an equivalence of categories $MF_K^{\ad}(\phi,N,G_K)\simeq \Rep_{\pst}(G_K)$.

  For $D\in MF_K(\phi,N,G_K)$, set \footnote{We hope that the notation below will not lead to confusion with the semistable case in general but if in doubt we will add the data of the field $K$ in the latter case.}
$$
C_{\pst}(D):=
\left[\begin{aligned}\xymatrix@C=36pt{D_{\st}\ar[d]^N \ar[r]^-{(1-\phi,\can)} & D_{\st}\oplus D_{K}/F^0\ar[d]^{(N,0)}\\
D_{\st}\ar[r]^{1-p\phi} & D_{\st}
}\end{aligned}\right]
$$
Here $D_{\st}:=D^{G_K}$.
Consider also the following complex (we set $D_{\ovk}:=D\otimes_{K_0^{\nr}}\ovk$)
$$
C^+(D):=
\left[\begin{aligned}\xymatrix@C=40pt{D\otimes_{K_0^{\nr}}B^+_{\st}\ar[d]^N \ar[r]^-{(1-\phi,\can\otimes \iota)} & (D\otimes_{K_0^{\nr}}B^+_{\st})\oplus (D_{\ovk}\otimes_{\ovk}B^+_{\dr})/F^0\ar[d]^{(N,0)}\\
D\otimes_{K_0^{\nr}}B^+_{\st}\ar[r]^{1-p\phi} & D\otimes_{K_0^{\nr}}B^+_{\st}}\end{aligned}\right]
$$
Define $C(D)$ by omitting the superscript $+$ in the above diagram. We have $C_{\pst}(D)=C(D)^{G_K}$.
\begin{remark}
\label{pst=st}
If $V$ is potentially semistable then $C_{\st}(V)=C_{\pst}(D_{\pst}(V))$ hence $H^*(C_{\pst}(D_{\pst}(V)))=H^*_{\st}(G_K,V)$.
\end{remark}
\begin{remark}
\label{resolution2}
If $D=D_{\pst}(V)$ is admissible  then we have quasi-isomorphisms
$$
C(D)\stackrel{\sim}{\leftarrow}V\otimes_{{\mathbf Q}_p}[B_{\crr}\verylomapr{(1-\phi,\can)}B_{\crr}\oplus B_{\dr}/F^0]\stackrel{\sim}{\leftarrow}V\otimes_{{\mathbf Q}_p}(B_{\crr}^{\phi=1}\cap F^0)=V
$$
and the  map of complexes $C_{\pst}(D)\to C(D)$ represents  the canonical map $H^i_{\st}(G_K,V)\to H^i(G_K,V)$.
\end{remark}
\begin{remark}
\label{BKexp}Let $D=D_{\pst}(V)$ be admissible. 
The  Bloch-Kato exponential $$ (Z^1 C(D))^{G_K} \to H^1(G_K,V)$$
is given by the coboundary map arising from the exact sequence
$$ 0 \to V \to C^0(D) \to Z^1 C(D) \to 0. $$ Its restriction to the de Rham
part of $Z^1 C(D)$ is the  Bloch-Kato exponential
$$\exp_{\bk}: D_K/F^0\to H^1(G_K,V).$$ It is also obtained by applying
$Rf$, where $f(-) = (-)^{G_K}$, to the coboundary map $\partial : Z^1 C(D)
\to V[1]$ arising from the above exact sequence (see the proof of Theorem
\ref {stHS} for an appropriate formalism of continuous cohomology).
Note that the composition of the canonical maps
$$ Z^1 C(D) \to (\sigma_{\geq 1} C(D))[1] \to C(D) [1]
\stackrel{\sim}{\leftarrow} V[1] $$ is not equal to $\partial$, but to
$- \partial$, by \eqref{sign}.
\end{remark}
\begin{corollary}
\label{resolution3}
If $D\in MF^{\ad}_K(\phi, N, G_K)$ and $F^1D_K=0$ then
$$H^i(C^+(D))\stackrel{\sim}{\to}H^i(C(D))=
\begin{cases}
V_{\pst}(D)  &i=0\\
0  & i\neq 0
\end{cases}
$$
(i.e., $C^+(D)\stackrel{\sim}{\to} C(D)$).
\end{corollary}
\begin{proof}
By Remark \ref{resolution2} we have $C(D)\simeq V_{\pst}(D)[0]$. To prove the isomorphism $H^i(C^+(D))\stackrel{\sim}{\to}H^i(C(D))$, $i\geq 0$,  take a finite Galois extension $K^{\prime}/K$  such that $D$ becomes semistable over $K^{\prime}$, i.e.,  $I_{K^{\prime}}$ acts trivially on $D$. We have   $(D^{\prime},\phi,N)\in MF_{K^{\prime}}^{\ad}(\phi,N)$, where $D^{\prime}:=D^{G_{K^{\prime}}}$ and  (compatibly)   $D\simeq D^{\prime}\otimes_{K_0^{\prime}}K_0^{\nr}$, $F^{\scriptscriptstyle\bullet}D^{\prime}_{K^{\prime}}\simeq F^{\scriptscriptstyle\bullet}D_K\otimes_KK^{\prime}$. It easily follows that
$C^+(D)=C^+(K^{\prime},D^{\prime})$ and
$C(D) = C(K^{\prime}, D^{\prime})$. Since $F^1D^{\prime}_{K^{\prime}}=0$, our corollary is now a consequence of
 Corollary \ref{MF1}
\end{proof}
\begin{proposition}
\label{resolution33}
If $D\in MF^{\ad}_K(\phi, N, G_K)$ and $F^1D_K=0$ then, for $i\geq 0$, the natural map
$$H^i_{\st}(G_K,V_{\pst}(D))\stackrel{\sim}{\to} H^i(G_K,V_{\pst}(D))
$$
is an isomorphism. 
\end{proposition}
\begin{proof}
Both sides satisfy Galois descent for finite Galois extensions. We can
assume, therefore, that $D = D_{\st}(V)$ for a semistable representation
$V$ of $G_K$. For $i=0$ we have (even without assuming $F^1D_K=0$)
$H^0(C_{\st}(D)) = H^0(C(D)^{G_K}) = H^0(C(D))^{G_K} = V^{G_K}$.
For $i=1$ the statement is proved in \cite[Thm. 6.2, Lemme 6.5]{BER}. For $i=2$
it follows from the assumption $F^1D_K=0$ (by weak admissibility of $D$) that
there is a $W(k)$-lattice $M \subset D$ such that $\phi^{-1}(M) \subset
p^2 M$, which implies that $1 - p\phi = -p\phi(1 - p^{-1}\phi^{-1}) : D \to D$
is surjective, hence $H^2(C_{\st}(D)) = 0$ (cf. the proof of
\cite[Lemme 6.7]{BER}). The proof of the fact that $H^2(G_K, V) = 0$ if
$F^1D_K=0$ was kindly communicated to us by L.~Berger; it is reproduced
in Appendix~A (cf. Theorem \ref{main}). For $i > 2$ both terms vanish.
\end{proof}
\subsection{Comparison of spectral sequences}
\label{Jan}
The purpose of this subsection is to prove a derived category theorem
(Theorem \ref{speccomp}) that will be used later to relate the syntomic
descent spectral sequence with  the \'etale Hochschild-Serre spectral
sequence (cf. Theorem \ref{stHS}). 
Let $D$ be a triangulated category and $H : D \to A$ a cohomological
functor to an abelian category $A$. A finite collection of adjacent
exact triangles (a ``Postnikov system" in the language of \cite
[IV.2, Ex. 2]{GM})

\begin{equation}
\label{postnikov}
\xymatrix@C=10pt@R=15pt{
& Y^0 \ar[dr] && Y^1 \ar[dr] &&&& Y^n \ar[dr] &\\
X = X^0 \ar[ur] && X^1 \ar [ur] \ar[ll]^(.45){[1]} && X^2 \ar[ll]^{[1]}
&& X^n \ar@{.}[ll] \ar[ru] && X^{n+1} = 0 \ar[ll]^(.55){[1]}\\}
\end{equation}
gives rise to an exact couple

$$
D_1^{p,q} = H^q(X^p) = H(X^p[q]),\qquad E_1^{p,q} = H^q(Y^p)
\Longrightarrow H^{p+q}(X).
$$
The induced filtration on the abutment is given by

$$
F^p H^{p+q}(X) = \im \left(D_1^{p,q} = H^q(X^p) \to H^{p+q}(X)\right).
$$

\begin{remark}
In the special case when $A$ is the heart of a non-degenerate $t$-structure
$(D^{\leq n}, D^{\geq n})$ on $D$ and $H = \tau_{\leq 0} \tau_{\geq 0}$,
 the following conditions are equivalent:

\begin{enumerate}
\label{ass}
\item $E_2^{p,q} = 0$ for $p\not= 0$;
\item $D_2^{p,q} = 0$ for all $p, q$;
\item $D_r^{p,q} = 0$ for all $p, q$ and $r > 1$;
\item the sequence $0 \to H^q(X^p) \to H^q(Y^p) \to H^q(X^{p+1}) \to 0$
is exact for all $p, q$;
\item the sequence $0 \to H^q(X) \to H^q(Y^0) \to H^q(Y^1) \to \cdots$
is exact for all $q$;
\item the canonical map $H^q(X) \to E_1^{{\scriptscriptstyle\bullet},q}$
is a quasi-isomorphism, for all $q$;
\item the triangle $\tau_{\leq q} X^p \to \tau_{\leq q} Y^p \to
\tau_{\leq q} X^{p+1}$ is exact for all $p, q$.
\end{enumerate}
\end{remark}

From now on until the end of \ref{Jan}  assume that $D = D(A)$ is the derived
category of $A$ with the standard $t$-structure and that $X^i, Y^i \in
D^+(A)$, for all $i$. Furthermore, assume that $f : A \to A^\prime$
is a left exact functor to an abelian category $A^\prime$ and that
$A$ admits a class of $f$-adapted objects (hence the derived functor
$\R f : D^+(A) \to D^+(A^\prime)$ exists).

Applying $\R f$ to (\ref{postnikov}) we obtain another Postnikov system,
this time in $D^+(A^\prime)$. The corresponding exact couple

\begin{equation}
\label{firstcouple}
{}^I D_1^{p,q} = (\R^q f)(X^p),\qquad {}^I E_1^{p,q} = (\R^q f)(Y^p)
\Longrightarrow (\R^{p+q}f)(X)
\end{equation}
induces filtration

$$
{}^I F^p (\R^{p+q}f)(X) = \im \left({}^I D_1^{p,q} = (\R^q f)(X^p) \to
(\R^{p+q}f)(X) \right).
$$
Our goal is to compare (\ref{firstcouple}), under the equivalent assumptions
(\ref{ass}), to the hypercohomology exact couple

\begin{equation}
\label{secondcouple}
{}^{II} D_2^{p,q} = (\R^{p+q}f)(\tau_{\leq q-1} X),\qquad {}^{II} E_2^{p,q} =
(\R^p f)(H^q(X)) \Longrightarrow (\R^{p+q}f)(X)
\end{equation}
for which

$$
{}^{II} F^p (\R^{p+q}f)(X) = \im \left({}^{II} D_2^{p-1,q+1} =
(\R^{p+q} f)(\tau_{\leq q} X) \to (\R^{p+q}f)(X) \right).
$$

\begin{theorem}
\label{speccomp}
Under the assumptions (\ref{ass}) there is a natural morphism of exact couples
$(u, v) : ({}^I D_2, {}^I E_2) \to ({}^{II} D_2, {}^{II} E_2)$. Consequently,
we have ${}^I F^p \subseteq {}^{II} F^p$ for all $p$ and there is a natural
morphism of spectral sequences ${}^I E_r^{*,*} \to {}^{II} E_r^{*,*}$
($r > 1$) compatible with the identity map on the common abutment.
\end{theorem}
\begin{proof}
{\bf Step 1:} we begin by constructing a natural map $u : {}^I D_2 \to
{}^{II} D_2$.

For each $p > 0$ there is a commutative diagram in $D^+(A^\prime)$

$$
\xymatrix{
(\R^{p+q}f)((\tau_{\leq q} Y^{p-1})[-p]) \ar [r] \ar[d]^\wr &
(\R^{p+q}f)((\tau_{\leq q} X^p)[-p]) \ar[r] \ar[d]^\wr &
(\R^{p+q}f)(\tau_{\leq q} X) \ar[d]^{\alpha_{II}} \\
{}^I E_1^{p-1,q} = (\R^{p+q}f)(Y^{p-1}[-p]) \ar[r]^(.54){k_1} &
{}^I D_1^{p,q} = (\R^{p+q}f)(X^p[-p]) \ar [ru]^{u^\prime}
\ar[r]^(.62){\alpha_I} & (\R^{p+q}f)(X)\\}
$$
whose both rows are complexes. This defines a map $u^\prime : {}^I D_1^{p,q}
\to {}^{II} D_2^{p-1,q+1}$ such that $u^\prime k_1 = 0$ and $\alpha_{II}
u ^\prime = \alpha_I$ (hence ${}^I F^p  = \im(\alpha_I) \subseteq
\im(\alpha_{II}) = {}^{II} F^p$).
By construction, the diagram (with exact top row)

$$
\xymatrix{
{}^I E_1^{p,q-1} \ar[r]^{k_1} \ar[rd]^0 & {}^I D_1^{p+1,q-1} \ar[r]^(.55){i_1}
\ar[d]^{u^\prime} & {}^I D_1^{p,q} \ar[d]^{u^\prime}\\
& {}^{II} D_2^{p,q} \ar[r]^(.45){i_2} & {}^{II} D_2^{p-1,q+1}\\}
$$
is commutative for each $p\geq 0$, which implies that the map

$$
u = u^\prime i_1^{-1} : {}^I D_2^{p,q} = i_1({}^I D_1^{p+1,q-1})
\to {}^{II} D_2^{p,q}
$$
is well-defined and satisfies $u i_2 = i_2 u$.

\smallskip
\noindent
{\bf Step 2:} for all $q$, the canonical quasi-isomorphism $H^q(X) \to
E_1^{{\scriptscriptstyle\bullet},q}$ induces natural morphisms

\begin{align*}
v^\prime : {}^I E_2^{p,q} &= H^p(i \mapsto (\R^q f)(Y^i)) \to
H^p(i \mapsto f(H^q(Y^i))) \to (\R^p f)(i \mapsto H^q(Y^i))\\
&= (\R^p f)(E_1^{{\scriptscriptstyle\bullet},q}) \stackrel\sim\longleftarrow
(\R^p f)(H^q(X)) = {}^{II} E_2^{p,q};\\
\end{align*}
set $v = (-1)^p v^\prime : {}^I E_2^{p,q} \to {}^{II} E_2^{p,q}$.

It remains to show that $u$ and $v$ are compatible with the maps

$$
{}^? D_2^{p-1,q+1} \stackrel {j_2} \longrightarrow {}^? E_2^{p,q}
\stackrel {k_2} \longrightarrow {}^? D_2^{p+1,q} \qquad\qquad (? = I, II).
$$
{\bf Step 3:} for any complex $M^{\scriptscriptstyle\bullet}$ over $A$
denote by $Z^i(M^{\scriptscriptstyle\bullet}) = \kker (\delta^i : M^i \to
M^{i+1})$ the subobject of cycles in degree $i$.

If $M^{\scriptscriptstyle\bullet}$ is a resolution of an object $M$ of $A$,
then each exact sequence

\begin{equation}
\label{cycle}
0 \longrightarrow Z^p(M^{\scriptscriptstyle\bullet})
\longrightarrow M^p \stackrel{\delta^p}\longrightarrow
Z^{p+1}(M^{\scriptscriptstyle\bullet}) \longrightarrow 0
\qquad\qquad (p\geq 0)
\end{equation}
can be completed to an exact sequence of resolutions

$$
\xymatrix@C=15pt@R=15pt{
0 \ar[r] & Z^p(M^{\scriptscriptstyle\bullet}) \ar[r] \ar[d]^{\can} &
M^p \ar[r] \ar[d]^{\can} & Z^{p+1}(M^{\scriptscriptstyle\bullet}) \ar[r]
\ar[d]^{-\can} & 0\\
0 \ar[r] & (\sigma_{\geq p}(M^{\scriptscriptstyle\bullet}))[p] \ar[r] &
(\sigma_{\geq p} {\rm Cone}(M^{\scriptscriptstyle\bullet} \stackrel
{\rm id} \to M^{\scriptscriptstyle\bullet}))[p] \ar[r] &
(\sigma_{\geq p+1}(M^{\scriptscriptstyle\bullet}))[p+1] \ar[r] & 0.\\}
$$
By induction, we obtain that the following diagram, whose top arrow
is the composition of the natural maps $Z^i \to Z^{i-1}[1]$ induced by
(\ref{cycle}), commutes in $D^+(A)$.

\begin{equation}
\label{sign}
\xymatrix@C=10pt@R=15pt{
Z^p(M^{\scriptscriptstyle\bullet}) \ar[r] \ar[d]^{\can} &
Z^0(M^{\scriptscriptstyle\bullet})[p] = M[p] \ar[d]^{(-1)^p\can}\\
(\sigma_{\geq p}(M^{\scriptscriptstyle\bullet}))[p] \ar[r]^(.58){\can} &
M^{\scriptscriptstyle\bullet}[p]\\}
\end{equation}
We are going to apply this statement to $M = H^q(X)$ and
$M^{\scriptscriptstyle\bullet} = E_1^{\scriptscriptstyle\bullet,q}$, when
$Z^p(M^{\scriptscriptstyle\bullet}) = D_1^{p,q} = H^q(X^p)$ and
$Z^0(M^{\scriptscriptstyle\bullet}) = H^q(X)$.

\smallskip
\noindent
{\bf Step 4:} we are going to investigate ${}^I E_2^{p,q}$.

Complete the morphism $Y^p \to Y^{p+1}$ to an exact triangle
$U^p \to Y^p \to Y^{p+1}$ in $D^+(A)$ and fix a lift $X^p \to U^p$
of the morphism $X^p \to Y^p$.

There are canonical epimorphisms

\begin{equation}
\label{epi}
(\R^q f)(U^p) \twoheadrightarrow
\kker((\R^q f)(Y^p) \stackrel {j_1 k_1} \longrightarrow (\R^q f)(Y^{p+1})) =
Z^p({}^I E_1^{\scriptscriptstyle\bullet,q}) \twoheadrightarrow
{}^I E_2^{p,q}
\end{equation}
and the map

$$
k_2 : {}^I E_2^{p,q} \to {}^I D_2^{p+1,q} = \kker({}^I D_1^{p+1,q}
\stackrel {j_1} \longrightarrow {}^I E_1^{p+1,q})
$$
is induced by the restriction of $k_1 : {}^I E_1^{p,q} \to {}^I D_1^{p+1,q}$
to $Z^p({}^I E_1^{\scriptscriptstyle\bullet,q})$.

The following octahedron (in which we have drawn only the four exact faces)

$$
\xymatrix@C=15pt@R=15pt{
X^{p+2} \ar[rd]_{[1]} && Y^{p+1} \ar[ll]\\
& X^{p+1} \ar[ru] \ar[dl] &\\
X^p[1] \ar[rr]^{[1]} && Y^p \ar[lu]\\}
\qquad\qquad
\xymatrix@C=15pt@R=15pt{
X^{p+2} \ar[dd]_{[1]} && Y^{p+1} \ar[dl]\\
& U^p[1] \ar[lu] \ar[dr]^{[1]} &\\
X^p[1] \ar[ru] && Y^p \ar[uu]\\}
$$
shows that the triangle $X^p \to U^p \to X^{p+2}[-1]$ is exact and the
diagrams

$$
\xymatrix@C=15pt@R=15pt{
U^p[1] \ar[r] \ar[d] & Y^p[1] \ar[d]\\
X^{p+2} \ar[r] & X^{p+1}[1]\\}
\qquad\qquad
\xymatrix@C=15pt@R=15pt{
(\R^q f)(U^p) \ar[r] \ar[d] & Z^p({}^I E_1^{\scriptscriptstyle\bullet,q})
\ar[d]_{k_1}\\
(\R^q f)(X^{p+2}[-1]) = {}^I D_2^{p+2,q-1} \ar[r]^(.72){i_1} &
{}^I D_2^{p+1,q}\\}
$$
commute. The previous discussion implies that the composite map

$$
(\R^q f)(U^p) \twoheadrightarrow Z^p({}^I E_1^{\scriptscriptstyle\bullet,q})
\twoheadrightarrow {}^I E_2^{p,q} \stackrel {k_2} \longrightarrow
{}^I D_2^{p+1,q} \stackrel u \to {}^{II} D_2^{p+1,q} =
(\R^q f)((\tau_{\leq q-1} X)[p+1])
$$
is obtained by applying $\R^q f$ to

\begin{equation}
\label{comp1}
\tau_{\leq q}\, U^p \to \tau_{\leq q}(X^{p+2}[-1]) =
(\tau_{\leq q-1} X^{p+2})[-1] \to (\tau_{\leq q-1}\, X)[p+1].
\end{equation}
{\bf Step 5:} all boundary maps $H^q(X^{p+2}[-1]) \to H^q(X^p)$ vanish by
(\ref{ass}), which means that the following triangles are exact.

$$
\tau_{\leq q}\, X^p \to \tau_{\leq q}\, U^p \to \tau_{\leq q}(X^{p+2}[-1]) =
(\tau_{\leq q-1}\, X^{p+2})[-1]
$$
The commutative diagram

$$
\xymatrix@C=15pt@R=15pt{
\tau_{\leq q}\, U^p \ar[r] & H^q(U^p)[-q] \ar[r] &
\kker\left(H^q(Y^p) \to H^q(Y^{p+1})\right)[-q]\\
\tau_{\leq q}\, X^p \ar[r] \ar[u] & H^q(X^p)[-q] \ar[u] \ar@2{-}[ru] &\\}
$$
gives rise to an octahedron

$$
\xymatrix@C=10pt@R=15pt{
V^p \ar[rd]_{[1]} && H^q(X^p)[-q] \ar[ll]\\
& \tau_{\leq q}\, U^p \ar[ru] \ar[dl] &\\
\tau_{\leq q}(X^{p+2}[-1]) \ar[rr]^{[1]} && \tau_{\leq q}\,X^p \ar[lu]\\}
\qquad
\xymatrix@C=10pt@R=15pt{
V^p \ar[dd]_{[1]} && H^q(X^p)[-q] \ar[dl]\\
& (\tau_{\leq q-1}\, X^p)[1] \ar[lu] \ar[dr]^{[1]} &\\
X^p[1] \ar[ru] && Y^p \ar[uu]\\}
$$
In particular, the following diagram commutes.

\begin{equation}
\label{comp2}
\xymatrix@C=15pt@R=15pt{
\tau_{\leq q}\, U^p \ar[r] \ar[d] & H^q(X^p)[-q] \ar[d]\\
\tau_{\leq q}(X^{p+2}[-1]) \ar[r] & (\tau_{\leq q-1}\, X^p)[1]\\}
\end{equation}
{\bf Step 6:} the diagram (\ref{sign}) implies that the composition of
$v : {}^I E_2^{p,q} \to {}^{II} E_2^{p,q}$ with the second epimorphism
in (\ref{epi}) is equal to the composite map

\begin{align*}
&Z^p({}^I E_1^{\scriptscriptstyle\bullet,q}) =
\kker\left((\R^q f)(\tau_{\leq q}\,Y^p) \to
(\R^q f)(\tau_{\leq q}\,Y^{p+1})\right)\to\\
&\to \kker\left((\R^q f)(H^q(Y^p)[-q]) \to (\R^q f)(H^q(Y^{p+1})[-q])\right) =\\
&= (\R^q f)(Z^p(E_1^{\scriptscriptstyle\bullet,q})[-q]) \to
(\R^q f)(Z^0(E_1^{\scriptscriptstyle\bullet,q})[-q+p]) = (\R^p f)(H^q(X)) =
{}^{II} E_2^{p,q}.\\
\end{align*}
As a result, the composition of $v$ with (\ref{epi}) is obtained
by applying $\R^q f$ to

\begin{equation}
\label{comp3}
\tau_{\leq q}\, U^p \to H^q(X^p)[q] \to H^q(X)[-q+p].
\end{equation}
Consequently, the composite map

$$
{}^I D_1^{p,q} = (\R^q f)(\tau_{\leq q}\, X^p) \stackrel {j_1} \longrightarrow
Z^p({}^I E_1^{\scriptscriptstyle\bullet,q}) \twoheadrightarrow {}^I E_2^{p,q}
\stackrel v \to {}^{II} E_2^{p,q}
$$
is given by applying $\R^q f$ to

$$
\tau_{\leq q}\, X^p \to H^q(X^p)[q] \to H^q(X)[-q+p],
$$
hence is equal to $j_2 u^\prime$. It follows that $v j_2 = v j_1 i_1^{-1}
= j_2 u^\prime i_1^{-1} = j_2 u$.

\smallskip\noindent
{\bf Step 7:} the diagram (\ref{comp2}) implies that the map
(\ref{comp1}) coincides with the composition of (\ref{comp3})
with the canonical map $H^q(X)[-q+p] \to (\tau_{\leq q-1}\, X)[p+1]$,
hence $u k_2 = k_2 v$. Theorem is proved.
\end{proof}

\smallskip\noindent
\begin{example}
\label{filtered}
If $K^{\scriptscriptstyle\bullet}$ is a bounded below filtered complex
over $A$ (with a finite filtration)

$$
K^{\scriptscriptstyle\bullet} = F^0 K^{\scriptscriptstyle\bullet} \supset
F^1 K^{\scriptscriptstyle\bullet} \supset \cdots \supset
F^n K^{\scriptscriptstyle\bullet} \supset F^{n+1} K^{\scriptscriptstyle\bullet}
= 0,
$$
then the objects

$$
X^p = F^p K^{\scriptscriptstyle\bullet}[p],\qquad
Y^p =
(F^p K^{\scriptscriptstyle\bullet}/F^{p+1} K^{\scriptscriptstyle\bullet})[p]
= gr^p_F(K^{\scriptscriptstyle\bullet})[p] \in D^+(A)
$$
form a Postnikov system of the kind considered in (\ref{postnikov}).
The corresponding spectral sequences are equal to

$$
E_1^{p,q} = H^{p+q}(gr^p_F(K^{\scriptscriptstyle\bullet})) \Longrightarrow
H^{p+q}(K^{\scriptscriptstyle\bullet}),\qquad
{}^I E_1^{p,q} = (\R^{p+q}f)(gr^p_F(K^{\scriptscriptstyle\bullet}))
\Longrightarrow (\R^{p+q}f)(K^{\scriptscriptstyle\bullet}).
$$
In the special case when $K^{\scriptscriptstyle\bullet}$ is the total
complex associated to a first quadrant bicomplex
$C^{\scriptscriptstyle\bullet, \scriptscriptstyle\bullet}$ and the filtration
$F^p$ is induced by the column filtration on
$C^{\scriptscriptstyle\bullet, \scriptscriptstyle\bullet}$, then the complex
$f(K^{\scriptscriptstyle\bullet})$ over $A^\prime$ is equipped with
a canonical filtration $(f F^p)(f(K^{\scriptscriptstyle\bullet})) =
f(F^p K^{\scriptscriptstyle\bullet})$ satisfying

$$
gr^p_{f(F)}(f(K^{\scriptscriptstyle\bullet})) =
f(gr^p_F(K^{\scriptscriptstyle\bullet})).
$$
Under the assumptions (\ref{ass}), the corresponding exact couple

$$
{}^f D_1^{p,q} = H^{p+q}(f(F^p K^{\scriptscriptstyle\bullet})),\qquad
{}^f E_1^{p,q} = H^{p+q}(gr^p_{f(F)}(f(K^{\scriptscriptstyle\bullet}))) =
H^{p+q}(f(gr^p_F(K^{\scriptscriptstyle\bullet}))) \Longrightarrow
H^{p+q}(f(K^{\scriptscriptstyle\bullet}))
$$
then naturally maps to the exact couple (\ref{firstcouple}), hence
(beginning from $(D_2, E_2)$) to the exact couple (\ref{secondcouple}),
by Theorem \ref{speccomp}.
\end{example}

\section{Syntomic cohomology}
In this section we will define the arithmetic and geometric syntomic cohomologies of varieties over $K$ and $\ovk$, respectively,  and study their basic properties.
\subsection{Hyodo-Kato morphism revisited}  We will need to use the Hyodo-Kato morphism on the level of derived categories and vary it in $h$-topology. Recall that the original morphism depends on the choice of a uniformizer and a change of such is encoded in a transition function involving exponential of the monodromy.  Since the fields of definition of semistable models in the bases for $h$-topology change we will need to use these transitions functions. The problem though is that in the most obvious (i.e., crystalline) definition of the Hyodo-Kato complexes the monodromy is (at best) homotopically nilpotent - making the exponential in the transition functions impossible to define. Beilinson \cite{BE2} solves this problem by representing Hyodo-Kato complexes using modules with nilpotent monodromy. In this subsection we will summarize what we need from his approach.

   At first a quick reminder. Let $(U,\overline{U})$ be a log-scheme, log-smooth over $V^{\times}$. For any $r\geq 0$, consider its absolute (meaning over $W(k)$) log-crystalline cohomology complexes
\begin{align*}
\R\Gamma_{\crr}(U,\overline{U},\sj^{[r]})_n: &  =\R\Gamma(\overline{U}_{\eet},\R u_{U^{\times}_n/W_n(k)*}\sj^{[r]}_{U^{\times}_n/W_n(k)}),\quad
\R\Gamma_{\crr}(U,\overline{U},\sj^{[r]}):=\holim_n\R\Gamma_{\crr}(U,\overline{U},\sj^{[r]})_n,\\
\R\Gamma_{\crr}(U,\overline{U},\sj^{[r]})_\bq: & =\R\Gamma_{\crr}(U,\overline{U},\sj^{[r]})\otimes\bq_p,
\end{align*}
where $U^{\times}$ denotes  the log-scheme $(U,\overline{U})$ and $u_{U^{\times}_n/W_n(k)}: (U^{\times}_n/W_n(k))_{\crr}\to \overline{U}_{\eet}$ is the projection from the log-crystalline to the \'etale topos. For $r\geq 0$,   we write $\sj^{[r]}_{U^{\times}_n/W_n(k)}$ for the r'th divided power of the canonical PD-ideal $\sj_{U^{\times}_n/W_n(k)}$; for $r\leq 0$, we set
$\sj^{[r]}_{U^{\times}_n/W_n(k)}:=\so_{U^{\times}_n/W_n(k)}$ and we will often omit it from the notation. The absolute  log-crystalline cohomology complexes 
are filtered $E_{\infty}$ algebras over $W_n(k)$, $W(k)$, or $K_0$, respectively. Moreover, the rational ones are filtered commutative dg algebras.
\begin{remark}
The canonical pullback map
$$\R\Gamma(\overline{U}_{\eet},\R u_{U^{\times}_n/W_n(k)*}\sj^{[r]}_{U^{\times}_n/W_n(k)})\stackrel{\sim}{\to} \R u_{U^{\times}_n/{\mathbf Z}/p^n*}\sj^{[r]}_{U^{\times}_n/{\mathbf Z}/p^n})
$$
is a quasi-isomorphism. In what follows we will often call both the "absolute crystalline cohomology".
\end{remark} 

 Let   $W(k)<t_l> $ be the divided powers  polynomial algebra generated by elements $t_l$, $l\in {\mathfrak m}_K/ {\mathfrak m}^2_K\setminus \{0\}$, subject to the relations $t_{al}=[\overline{a}]t_l,$ for $a\in V^*$, where   $[\overline{a}]\in W(k)$ is the Teichm\"{u}ller lift of $\overline{a}$ - the reduction mod $\mathfrak m_K$ of $a$. Let $R_{V}$ (or simply $R$) be the $p$-adic completion  of the subalgebra of $W(k)<t_l>$ generated by $t_l$ and $t_l^{ie_K}/i!$, $i\geq 1$.  For a fixed $l$, the ring $R$   is the following $W(k)$-subalgebra of $K_0[[t_l]]$:
\begin{align*}
R 
  =\{\sum_{i=0}^{\infty} a_i\frac{t_l^i}{\lfloor i/e_K\rfloor !}\mid a_i\in W(k), \lim_{i\rightarrow \infty}a_i=0\}.
\end{align*}
One extends the Frobenius $\phi_R$ (semi-linearly) to $R$ by setting $\phi_R(t_l)=t_l^p$ and defines a  monodromy operator $N_R$ as a $W(k)$-derivation by setting 
$
N_R(t_l)=-t_l.
$
Let $E:=\Spec(R)$ equipped with the log-structure generated by the $t_l$'s. 

  We have two  exact closed embeddings $$
  i_0:W(k)^0\hookrightarrow E,\quad i_{\pi}: V^{\times}\hookrightarrow E. 
  $$ 
The first one is canonical and   induced by $
t_l\mapsto 0$.   The second one depends on the choice of the class of the uniformizing parameter $\pi\in {\mathfrak m}_K/p{\mathfrak m}_{K}$ up to multiplication by Teichm\"{u}ller elements. It is induced by $t_l\mapsto [\overline{l/\pi}]\pi$.

 Assume that $(U,\overline{U})$ is  of Cartier type (i.e., the special fiber $\overline{U}_0$ is of Cartier type). Consider the log-crystalline and the Hyodo-Kato complexes (cf. \cite[1.16]{BE2})
 $$
\R\Gamma_{\crr}((U,\overline{U})/R,\sj^{[r]})_n:  =\R\Gamma_{\crr}((U,\overline{U})_{n}/R_n,\sj^{[r]}_{\overline{U}_n/R_n}),\quad
\R\Gamma_{\hk}(U,\overline{U})_n:=\R\Gamma_{\crr}((U,\overline{U})_0/W_n(k)^0).
$$
Let $
\R\Gamma_{\crr}((U,\overline{U})/R,\sj^{[r]})$ and $\R\Gamma_{\hk}(U,\overline{U})$ be their homotopy inverse limits. The last complex is called the {\em Hyodo-Kato} complex.
The complex $\R\Gamma_{\crr}((U,\overline{U})/R)$ is $R$-perfect and $$\R\Gamma_{\crr}((U,\overline{U})/R)_n\simeq \R\Gamma_{\crr}((U,\overline{U})/R)\otimes^{L}_RR_n\simeq \R\Gamma_{\crr}((U,\overline{U})/R)\otimes^L{\mathbf Z}/p^n. $$
In general, we have
 $\R\Gamma_{\crr}((U,\overline{U})/R,\sj^{[r]})_n\simeq \R\Gamma_{\crr}((U,\overline{U})/R,\sj^{[r]})\otimes^L{\mathbf Z}/p^n$.
 % if necessary use filtered derived category
 The complex $\R\Gamma_{\hk}(U,\overline{U})$ is $W(k)$-perfect and
 $$\R\Gamma_{\hk}(U,\overline{U})_n\simeq \R\Gamma_{\hk}(U,\overline{U})\otimes^L_{W(k)}W_n(k)\simeq \R\Gamma_{\hk}(U,\overline{U})\otimes^L{\mathbf Z}/p^n. $$
 We normalize  the  monodromy  operators $N$ on the rational complexes $\R\Gamma_{\crr}((U,\overline{U})/R)_{\bq}$ and $\R\Gamma_{\hk}(U,\overline{U})_\bq$ by replacing the standard $N$ \cite[3.6]{HK} by $N_R:=e_{K}^{-1}N$. This makes them  compatible with base change. 
 The embedding $i_0: (U,\overline{U})_0\hookrightarrow (U,\overline{U}) $ over $i_0: W_n(k)^0\hookrightarrow E_n$ yields compatible morphisms $i^*_{0,n}: \R\Gamma_{\crr}((U,\overline{U})/R)_n\to \R\Gamma_{\hk}(U,\overline{U})_n$. Completing, we get  a morphism $$i^*_0: \R\Gamma_{\crr}((U,\overline{U})/R)\to \R\Gamma_{\hk}(U,\overline{U}),$$ 
 which induces a quasi-isomorphism
 $i^*_0:\R\Gamma_{\crr}((U,\overline{U})/R)\otimes ^L_RW(k)\stackrel{\sim}{\to}\R\Gamma_{\hk}(U,\overline{U})$.  All the above objects have an action of Frobenius and these morphisms are compatible with Frobenius. The Frobenius action is invertible on $\R\Gamma_{\hk}(U,\overline{U})_{\mathbf Q}$.

   The  map $i^*_{0}: \R\Gamma_{\crr}((U,\overline{U})/R)_{{\mathbf Q}}\to \R\Gamma_{\hk}(U,\overline{U})_{{\mathbf Q}}$  admits a unique (in the classical derived category) $W(k)$-linear section $\iota_{\pi}$  \cite[1.16]{BE2}, \cite[4.4.6]{Ts} that commutes with $\phi$ and $N$. The map $\iota_{\pi}$ is functorial and its $R$-linear extension is a quasi-isomorphism
  $$\iota_{\pi}: R\otimes_{W(k)}\R\Gamma_{\hk}(U,\overline{U})_{{\mathbf Q}}\stackrel{\sim}{\to} \R\Gamma_{\crr}((U,\overline{U})/R) _{{\mathbf Q}}.$$  The composition (the {\em Hyodo-Kato map})
  $$
  \iota_{\dr,\pi}:=\gamma_r^{-1}i^*_{\pi}\cdot\iota_{\pi}:\quad  \R\Gamma_{\hk}(U,\overline{U})_{{\mathbf Q}}\to \R\Gamma_{\dr}(U,\overline{U}_{K}),  $$
where 
$$
\gamma_r^{-1}:\quad \R\Gamma_{\crr}(U,\overline{U},\so/\sj^{[r]})_{{\mathbf Q}}\stackrel{\sim}{\to}\R\Gamma_{\dr}(U,\overline{U}_{K})/F^r
$$
 is the quasi-isomorphism from Corollary \ref{Langer}, 
  induces a $K$-linear functorial quasi-isomorphism (the {\em Hyodo-Kato quasi-isomorphism}) \cite[4.4.8, 4.4.13]{Ts}
  \begin{equation}
  \label{HKqis}
  \iota_{\dr,\pi}: \R\Gamma_{\hk}(U,\overline{U})\otimes_{W(k)}K\stackrel{\sim}{\to}\R\Gamma_{\dr}(U,\overline{U}_{K}) \end{equation}

   We are going now to describe the Beilinson-Hyodo-Kato morphism and to  study it on a few examples. Let $S_n=\Spec({\mathbf Z}/p^n)$ equipped with the trivial log-structure and let $S=\Spf({\mathbf Z}_p)$ be the induced formal log-scheme. For any log-scheme $Y\to S_1$ let $D_{\phi}((Y/S)_{\crr},\so_{Y/S})$ denote the derived category of Frobenius $\so_{Y/S}$-modules and $D_{\phi}^{\pcr}(Y/S)$ its thick subcategory of perfect F-crystals, i.e., those Frobenius modules that are perfect crystals \cite[1.11]{BE2}. We call a perfect F-crystal $(\sff,\phi)$ {\em non-degenerate} if the map $L\phi^*(\sff)\to \sff$ is an isogeny. The corresponding derived category is denoted by $D_{\phi}^{\pcr}(Y/S)^{\nd}$. It has a dg category structure \cite[1.14]{BE2} that we denote by $\sd_{\phi}^{\pcr}(Y/S)^{\nd}$.
We will omit $S$ if understood.

  Suppose now that $Y$ is a fine log-scheme that is affine. Assume also that
 there is a PD-thickening 
$P=\Spf R$ of  $Y$ that is formally smooth over $S$
 and such that  $R$ is a $p$-adically complete ring with no  $p$-torsion.
 Let $f:Z\to Y$ be a log-smooth map of Cartier type with $Z$ fine and proper over $Y$. Beilinson \cite[1.11,1.14]{BE2} proves the following theorem.
\begin{theorem}
\label{kk1}
The complex  $\sff:=Rf_{\crr*}(\so_{Z/S})$ is a non-degenerate perfect F-crystal.
\end{theorem}
   
   Let $D_{\phi,N}(K_0)$ denote the bounded derived category of $(\phi, N)$-modules. By \cite[1.15]{BE2}, it has a dg category structure that we will denote by $\sd_{\phi,N}(K_0)$. We call $(\phi,N)$-module {\em effective} if it contains a $W(k)$-lattice preserved by $\phi$ and $N$. Denote by $\sd_{\phi,N}(K_0)^{\eff}\subset \sd_{\phi,N}(K_0)$  the bounded derived category of the abelian category of effective modules. 

  Let $f:Y\to k^0$ be a log-scheme. We think of $k^0$ as $W(k)^{\times}_1$. Then the map $f$ is given by a $k$-structure on $Y$ plus a section $l=f^*(\overline{p})\in\Gamma(Y,M_Y)$ such that its image in $\Gamma(Y,\so_Y)$ equals $0$. We will often write $f=f_l,l=l_f$. 

  Beilinson proves the following theorem \cite[1.15]{BE2}. 
\begin{theorem}
\label{kk2}
\begin{enumerate}
\item There is a natural functor
\begin{equation}
\label{kwaku-kwik}
\epsilon_{f}=\epsilon_l: \sd_{\phi,N}(K_0)^{\eff}\to \sd^{\pcr}_{\phi}(Y)^{\nd}\otimes \bq.
\end{equation}
\item $\epsilon_{f}$ is compatible with base change, i.e., for any $\theta: Y^{\prime}\to Y$ one has a canonical identification $\epsilon_{f\theta}\stackrel{\sim}{\to}L\theta^*_{\crr}\epsilon_{f}$. For any $a\in k^*, m\in {\mathbf Z}_{>0}$, there is a canonical identification $\epsilon_{al^m}(V,\phi,N)\stackrel{\sim}{\to}\epsilon_l(V,\phi,mN)$.
\item Suppose that $Y$ is a local scheme with residue field $k$ and nilpotent maximal ideal, $M_Y/\so^*_Y={\mathbf Z}_{>0}$, and the map $f^*:M_{k^0}/k^*\to M_Y/\so^*_Y$ is injective. Then (\ref{kwaku-kwik}) is an equivalence of dg categories. 
\end{enumerate}
\end{theorem}
In particular, we have an equivalence of dg categories
$$\epsilon:=\epsilon_{\overline{p}}: \sd_{\phi,N}(K_0)^{\eff}\stackrel{\sim}{\to}\sd^{\pcr}_{\phi}(k^0)^{\nd}\otimes\bq
$$ and a canonical identification $\epsilon_f=Lf^*_{\crr}\epsilon$.

   On the level of sections the functor (\ref{kwaku-kwik}) has a simple description \cite[1.15.3]{BE2}. Assume that $Y=\Spec(A/J)$, where $A$ is a $p$-adic algebra and $J$ is a PD-ideal in $A$, and that we have a PD-thickening $i:Y\hookrightarrow T=\Spf(A)$. Let $\lambda_{l,n}$ be the preimage of $l$ under the map $\Gamma(T_n,M_{T_n})\to i_*\Gamma(Y,M_Y)$. It is a trivial $(1+J_n)^{\times}$-torsor. Set $\lambda_A:=\invlim{_n}\Gamma(T_n,\lambda_{l,n})$. It is a $(1+J)^{\times}$-torsor. Let $\tau_{A_\bq}$ be the {\em Fontaine-Hyodo-Kato} torsor: $A_\bq$-torsor obtained from $\lambda_A$ by the pushout by $(1+J)^{\times}\stackrel{\log}{\to}J\to A_\bq$. We call the ${\mathbb G}_a$-torsor $\Spec A_\bq^{\tau}$ over $\Spec A_\bq$ with sections $\tau_{A_\bq}$ the same name. Denote by $N_{\tau}$ the $A_\bq$-derivation of $A_\bq^{\tau}$ given by the action of the generator of $\Lie_{{\mathbb G_a}}$.

  Let $M$ be an $(\phi, N)$-module. Integrating the action of the monodromy  $N_M$ we get an action of the group ${\mathbb G}_a$ on $M$. Denote by $M^{\tau}_{A_\bq}$ the $\tau_{A_\bq}$-twist of $M_{A_\bq}:=M\otimes_{K_0}A_\bq$. It can be represented as  the module of maps $v:\tau_{A_\bq}\to M_{A_\bq}$ that are $A_{\bq}$-equivariant, i.e., such that $v(\tau+a)=\exp(aN)(v(\tau))$, $\tau\in \tau_{A_\bq}$, $a\in A_\bq$. We can also write 
$$M^{\tau}_{A_\bq}=(M\otimes_{K_0}A^{\tau}_\bq)^{{\mathbb G}_a}=(M\otimes_{K_0}A^{\tau}_\bq)^{N=0},
$$
where $N:=N_M\otimes 1+ 1\otimes N_{\tau}$. Now, by definition,
\begin{equation}
\label{isom}
\epsilon_f(M)(Y,T)=M^{\tau}_{A_\bq}
\end{equation}

  The algebra $A^{\tau}_\bq$ has a concrete description. Take the natural map $a:\tau_{A_\bq}\to A^{\tau}_\bq$ of $A_\bq$-torsors which maps $\tau\in \tau_{A_\bq}$ to a function $a({\tau})\in A^{\tau}_\bq$ whose value on any $\tau^{\prime}\in \tau_{A_\bq}$ is $\tau-\tau^{\prime}\in A_\bq$. This map is compatible with the logarithm $\log: (1+J)^{\times}\to A$. The algebra $A^{\tau}_{A_\bq}$ is freely generated over $A_\bq$ by  $a({\tau})$ for any $\tau\in\tau_{A_\bq}$; the $A_\bq$-derivation $N_{\tau}$ is defined by $N_{\tau}(a({\tau}))=-1$.
That is, for chosen $\tau\in\tau_{A_\bq}$, we can write
\begin{align*}
A^{\tau}_{\bq}  =A_\bq[a({\tau})],\quad
N_{\tau}(a({\tau}))=-1
\end{align*}
For every lifting $\phi_T$ of Frobenius to $T$ we have $\phi^*_T\lambda_A=\lambda_A^p$. Hence  Frobenius $\phi_T$ extends canonically to a Frobenius $\phi_{\tau}$ on $A_\bq^{\tau}$ in such a way that $N_\tau\phi_{\tau}=p\phi_{\tau} N_{\tau}$. The isomorphism (\ref{isom}) is compatible with Frobenius.

  \begin{example}
\label{standard}
  As an example, consider the case when the pullback map $f^*:\bq=(M_{k^0}/k^*)^{\gp}\otimes\bq\stackrel{\sim}{\to} (\Gamma(Y,M_Y)/k^*)^{\gp}\otimes\bq$ is an isomorphism. We have a surjection $v: (\Gamma(T,M_T)/k^*)^{\gp}\otimes \bq \to \bq$ with the kernel $\log: (1+J)^{\times}_\bq\stackrel{\sim}{\to} J_\bq=A_{\bq}$. We obtain an identification of $A_\bq$-torsors $\tau_{A_\bq}\simeq v^{-1}(1)$. Hence every non-invertible $t\in \Gamma(T,M_T)$ yields an element $t^{1/v(t)}\in v^{-1}(1)$ and a trivialization of $\tau_{A_\bq}$.

   For a fixed element $t^{1/v(t)}\in v^{-1}(1)$, we can write
\begin{align*}
A^{\tau}_{A_\bq}  =A_\bq[a (t^{1/v(t))}], \quad
N_{\tau}(a( t^{1/v(t))})=-1
\end{align*}
 For an $(\phi, N)$-module $M$, the twist $M^{\tau}_{A_\bq}$ can be trivialized \begin{align*}
\beta_t: M\otimes_{K_0  }A_\bq & \stackrel{\sim}{\to} M^{\tau}_{A_\bq}=(M\otimes_{K_0}A_\bq[a( t^{1/v(t)})])^{N=0} \\
m & \mapsto \exp(N_M(m) a(t^{1/v(t))})
\end{align*}
For a different choice  $t_1^{1/v(t_1)}\in v^{-1}(1)$, the two trivializations $\beta_t,\beta_{t_1}$ are related by the formula 
$$
\beta_{t_1}=\beta_t\exp(N_M(m)a(t_1,t)),\quad a(t_1,t)=a(t_1)/v(t_1)-a(t)/v(t).
$$ 
\end{example}

  Consider the map $f:V^{\times}_1\to k^0$. By Theorem \ref{kk2},  we have the equivalences of dg categories
\begin{align*}
\epsilon:  & \quad \sd_{\phi,N}(K_0)^{\eff}\stackrel{\sim}{\to}\sd^{\pcr}_{\phi}(k^0)^{\nd}\otimes\bq,\\
\epsilon_{f}=Lf^*_{\crr}\epsilon: & \quad \sd_{\phi,N}(K_0)^{\eff}\stackrel{\sim}{\to}\sd^{\pcr}_{\phi}(V_1^{\times})^{\nd}\otimes\bq
\end{align*}
  
  Let $Z_1\to V_1^{\times}$ be a log-smooth map of Cartier type with $Z_1$  fine and proper over $V_1$. By Theorem \ref{kk1} $Rf_{\crr*}(\so_{Z_1/{\mathbf Z}_p})$ is a non-degenerate perfect F-crystal on $V_{1,\crr}$. Set
$$
\R\Gamma^{B}_{\hk}(Z_1):=\epsilon^{-1}_{f}Rf_{\crr*}(\so_{Z_1/{\mathbf Z}_p})_\bq \in \sd_{\phi,N}(K_0).
$$
We will call it the {\em Beilinson-Hyodo-Kato} complex \cite[1.16.1]{BE2}. \begin{example}
\label{crucial}
To get familiar with the Beilinson-Hyodo-Kato complexes we will work out some examples.
\begin{enumerate}
 \item Let $g:X\to V^{\times}$ be a log-smooth log-scheme, proper, and of Cartier type. 
Adjunction yields a quasi-isomorphism
\begin{align}
\label{adjunction}
\epsilon_f\R\Gamma^{B}_{\hk}(X_1)=\epsilon_f\epsilon^{-1}_{f}Rg_{\crr*}(\so_{X_1/{\mathbf Z}_p})_\bq\stackrel{\sim}{\to} Rg_{\crr*}(\so_{X_1/{\mathbf Z}_p})_\bq
\end{align}
Evaluating it on the PD-thickening $V^{\times}_1\hookrightarrow V^{\times}$ (here $A=V$, $J=pV$, $l=\overline{p}$, $\lambda_V=p(1+J)^{\times}$, $\tau_K=p(1+J)^{\times}\times_{(1+J)^{\times}}K$), we get a map 
\begin{align*}
\R\Gamma^{B}_{\hk}(X_1)^{\tau}_K & =\epsilon_f\R\Gamma^{B}_{\hk}(X_1)(V^{\times}_1\hookrightarrow V^{\times})
\stackrel{\sim}{\to} Rg_{\crr*}(\so_{X_1/{\mathbf Z}_p})(V^{\times}_1\hookrightarrow V^{\times})_\bq=\R\Gamma_{\crr}(X_1/V^{\times})_\bq\\
 & \simeq \R\Gamma_{\crr}(X/V^{\times})_\bq\simeq \R\Gamma_{\dr}(X_K)
\end{align*}
We will call it the {\em Beilinson-Hyodo-Kato} map \cite[1.16.3]{BE2}
\begin{equation}
\label{HK1}
\iota^B_{\dr}:\quad \R\Gamma_{\hk}^B(X_1)^{\tau}_K\stackrel{\sim}{\to}\R\Gamma_{\dr}(X_K)
\end{equation}
Recall that 
$$
\R\Gamma_{\hk}^B(X_1)^{\tau}_K=(\R\Gamma_{\hk}^B(X_1)\otimes_{K_0}K[a(\tau)])^{N=0},\quad \tau\in\tau_K
$$
This makes it clear that the Beilinson-Hyodo-Kato map is not only functorial for log-schemes over $V^{\times}$ but, by Theorem \ref{kk2}, it is also compatible with base change of $V^{\times}$. Moreover, if we use
 the canonical trivialization by $p$
\begin{align*}
\beta=\beta_p:\quad \R\Gamma_{\hk}^B(X_1)_K & \stackrel{\sim}{\to}\R\Gamma^B_{\hk}(X)_K^{\tau}
=(\R\Gamma_{\hk}^B(X_1)\otimes_{K_0}K[a(p)])^{N=0}\\
& x\mapsto \exp(N(x)a(p))
\end{align*}
we get that the composition (which we also call the Beilinson-Hyodo-Kato map and denote by $\iota^B_{\dr}$) 
$$
\iota^B_{\dr}=\iota^B_{\dr}\beta:\quad  \R\Gamma_{\hk}^B(X_1)\to \R\Gamma_{\dr}(X_K)
$$
is functorial and compatible with base change. 

\item 
Evaluating the map (\ref{adjunction}) on the PD-thickening $V^{\times}_1\hookrightarrow E$ associated to a uniformizer $\pi$ (here $A=R$, $l=\overline{p}$), we get a map 
\begin{equation}
\label{kappar}
\kappa_R:\quad \R\Gamma^{B}_{\hk}(X_1)^{\tau}_{R_\bq} \stackrel{\sim}{\to}   \R\Gamma_{\crr}(X/R)_\bq
\end{equation}
as the composition
\begin{align*}
\R\Gamma^{B}_{\hk}(X_1)^{\tau}_{R_\bq} & =\epsilon_f\R\Gamma^{B}_{\hk}(X_1)(V^{\times}_1\hookrightarrow E)
\stackrel{\sim}{\to} Rg_{\crr*}(\so_{X_1/{\mathbf Z}_p})(V^{\times}_1\hookrightarrow E)_\bq=\R\Gamma_{\crr}(X_1/R)_\bq\\
 & \simeq \R\Gamma_{\crr}(X/R)_\bq
\end{align*}
We have 
$$
\R\Gamma_{\hk}^B(X_1)^{\tau}_{R_\bq}=(\R\Gamma_{\hk}^B(X_1)\otimes_{K_0}R_{\bq}[a(\tau)])^{N=0},\quad \tau\in\tau_{R_\bq}
$$
Since the map $\kappa_R$ is compatible with the log-connection on $R$ it is also compatible with the normalized monodromy operators. Specifically, if we define the monodromy on the left hand side of (\ref{kappar}) as 
\begin{align*}
N:\quad \R\Gamma^{B}_{\hk}(X_1)^{\tau}_{R_\bq} & \to \R\Gamma^{B}_{\hk}(X_1)^{\tau}_{R_\bq},\\
\sum_I m_{\tau_I}\otimes r_{\tau_I}a^{k_I}(\tau_I) & \mapsto \sum_I (N_M(m_{\tau_I})\otimes r_{\tau_I}a^{k_I}(\tau_I) + m_{\tau_I}\otimes N_R(r_{\tau_I})a^{k_I}(\tau_I))
\end{align*}
 the two operators will correspond under the map $\kappa_R$. 

    The exact immersion $i_{\pi}: V^{\times}\hookrightarrow E$, yields a commutative diagram
$$
\xymatrix{
\R\Gamma^{B}_{\hk}(X_1)^{\tau}_{R_\bq}\ar[r]^{\sim}\ar[d]^{i^*_{\pi}}
 & \R\Gamma_{\crr}(X/R)_\bq\ar[d]^{i^*_{\pi}}\\
\R\Gamma^{B}_{\hk}(X_1)^{\tau}_{K}\ar[r]^{\sim} &  \R\Gamma_{\crr}(X/V^{\times})_\bq }
$$
If $p=u\pi^{e_K}$, $u\in V^{\times}$,  we have $\lambda_R=\tilde{u}t_{\pi}^{e_K}(1+J)^{\times}$, where $\tilde{u}\in R$ is such that $\tilde{u}$ lifts $u$. Alternatively, $\lambda_R=[\overline{u}]t_{\pi}^{e_K}(1+J)^{\times}$. We have the associated trivialization
\begin{align*}
\beta_{\pi}:\quad \R\Gamma_{\hk}^B(X_1)\otimes_{K_0}R_{\bq} & \stackrel{\sim}{\to}\R\Gamma_{\hk}^B(X_1)^{\tau}_{R_\bq}=(\R\Gamma_{\hk}^B(X_1)\otimes_{K_0}R_{\bq}[a(\tau_{\pi})])^{N=0}, \quad \tau_{\pi}:=[\overline{u}]t^{e_K}_{\pi},\\
& x\mapsto \exp(N(x)a(\tau_{\pi}))
\end{align*}

\item  Consider the log-scheme $k^0_1$: the scheme $\Spec(k)$ with the log-structure induced by the exact closed immersion $i:k^0_1\hookrightarrow V^{\times}_1$. We have the commutative diagram
$$
\xymatrix{
X_0\ar@{^{(}->}[r]^i\ar[d]^{g_0} & X_1\ar[d]^g\\
k^0_1\ar[rd]_{f_0}\ar@{^{(}->}[r]^i & V^{\times}_1\ar[d]^f\\
 & k^0
}
$$
The morphisms $f, f_0$ map $\overline{p}$ to $\overline{p}$. By log-smooth base change we have a canonical quasi-isomorphism 
$Li^*Rg_{\crr*}(\so_{X_1/{{\mathbf Z}_p}
})\simeq Rg_{0\crr*}(\so_{X_1/{{\mathbf Z}_p}})$. By Theorem \ref{kk2} we have the equivalence of dg categories
$$
\epsilon_{f_0}:  \quad \sd_{\phi,N}(K_0)^{\eff}\stackrel{\sim}{\to}\sd^{\pcr}_{\phi}(k^0_1)^{\nd}\otimes\bq,\quad 
\epsilon_{f_0}=Li^*\epsilon_f.
$$
This implies the natural quasi-isomorphisms
\begin{align*}
\R\Gamma_{\hk}^B(X_1) & =\epsilon^{-1}_fRg_{\crr*}(\so_{X_1/{{\mathbf Z}_p}})_\bq \simeq 
\epsilon^{-1}_{f_0}Li^*Rg_{\crr*}(\so_{X_1/{{\mathbf Z}_p}})_\bq\\
& \simeq 
\epsilon_{f_0}^{-1}Rg_{0\crr*}(\so_{X_0/{{\mathbf Z}_p}})_\bq
\end{align*}
Hence, by adjunction,
$$
\epsilon_{f_0}\R\Gamma_{\hk}^B(X_1) =\epsilon_{f_0}\epsilon_{f_0}^{-1}Rg_{0\crr*}(\so_{X_0/{{\mathbf Z}_p}})_\bq\simeq 
Rg_{0\crr*}(\so_{X_0/{{\mathbf Z}_p}})_\bq
$$
We will evaluate  both sides on the PD-thickening $ k^0_1\hookrightarrow W(k)^0$. Here we write the log-structure on $W(k)^0$ as associated to the map $\Gamma(V^{\times},M_{V^{\times}})\to k\to W(k)$, $a\mapsto \overline{a}$. We take $A=W(k),$  $l={p}$, $J=pW(k)$, $\lambda_{W(k)}=\overline{p}(1+pW(k))^{\times}$, $\tau_{K_0}=\overline{p}(1+pW(k))^{\times}\times_{(1+pW(k))^{\times}}K_0$. We get a quasi-isomorphism
$$\kappa:\quad 
\R\Gamma^{B}_{\hk}(X_1)^{\tau}_{K_0} \stackrel{\sim}{\to}  \R\Gamma_{\hk}(X)_\bq
$$
as the composition
\begin{align*}
\R\Gamma_{\hk}^B(X_1)^{\tau}_{K_0} & =\epsilon_{f_0}\R\Gamma_{\hk}^B(X_1) (k^0_1\hookrightarrow W(k)^0)\simeq 
Rg_{0\crr*}(\so_{X_0/{{\mathbf Z}_p}})(k^0_1\hookrightarrow W(k)^0)_\bq\\
& =\R\Gamma_{\crr}(X_0/W(k)^0)_\bq=\R\Gamma_{\hk}(X)_\bq
\end{align*}
To compare the monodromy operators on  both sides of the map $\kappa$, note that by Theorem \ref{kk2}, we have the canonical identification
$$Rg_{0\crr*}(\so_{X_0/{{\mathbf Z}_p}})_\bq\simeq 
\epsilon_{f_0}(\R\Gamma_{\hk}^B(X_1),N) \simeq \epsilon_{\overline{p}}(\R\Gamma_{\hk}^B(X_1) ,e_KN)
$$
Hence, from the description of the Hyodo-Kato monodromy in \cite[3.6]{HK}, it follows easily that the map $\kappa$ pairs   the operator $N$ on $\R\Gamma_{\hk}^B(X_1)^{\tau}_{K_0}$ defined by
$$N(\sum_I m_{\tau_I}\otimes r_{\tau_I}a^{k_I}(\tau_I))=\sum_I (N_M(m_{\tau_I})\otimes r_{\tau_I}a^{k_I}(\tau_I) + m_{\tau_I}\otimes N_R(r_{\tau_I})a^{k_I}(\tau_I)),
$$
with the normalized Hyodo-Kato monodromy on $\R\Gamma_{\hk}(X)_\bq$. 

  Composing the map $\kappa$ with the trivialization
\begin{align*}
\beta=\beta_{p}:\quad \R\Gamma_{\hk}^B(X_1) & \stackrel{\sim}{\to}\R\Gamma_{\hk}^B(X_1)^{\tau}_{K_0}
=(\R\Gamma_{\hk}^B(X_1)[a(\overline{p})])^{N=0}\\
& x\mapsto \exp(N(x)a(\overline{p}))
\end{align*}
we get a quasi-isomorphism between Beilinson-Hyodo-Kato complexes and the (classical) Hyodo-Kato complexes.
\begin{align}
\label{B=K}
\kappa=\beta\kappa:\quad \R\Gamma_{\hk}^B(X_1)\stackrel{\sim}{\to}\R\Gamma_{\hk}(X)_\bq
\end{align}
The trivialization above is compatible with Frobenius and the normalized monodromy hence so is the quasi-isomorphism (\ref{B=K}). It is clearly functorial and, by Theorem \ref{kk2}, compatible with base change.

  By functoriality (Theorem \ref{kk2}), the morphism of PD-thickenings  (exact closed immersion) $i_0: (k^0_1\hookrightarrow W(k)^0)\hookrightarrow (V^{\times}_1\hookrightarrow R)$ yields the right square in the following  diagram
\begin{equation}
\label{infinity-category}
\xymatrix{
\R\Gamma_{\hk}(X)_\bq\ar[r]^{\iota_{\pi}} & \R\Gamma_{\crr}(X_1/R)_\bq \ar[r]^{i^*_0} & \R\Gamma_{\hk}(X)_\bq\\ 
  \R\Gamma_{\hk}^B(X_1)^{\tau}_{K_0}\ar[u]^{\wr}_{\kappa}\ar[r]^{\iota_{\pi}}
 & \R\Gamma_{\hk}^B(X_1)^{\tau}_{R_\bq}\ar[u]^{\wr}_{\kappa_R}\ar[r]^{i^*_0} &
 \R\Gamma_{\hk}^B(X_1)^{\tau}_{K_0}\ar[u]^{\wr}_{\kappa}
}
\end{equation}
In the left square the bottom map $\iota_{\pi}$ is induced by the natural map $ K_0\to R$ and by sending $a(\overline{p})\mapsto a(\tau_{\pi})$.
 It is a (right) section to $i_0^*$ and it (together with the vertical maps) commutes with Frobenius. By uniqueness of the top map $\iota_{\pi}$ this makes the left square commute in the classical derived category (of abelian groups).

  It is easy to check that we have the  following commutative diagram
$$
\xymatrix{
\R\Gamma_{\hk}^B(X_1)^{\tau}_{K_0}\ar[r]^{\iota_{\pi}} & \R\Gamma_{\hk}^B(X_1)^{\tau}_{R_\bq}\ar[r]^{i_{\pi}^*} & \R\Gamma_{\hk}^B(X_1)^{\tau}_K\\
\R\Gamma_{\hk}^B(X_1)\ar[u]^{\beta_p}_{\wr}\ar[rr]^{\can} & & \R\Gamma^B_{\hk}(X_1)_K\ar[u]^{\beta_p}_{\wr}
}
$$
and that the composition of maps on the top of it is equal to the map induced by the canonical map $K_0\to K$ and the map $\lambda_{W(k)^0}
\to \lambda_V^{\times}$, $\overline{p}\to p$.
\end{enumerate}
  Combining the commutative diagrams in  parts (2) and (3) of this example we get the following commutative diagram.
$$
\xymatrix{
\R\Gamma_{\hk}(X)\ar[r]^{\iota_{\pi}} & \R\Gamma_{\crr}(X_1/R)_\bq\ar[r]^{i_{\pi}^*} & \R\Gamma_{\crr}(X_1/V^{\times})_\bq\\
\R\Gamma^B_{\hk}(X_1)^{\tau}_{K_0} \ar[u]^{\wr}_{\kappa}\ar[r]^{\iota_{\pi}} & \R\Gamma^B_{\hk}(X_1)_{R_\bq}^{\tau} \ar[r]^{i^*_{\pi}}\ar[u]^{\wr}_{\kappa_R} &  \R\Gamma^B_{\hk}(X_1)_K^{\tau}\ar[u]^{\wr}_{\iota^B_{\dr}}\\
\R\Gamma_{\hk}^B(X_1)\ar[u]^{\beta_p}_{\wr}\ar[rr]^{\can} & & \R\Gamma^B_{\hk}(X_1)_K\ar[u]^{\beta_p}_{\wr}
}
$$
Since the composition of the top maps is equal to the Hyodo-Kato map $\iota_{\dr}$ and the bottom map is just the canonical map 
$\R\Gamma_{\hk}(X_1)\to \R\Gamma_{\hk}(X_1)_K$ we obtain that the Hyodo-Kato and the Beilinson-Hyodo-Kato maps are related by a natural quasi-isomorphism, i.e., that the following diagram commutes.
\begin{equation}
\label{Beilinson=HK}
\xymatrix{
\R\Gamma_{\hk}(X)\ar[r]^{\iota_{\dr,\pi}} & \R\Gamma_{\dr}(X_{K})\\
\R\Gamma^B_{\hk}(X_1)\ar[ru]_{\iota_{\dr}^B}\ar[u]^{\wr}_{\kappa} 
}
\end{equation}
\end{example} 

   The above examples can be generalized \cite[1.16]{BE2}. It turns out that the relative crystalline cohomology of all the base changes of the map $f$ can be described using the Beilinson-Hyodo-Kato complexes \cite[1.16.2]{BE2}. Namely, let $\theta:Y\to V^{\times}_1$ be an affine log-scheme and let $T$ be a $p$-adic PD-thickening of $Y$, $T=\Spf(A),$ $Y=\Spec(A/J)$. Denote by $f_Y:Z_{1Y}\to Y$ the $\theta$-pullback of $f$. Beilinson proves the following theorem \cite[1.16.2]{BE2}.
\begin{theorem}
\label{Bthm}
 \begin{enumerate}
  \item The $A$-complex $\R\Gamma_{\crr}(Z_{1Y}/T,\so_{Z_{1Y/T}})$ is perfect, and one has $$
\R\Gamma_{\crr}(Z_{1Y}/T_n,\so_{Z_{1Y/T_n}})=\R\Gamma_{\crr}(Z_{1Y}/T,\so_{Z_{1Y/T}})\otimes^{L}{\mathbf Z}/p^n.
$$
\item There is a canonical Beilinson-Hyodo-Kato quasi-isomorphism of $A_\bq$-complexes
\begin{equation*}
 \kappa_{A_\bq}^B: \R\Gamma^B_{\hk}(Z_1)^{\tau}_{A_\bq}\stackrel{\sim}{\to}\R\Gamma_{\crr}(Z_{1Y}/T,\so_{Z_{1Y/T}})_\bq
\end{equation*}
If there is a Frobenius lifting $\phi_T$, then $\kappa^B_{A_\bq}$ commutes with its action.
 \end{enumerate}
\end{theorem}
\subsection{Log-syntomic cohomology}
We will study now (rational) log-syntomic cohomology.
  Let $(U,\overline{U})$ be log-smooth over $V^{\times}$. For $r\geq 0$, define the mod $p^n$, completed, and rational log-syntomic complexes
\begin{align}
\label{log-syntomic}
\R\Gamma_{\synt}(U,\overline{U},r)_n & :=
\Cone(\R\Gamma_{\crr}(U,\overline{U},\sj^{[r]})_n\verylomapr{p^r-\phi} \R\Gamma_{\crr}(U,\overline{U})_n)[-1],\\
  \R\Gamma_{\synt}(U,\overline{U},r) & :=\holim _n\R\Gamma_{\synt}(U,\overline{U},r)_n,\notag\\
 \R\Gamma_{\synt}(U,\overline{U},r)_{\mathbf Q} & :=
\Cone(\R\Gamma_{\crr}(U,\overline{U},\sj^{[r]})_\bq\verylomapr{1-\phi_r} \R\Gamma_{\crr}(U,\overline{U})_\bq)[-1].\notag
\end{align}
Here the Frobenius $\phi$ is defined by the composition
\begin{align*}
\phi: \R\Gamma_{\crr}(U,\overline{U},\sj^{[r]})_n  & \to \R\Gamma_{\crr}(U,\overline{U})_n \stackrel{\sim}{\to }\R\Gamma_{\crr}((U,\overline{U})_1/W(k))_n
 \stackrel{\phi}{\to}
\R\Gamma_{\crr}((U,\overline{U})_1/W(k))_n\\
 & \stackrel{\sim}{\leftarrow}\R\Gamma_{\crr}(U,\overline{U})_n
\end{align*}
and $\phi_r:=\phi/p^r$. The mapping fibers are taken in the $\infty$-derived category of abelian groups.  The direct sums
$$
  \bigoplus _{r\geq 0}\R\Gamma_{\synt}(U,\overline{U},r)_n,\quad  \bigoplus_{r\geq 0} \R\Gamma_{\synt}(U,\overline{U},r),\quad
  \bigoplus_{r \geq 0} \R\Gamma_{\synt}(U,\overline{U},r)_\bq 
$$
are graded $E_{\infty}$ algebras over ${\mathbf Z}/p^n$, ${\mathbf Z}_p$, and $\bq_p$, respectively \cite[1.6]{HS}. The rational log-syntomic complexes are moreover graded commutative dg algebras over $\bq_p$ \cite[4.1]{HS}, \cite[3.22]{MG}, \cite{Lu2}.
Explicit definition of syntomic product structure  can be found in \cite[2.2]{Ts}.

We have
$ \R\Gamma_{\synt}(U,\overline{U},r)_n\simeq \R\Gamma_{\synt}(U,\overline{U},r)\otimes^L{\mathbf Z}/p^n. $
There is  a canonical quasi-isomorphism of graded $E_{\infty}$ algebras
\begin{align*}
\R\Gamma_{\synt}(U,\overline{U},r)_n  \stackrel{\sim}{\to}\Cone(\R\Gamma_{\crr}(U,\overline{U})_n\verylomapr{(p^r-\phi,\can)} \R\Gamma_{\crr}(U,\overline{U})_n\oplus \R\Gamma_{\crr}(U,\overline{U},\so/\sj^{[r]})_n)[-1].
\end{align*}
Similarly in the completed  and rational cases.

 Since, by Corollary \ref{Langer}, there is a quasi-isomorphism
$$
\gamma_r^{-1}:\quad \R\Gamma_{\crr}(U,\overline{U},\so/\sj^{[r]})_{{\mathbf Q}}\stackrel{\sim}{\to}\R\Gamma_{\dr}(U,\overline{U}_{K})/F^r,
$$ we have a particularly nice canonical description of rational log-syntomic cohomology
\begin{align*}
\R\Gamma_{\synt}(U,\overline{U},r)_{{\mathbf Q}} \stackrel{\sim}{\to}
[\R\Gamma_{\crr}(U,\overline{U})_{{\mathbf Q}}\verylomapr{(1-\phi_r,\gamma_r^{-1})} \R\Gamma_{\crr}(U,\overline{U})_{{\mathbf Q}}\oplus \R\Gamma_{\dr}(U,\overline{U}_{K})/F^r)],
\end{align*}
where square brackets stand for mapping fiber. 
\begin{remark}
In the above definition one can replace the map $1-\phi_r$ with any polynomial map $P\in 1+XK[X]$ to obtain the analog of Besser's finite polynomial cohomology. This was studied in \cite{BLZ}.
\end{remark}
  For arithmetic pairs  $(U,\overline{U})$ that are  log-smooth over $V^{\times}$ and of Cartier type this  can be simplified further  by  using Hyodo-Kato complexes (cf. Proposition \ref{reduction1} below). 
To do that, consider the  following sequence of maps of homotopy limits. Homotopy limits are taken in the $\infty$-derived category (to do that we define the maps $\iota_{\pi}$ by the zigzag from diagram (\ref{infinity-category})). We will describe the coherence data only if they are nonobvious. 
\begin{align*}
\R\Gamma_{\synt}(U,\overline{U},r)_{\mathbf Q}
 & \stackrel{\sim}{\to}
   \xymatrix@C=36pt{[\R\Gamma_{\crr}(U,\overline{U})_{\mathbf Q}\ar[r]^-{(1-\phi_r,\gamma_r^{-1})} & \R\Gamma_{\crr}(U,\overline{U})_{\mathbf Q}\oplus  \R\Gamma_{\dr}(U,\overline{U}_{K})/F^r ]}\\
    &  \stackrel{\sim}{\to}
  \left[\begin{aligned}{\xymatrix{\R\Gamma_{\crr}((U,\overline{U})/R)_{\mathbf Q}\ar[rr]^-{(1-\phi_r,i^*_{\pi}\gamma_r^{-1})}\ar[d]^{N}  && \R\Gamma_{\crr}((U,\overline{U})/R)_{\mathbf Q}\oplus \R\Gamma_{\dr}(U,\overline{U}_{K})/F^r \ar[d]^{(N,0)}\\
\R\Gamma_{\crr}((U,\overline{U})/R)_{\mathbf Q}\ar[rr]^-{1-\phi_{r-1}}  && \R\Gamma_{\crr}((U,\overline{U})/R)_{\mathbf Q}}}\end{aligned}\right]\\
&  \stackrel{\iota_{\pi}}{\leftarrow}
   \left[\begin{aligned}\xymatrix@C=40pt{\R\Gamma_{\hk}(U,\overline{U})_{\mathbf Q}\ar[r]^-{(1-\phi_r,\iota_{\dr,\pi})}\ar[d]^{N}  & \R\Gamma_{\hk}(U,\overline{U})_{\mathbf Q}
   \oplus \R\Gamma_{\dr}(U,\overline{U}_{K}) /F^r\ar[d]^{(N,0)}\\
\R\Gamma_{\hk}(U,\overline{U})_{\mathbf Q}\ar[r]^{1-\phi_{r-1}}  & \R\Gamma_{\hk}(U,\overline{U})_{\mathbf Q}}\end{aligned}\right]
\end{align*}
The first map was described above. The second one   is induced by the distinguished triangle  
$$\R\Gamma_{\crr}(U,\overline{U})\to \R\Gamma_{\crr}((U,\overline{U})/R)\stackrel{N}{\to}\R\Gamma_{\crr}((U,\overline{U})/R)$$ The third one - by the section
$\iota_{\pi}: \R\Gamma_{\hk}(U,\overline{U})_{\mathbf Q}\to \R\Gamma_{\crr}((U,\overline{U})/R)_{\mathbf Q}$ (notice that $\iota_{\dr,\pi}=\gamma_{r}^{-1}i^*_{\pi}\iota_{\pi}$). We will show below that the third map  is a quasi-isomorphism. 

   Set $C_{\st}(\R\Gamma_{\hk}(U,\overline{U})\{r\})$ equal to the last homotopy limit in the above diagram.
\begin{proposition}
\label{reduction1} Let $(U,\overline{U})$ be an arithmetic  pair that is  log-smooth over $V^{\times}$ and of Cartier type. Let $r\geq 0$. Then the above  diagram defines a canonical quasi-isomorphism.
$$\alpha_{\synt,\pi}:\quad \R\Gamma_{\synt}(U,\overline{U},r)_{\mathbf Q}\stackrel{\sim}{\to} C_{\st}(\R\Gamma_{\hk}(U,\overline{U})\{r\}).
$$
\end{proposition}
\begin{proof}We need to show that the map $\iota_{\pi}$ in the above diagram is a quasi-isomorphism. 
Define complexes ($r\geq -1$)
\begin{align*}
\R\Gamma_{\crr}((U,\overline{U})/R,r):= & \Cone(\R\Gamma_{\crr}((U,\overline{U})/R)_{{\mathbf Q}}\stackrel{1-\phi_r}{\longrightarrow}\R\Gamma_{\crr}((U,\overline{U})/R)_{{\mathbf Q}})[-1],\\
\R\Gamma_{\hk}(U,\overline{U},r):= & \Cone(\R\Gamma_{\hk}(U,\overline{U})_{{\mathbf Q}}\stackrel{1-\phi_r}{\longrightarrow}\R\Gamma_{\hk}(U,\overline{U})_{{\mathbf Q}})[-1]
\end{align*}
 It suffices to prove that the following 
maps
\begin{equation}
\label{reduction}
i^*_0: \R\Gamma_{\crr}((U,\overline{U})/R,r)  \stackrel{\sim}{\to} \R\Gamma_{\hk}(U,\overline{U},r),\quad
\iota_{\pi}: \R\Gamma_{\hk}(U,\overline{U},r)  \stackrel{\sim}{\to} \R\Gamma_{\crr}((U,\overline{U})/R,r)
\end{equation}
are quasi-isomorphisms.  Since $i^*_0\iota_{\pi}=\id$, it suffices to show that the map $i^*_0$ is a quasi-isomorphism. Base-changing to $W(\overline{k})$, we may assume that the residue field of $V$ is algebraically closed.
It suffices to show that, for $i\geq 0$, $t\geq -1$,  in the commutative diagram
$$
\begin{CD}
H^i_{\hk}(U,\overline{U})_{{\mathbf Q}}@>p^t-\phi >> H^i_{\hk}(U,\overline{U})_{{\mathbf Q}}\\
@AA i^*_0 A @AA i^*_0 A\\
H^i_{\crr}((U,\overline{U})/R)_{{\mathbf Q}}@>p^t-\phi >>H^i_{\crr}((U,\overline{U})/R)_{{\mathbf Q}}
\end{CD}
$$
the vertical maps induce isomorphisms between the kernels and cokernels of the horizontal maps.

  Since the   $W(k)$-linear map $\iota_{\pi}$ commutes with $\phi$ and its $R$-linear extension is a quasi-isomorphism
  $$\iota_{\pi}: R\otimes_{W(k)}\R\Gamma_{\hk}(U,\overline{U})_{{\mathbf Q}}\stackrel{\sim}{\to} \R\Gamma_{\crr}((U,\overline{U})/R) _{{\mathbf Q}}$$ it suffices to show that
  in the following commutative diagram
  $$
\begin{CD}
H^i_{\hk}(U,\overline{U})_{{\mathbf Q}}@>p^t-\phi >> H^i_{\hk}(U,\overline{U})_{{\mathbf Q}}\\
@AA i_0\otimes\id  A @AA i_0 \otimes\id A \\
R\otimes_{W(k)}H^i_{\hk}(U,\overline{U})_{{\mathbf Q}}@>p^t-\phi >> R\otimes_{W(k)}H^i_{\hk}(U,\overline{U})_{{\mathbf Q}}\end{CD}
$$
the vertical maps induce isomorphisms between the kernels and cokernels of the horizontal maps.  This will follow if we show that the following map
$$I\otimes_{W(k)}H^i_{\hk}(U,\overline{U})_{{\mathbf Q}}\stackrel{p^t-\phi}{\longrightarrow} I\otimes_{W(k)}H^i_{\hk}(U,\overline{U})_{{\mathbf Q}},
$$
for $I\subset R$ - the kernel of the projection $i_0:R_{\mathbf Q}\to K_0$, $t_l\mapsto 0$, is an isomorphism.  We argue as Langer in \cite[p. 210]{Ln}. Let $M:=H^i_{\hk}(U,\overline{U})/tor$.
It is a lattice in $H^i_{\hk}(U,\overline{U})_{{\mathbf Q}}$ that is stable under Frobenius. Consider the formal inverse $\psi:=\sum_{n\geq 0}(p^{-t}\phi)^n$ of $1-p^{-t}\phi$. It suffices to show that, for $y\in I\otimes_{W(k)}M$, $\psi(y)\in I\otimes_{W(k)}M$.
Fix $l$ and let  $T^{\{k\}}:=t_l^k/\lfloor k/e_K\rfloor !$.  We will show that, for any $m\in M$, $\psi(T^{\{k\}}\otimes m)\in I\otimes_{W(k)}M$ and the infinite  series converges uniformly in $k$. We have
$$(p^{-t}\phi)^n(T^{\{k\}}\otimes m )=\frac{\lfloor kp^n/e_K\rfloor !}{\lfloor k/e_K\rfloor !p^{tn}}T^{\{kp^n\}}\otimes m\pri $$
and $\ord_p(\lfloor kp^n/e_K\rfloor !/\lfloor k/e_K\rfloor !)\geq p^{n-1}$. Hence $\frac{\lfloor kp^n/e_K\rfloor !}{\lfloor k/e_K\rfloor !p^{tn}}$  converges $p$-adically to zero, uniformly in $k$, as wanted.
\end{proof}
\begin{remark}
It was Langer \cite[p.193]{Ln} (cf. \cite[Lemma 2.13]{JS} in the good reduction case) who observed the fact that while, in general, the crystalline cohomology $\R\Gamma_{\crr}(U,\ove{U})$ behaves badly (it is "huge"), after taking "filtered Frobenius eigenspaces" we obtain syntomic cohomology $\R\Gamma_{\synt}(U,\ove{U},r)_{\bq}$ that behaves well (it is "small"). In \cite[3.5]{JB}  this phenomenon is explained by relating syntomic cohomology to the complex $C_{\st}(\R\Gamma_{\hk}(U,\ove{U})\{r\})$.
\end{remark}
\begin{remark}
\label{reduction21}
The construction of the map $\alpha_{\synt,\pi}$ depends on the choice of the uniformizer $\pi$, which  makes the $h$-sheafification impossible. We will show now that there is a functorial and compatible with base change quasi-isomorphism $\alpha^{\prime}_{\synt}$ between rational syntomic cohomology and certain complexes built from Hyodo-Kato cohomology and de Rham cohomology that $h$-sheafify well.

Set
\begin{align*}
\alpha^{\prime}_{\synt}:\quad \R\Gamma_{\synt}(U,\overline{U},r)_{\mathbf Q}
 & \stackrel{\sim}{\to}
   [\R\Gamma_{\crr}(U,\overline{U},r)\lomapr{\gamma_r^{-1}}   \R\Gamma_{\dr}(U,\overline{U}_{K})/F^r ]\\
&  \stackrel{\beta}{\to}
   [\R\Gamma_{\hk}(U,\overline{U},r)^{N=0}\lomapr{\iota^{\prime}_{\dr}}  \R\Gamma_{\dr}(U,\overline{U}_{K}) /F^r]
\end{align*}
Here the two morphisms $\beta$ and $\iota^{\prime}_{\dr}$ are defined as the following compositions
\begin{align*}
   \beta:\quad &  \R\Gamma_{\crr}(U,\overline{U},r)\stackrel{\sim}{\to}   \R\Gamma_{\crr}(U_0,\overline{U}_0,r)\stackrel{\sim}{\to}      \R\Gamma_{\hk}(U,\overline{U},r)^{N=0}\\
  \iota^{\prime}_{\dr}:\quad &   \R\Gamma_{\hk}(U,\overline{U},r)^{N=0}\stackrel{\beta}{\leftarrow}  \R\Gamma_{\crr}(U,\overline{U},r)\stackrel{\gamma_r^{-1}}{\to}\R\Gamma_{\dr}(U,\overline{U}_{K}),
\end{align*}
where $(\cdots )^{N=0}$ denotes the mapping fiber of the monodromy. 
The map $\beta$ is a quasi-isomorphism because so is each of the intermediate maps. To see this for the map $i^*_0: \R\Gamma_{\crr}(U,\overline{U},r)\to   \R\Gamma_{\crr}(U_0,\overline{U}_0,r)$, consider the following  factorization 
$$
F^m: \R\Gamma_{\crr}(U,\overline{U},r)\stackrel{i^*_0}{\to}  \R\Gamma_{\crr}(U_0,\overline{U}_0,r)
\stackrel{\psi_m}{\to}\R\Gamma_{\crr}(U,\overline{U},r)
$$ of the $m$'th power of the Frobenius, where $m$ is large enough. We also have $ i^*_0\psi_m=F^m$. Since Frobenius is a quasi-isomorphism on $\R\Gamma_{\crr}(U,\overline{U},r)$ and $\R\Gamma_{\crr}(U_0,\overline{U}_0,r)$ both $i^*_0$ and $\psi_m$ are quasi-isomorphisms as well.  The second morphism in the sequence defining $\beta$ is a quasi-isomorphism by an argument similar to the one we used in the proof of Proposition \ref{reduction1}.

 Define the complex  $$  C^{\prime}_{\st}(\R\Gamma_{\hk}(U,\overline{U})\{r\}):= [\R\Gamma_{\hk}(U,\overline{U},r)^{N=0}\lomapr{\iota^{\prime}_{\dr}}  \R\Gamma_{\dr}(U,\overline{U}_{K}) /F^r].
 $$ We have obtained a quasi-isomorphism
 $$\alpha^{\prime}_{\synt}:\quad \R\Gamma_{\synt}(U,\overline{U},r)_{\mathbf Q}\stackrel{\sim}{\to} C^{\prime}_{\st}(\R\Gamma_{\hk}(U,\overline{U})\{r\})
 $$
 It is clearly functorial but it is also easy to check that it is compatible with base change (of the base $V$). 
\end{remark}

  Define the complex  
  $$  C_{\st}(\R\Gamma^B_{\hk}(U,\overline{U})\{r\}):= [\R\Gamma^B_{\hk}(U_1,\overline{U}_1,r)^{N=0}\lomapr{\iota^{B}_{\dr}}  \R\Gamma_{\dr}(U,\overline{U}_{K}) /F^r].
$$
From the commutative diagram  (\ref{Beilinson=HK})
we obtain the natural quasi-isomorphisms
\begin{align*}
\gamma:  & \quad   C_{\st}(\R\Gamma^B_{\hk}(U,\overline{U})\{r\})   \stackrel{\sim}{\to}C_{\st}(\R\Gamma_{\hk}(U,\overline{U})\{r\})\\
\alpha_{\synt,\pi}^B:=\gamma^{-1}\alpha_{\synt,\pi}:   & \quad \R\Gamma_{\synt}(U,\overline{U},r)_{\mathbf Q}  \stackrel{\sim}{\to}C_{\st}(\R\Gamma^B_{\hk}(U,\overline{U})\{r\})
\end{align*}

  We will show now that log-syntomic cohomology satisfies finite Galois descent.
Let $(U,\overline{U})$ be a fine log-scheme, log-smooth over $V^{\times}$, and of Cartier type. Let $r\geq 0$. Let $K^{\prime}$ be a finite Galois extension of $K$ and let $G=\Gal(K\pri/K)$. Let
$(T,\overline{T})=(U\times_{V}{V^{\prime}},\overline{U}\times _{V}{V^{\prime}})$, $V^{\prime}$ - the ring of integers in $K^{\prime}$,  be the base change of $(U,\overline{U})$ to $(K\pri,V\pri)$, and let $f: (T,\overline{T})\to (U,\overline{U})$ be the canonical projection. 
Take $R=R_{V}$,  $N$, $e$, $\pi$ associated to $V$.   Similarly, we define $R\pri:=R_{V\pri}$, $N\pri$, $e\pri$, $\pi\pri$. 
Write the map  $\alpha^B_{\synt,\pi}$ as 
$$
\xymatrix{
\R\Gamma_{\synt}(U,\overline{U},r)_{\mathbf Q}\ar[d]^{\alpha^B_{\synt,\pi}}_{\wr}\ar[r]^-{\sim}_-h & 
[\R\Gamma^{B,\tau}_{\hk}((U,\overline{U})_{R},r)^{N=0}\ar[r]^{i^*_{\pi}} & \R\Gamma_{\dr}(U,\overline{U}_{K}) /F^r]\\
C_{\st}(\R\Gamma^B_{\hk}(U,\overline{U}) \{r\})\ar[r]^{\sim} & [\R\Gamma^{B}_{\hk}(U,\overline{U},r)^{N=0}\ar[r]^{\iota^B_{\dr}}\ar[u]_{\wr}^{\iota_{\pi}\beta} & \R\Gamma_{\dr}(U,\overline{U}_{K}) /F^r]\ar@{=}[u]
}
$$
Here we defined the map  $h$  as the composition
\begin{equation}
\label{h}
\R\Gamma_{\synt}(U,\overline{U},r)_{\mathbf Q}\to \R\Gamma_{\crr}((U,\overline{U})/R)_\bq\stackrel{\sim}{\leftarrow} \R\Gamma^B_{\hk}(U_1,\overline{U}_1)_{R_\bq}^{\tau}
\end{equation}

  From the construction of the Beilinson-Hyodo-Kato map $\iota_{\dr}^B: \R\Gamma^{B}_{\hk}(T_1,\overline{T}_1)\to \R\Gamma_{\dr}(T,\overline{T}_{K^{\prime}}) 
$ it follows that it is $G$-equivariant; hence the complex $C_{\st}(\R\Gamma^B_{\hk}(T,\overline{T})\{r\})$ is equipped with a natural $G$-action. We claim that the map $\alpha^B_{\synt,\pi^\prime}$ induces a natural map 
\begin{align*}
  \tilde{\alpha}^B_{\synt,\pi^\prime}: & \quad  \R\Gamma(G,\R\Gamma_{\synt}(T,\overline{T},r)_{{\mathbf Q}})\to 
   \R\Gamma(G,C_{\st}(\R\Gamma^B_{\hk}(T,\overline{T})\{r\})),\\
     \tilde{\alpha}^B_{\synt,\pi^\prime} & :=  (1/|G|)\sum_{g\in G}\alpha^B_{\synt,g(\pi^\prime)}
   \end{align*}
   To see this it suffices to show that, for every $g\in G$, we have a commutative diagram
   $$
 \begin{CD}
   \R\Gamma_{\synt}(T,\overline{T},r)_{{\mathbf Q}}@>\alpha^B_{\synt,\pi^\prime}>>
   C_{\st}(\R\Gamma^B_{\hk}(T,\overline{T})\{r\})\\
   @VV g^* V @VV g^* V\\
     \R\Gamma_{\synt}(T,\overline{T},r)_{{\mathbf Q}}@>\alpha^B_{\synt,g(\pi^\prime)}>>
   C_{\st}(\R\Gamma^B_{\hk}(T,\overline{T})\{r\})
    \end{CD}
      $$ 
      We accomplish this by constructing natural morphisms 
 \begin{align*}
 g^*: \quad & \R\Gamma_{\crr}((T,\overline{T})/R_{\pi\pri}\pri)\to \R\Gamma_{\crr}((T,\overline{T})/R_{g(\pi\pri)}\pri),\\
  g^*: \quad & \R\Gamma^B_{\hk}(T_1,\overline{T}_1)_{R^\prime_{\pi^{\prime}}}^{\tau}\to \R\Gamma^B_{\hk}(T_1,\overline{T}_1)_{R^\prime_{g(\pi^{\prime})}}^{\tau}
     \end{align*}
  that are compatible with the maps in (\ref{h}) that define $h$, the maps  $\iota_?$ and $i^*_?$, and the trivialization $\beta$. We define the pullbacks $g^*$ from a map  $g:R\pri_{\pi\pri}\to R\pri_{g(\pi\pri)}$ constructed by lifting the action of $g$ from $V_1^\prime$ to $R\pri$ by setting $g(t\pri_{\pi\pri} )=t\pri_{g(\pi\pri)}$ and taking the induced action of $g$ on $W(k\pri)$.  This map is compatible with Frobenius and monodromy. The induced pullbacks $g^*$  are clearly compatible with the map $i^*_0$ and the maps $\iota_?$, the maps $i^*_{\pi\pri}$, $i^*_{g(\pi\pri)}$, and  the trivialization $\beta$.   From the construction of the Beilinson-Hyodo-Kato map, the pullbacks $g^*$ are also compatible with the maps $\kappa_{R^\prime_?}$; hence with the map $h$, as wanted.

  \begin{proposition}
\label{hypercov11}
\begin{enumerate}
\item The following diagram commutes in the (classical) derived category.
$$
\begin{CD}
\R\Gamma_{\synt}(U,\overline{U},r)_{{\mathbf Q}} @> f^*>>  \R\Gamma(G,\R\Gamma_{\synt}(T,\overline{T},r)_{{\mathbf Q}})\\
@VV\alpha^B_{\synt,\pi}V @VV\tilde{\alpha}^B_{\synt,\pi^{\prime}}V\\
C_{\st}(\R\Gamma^B_{\hk}(U,\overline{U}) \{r\}) @> f^* >> \R\Gamma(G,C_{\st}(\R\Gamma^B_{\hk}(T,\overline{T})\{r\}))
\end{CD}
$$
\item  The natural 
map
$$f^*: \R\Gamma_{\synt}(U,\overline{U},r)_{{\mathbf Q}} \stackrel{\sim}{\to} \R\Gamma(G,\R\Gamma_{\synt}(T,\overline{T},r)_{{\mathbf Q}})
$$
is  a quasi-isomorphism.\end{enumerate}
\end{proposition}
\begin{proof} 
The second claim of the proposition follows from the first one and the fact that the Hyodo-Kato and de Rham cohomologies satisfy finite Galois decent.

  Since everything in sight is functorial and satisfies finite unramified Galois descent we may assume that the extension $K\pri/K$ is totally ramified. 
First,  we will construct a $G$-equivariant (for the trivial action of $G$ on $R$) map
    $$
   f^*:  \R\Gamma_{\crr}((U,\overline{U})/R,r)^{N=0} \to \R\Gamma_{\crr}((T,\overline{T})/R\pri,r)^{N\pri=0}   
    $$ 
   such   that   the following diagram commutes
\begin{equation}
\label{diag1}
\begin{CD}
\R\Gamma_{\crr}(U,\overline{U},r)@> f^*>>  \R\Gamma_{\crr}(T,\overline{T},r)\\
@VV\wr V @VV\wr V\\
\R\Gamma_{\crr}((U,\overline{U})/R,r)^{N=0} @> f^*>>  \R\Gamma_{\crr}((T,\overline{T})/R\pri,r)^{N\pri=0}\\
@A\wr A\iota_{\pi} A @A\wr A\iota_{\pi^{\prime}} A\\
\R\Gamma_{\hk}(U,\overline{U},r)^{N=0}@> f^* > \sim >  \R\Gamma_{\hk}(T,\overline{T},r)^{N\pri=0}
\end{CD}
\end{equation}
\begin{remark}
Note that the bottom map is an isomorphism because $f^*$ acts trivially on the Hyodo-Kato complexes.
 The commutativity of the above diagram and  the quasi-isomorphisms (\ref{reduction}) will imply that a totally ramified Galois extension does not change the log-crystalline complexes $\R\Gamma_{\crr}(U,\overline{U},r)$ and 
$\R\Gamma_{\crr}((U,\overline{U})/R,r)^{N=0}$. 
\end{remark}

 Let $e_1$ be  the ramification index of $V\pri/V$. 
 Set   $v=(\pi\pri)^{e_1}\pi^{-1}$, and choose an integer $s$
such that $(\pi\pri)^{p^s}\in pV\pri$. Set $T:=t_{\pi}, T\pri:=t_{\pi\pri}$ and define the morphism $a: R\to R\pri $ by
$T\mapsto (T\pri)^{e_1}[\overline{v}]^{-1}$. Since $V\pri_1$ and $V_1$ are defined by $pR+T^{e}R$ and by $pR\pri+(T\pri)^{e\pri}R\pri$, respectively, $a$ induces a morphism
$a_1: V_{1}\to V\pri_1$. We have $F^sa_1=F^sf_1$, where $F$ is the absolute Frobenius on $\Spec(V_{1})$. Notice that in general $f_1\neq a_1$ if $v[\overline{v}]^{-1}\ncong 1 \mod pV\pri$.
The morphism $\phi_R^sa: \Spec(R\pri)\to \Spec(R)$ is compatible with $F^sf_1:\Spec(V\pri_1)\to\Spec(V_{1})$ and it commutes with the operators $N$ and $p^sN\pri$.
  We have the following commutative diagram
$$
\xymatrix{
(T,\overline{T})_1\ar[rr]^-{F^sf_1} \ar[d] && (U,\overline{U})_1\ar[d]\\
\Spec(V^{\prime}_1)\ar[rr] ^-{F^sa_1=F^sf_1}\ar[d]&& \Spec(V_{1})\ar[d]\\
\Spec(R^{\prime}) \ar[rr] ^-{\phi_R^sa}&& \Spec(R)
}
$$
 Hence the commutative diagram of distinguished triangles
\begin{equation}
\label{brrrr}
\begin{CD}
\R\Gamma_{\crr}(U,\overline{U})_\bq @>>> \R\Gamma_{\crr}((U,\overline{U})/R)_\bq@> eN >> \R\Gamma_{\crr}((U,\overline{U})/R)_\bq\\
@VV f^*F^s V @VV f^*F^sV @VV p^se_1f^*F^sV \\
 \R\Gamma_{\crr}(T,\overline{T})_\bq@>>> \R\Gamma_{\crr}((T,\overline{T})/R\pri)_\bq@> e^{\prime}N\pri >> \R\Gamma_{\crr}((T,\overline{T})/R\pri)_\bq.
\end{CD}
\end{equation}

    To see how this diagram arises we may assume (by the usual \v{C}ech argument) that we have a fine affine log-scheme $X_n/V^{\times}_n$  that is log-smooth over $V^{\times}_n$. We can also assume that we have a lifting of $X_n\hookrightarrow Z_n$ over $\Spec(W_n(k)[T])$ (with the log-structure coming from $T$) and a lifting of Frobenius $\phi_Z$ on $Z_n$ that is compatible with the Frobenius $\phi_R$.  Recall \cite[Lemma 4.2]{Kas} that the horizontal distinguished triangles in the above diagram arise from an exact sequence of complexes of sheaves on $X_{n,\eet}$
\begin{equation}
\label{monodromyC}
0\to C_V\pri[-1]\verylomapr{\wedge \dlog T} C_V\to C_V\pri \to 0
\end{equation}
where $C_V:=R_n\otimes _{W_n(k)[T]}\Omega\kr_{Z_n/W_n(k)}$ and $C_V\pri:=R_n\otimes _{W_n(k)[T]}\Omega\kr_{Z_n/W_n(k)[T]}$. Now consider the base change of $Z_n/W_n(k)[T]$ by the map $F^sa:\Spec(W_n(k)[T\pri])\to\Spec(W_n(k)[T])$ and the related complexes (\ref{monodromyC}). We get a commutative diagram of complexes of sheaves on $X_{n,\eet}$ (note that $X_{V\pri,n,\eet}=X_{n,\eet}$)
$$
\begin{CD}
0@>>> C_{V\pri}\pri[-1] @>\wedge\dlog T\pri >>C_{V\pri} @>>> C_{V\pri }\pri@>>>  0\\
@. @AA p^se_1a^*\phi_Z^sA @AA a^*\phi_Z^s A @AA a^*\phi_Z^s A @.\\
0 @>>> C_V\pri[-1] @> \wedge\dlog T >>C_V @>>> C_V\pri @>>> 0
\end{CD}
$$
Hence diagram (\ref{brrrr}).

Combining diagram (\ref{brrrr}) with Frobenius we obtain  the following commutative diagram
$$
\begin{CD}
\R\Gamma_{\crr}(U,\overline{U},r) @< F^s << \R\Gamma_{\crr}(U,\overline{U},r)@> f^*F^s >> \R\Gamma_{\crr}(T,\overline{T},r)\\
@VV\wr V @VV\wr V @VV\wr V\\
\R\Gamma_{\crr}((U,\overline{U})/R,r)^{N=0} @<(F^s, p^sF^s) <<\R\Gamma_{\crr}((U,\overline{U})/R,r)^{N=0} @> (a^*F^s, p^sa^*F^s) >> \R\Gamma_{\crr}((T,\overline{T})/R\pri,r)^{N\pri=0}\\
@V\wr V i_0^* V @V\wr V i_0^* V @V\wr V i_0^* V\\
\R\Gamma_{\hk}(U,\overline{U},r)^{N=0} @< (F^s,p^sF^s) <\sim <\R\Gamma_{\hk}(U,\overline{U},r)^{N=0} @> (F^s,p^sF^s) >\sim > \R\Gamma_{\hk}(T,\overline{T},r)^{N\pri=0}
\end{CD}$$
It follows that all the maps in the above diagram are quasi-isomorphisms. We define the map $$
f^*: \R\Gamma_{\crr}((U,\overline{U})/R,r)^{N=0}\to \R\Gamma_{\crr}((T,\overline{T})/R\pri,r)^{N\pri=0}
$$
 by the middle row. Since, for any $g\in G$, we have ${v}_{g(\pi\pri)}=g({v}_{\pi\pri})$, the map $f^*$ is $G$-equivariant. In the (classical) derived category, this definition is independent of the constant $s$ we have chosen. Since $i^*_0$ is a quasi-isomorphism and $i^*_0\iota_{?*}=\id$, the diagram (\ref{diag1}) commutes as well, as wanted.  

  We define the map 
\begin{equation}
\label{BHK}
f^*: \R\Gamma^{B,\tau}_{\hk}((U,\overline{U})/R,r)^{N=0}\to \R\Gamma^{B,\tau}_{\hk}((T,\overline{T})/R\pri,r)^{N\pri=0}
\end{equation}
in an analogous way. By the above diagram and by compatibility of the Beilinson-Hyodo-Kato constructions with base change and with Frobenius, the two pullback maps $f^*$ are compatible via the morphism $h$, i.e.,  the following diagram commutes
$$
\xymatrix{
\R\Gamma_{\crr}(U,\overline{U},r) \ar[r]\ar[d]^{f^*} &
\R\Gamma_{\crr}((U,\overline{U})/R,r)^{N=0}\ar[d]^{f^*}  & \R\Gamma^{B,\tau}_{\hk}((U,\overline{U})/R,r)^{N=0} \ar[l]_{\kappa_R}\ar[d]^{f^*} \\
\R\Gamma_{\crr}(T,\overline{T},r)\ar[r] & \R\Gamma_{\crr}((T,\overline{T})/R\pri,r)^{N\pri=0} & \R\Gamma^{B,\tau}_{\hk}((T,\overline{T})/R\pri,r)^{N\pri=0}\ar[l]_{\kappa_{R^\prime} }
}
 $$
 From the analog of  diagram (\ref{diag1}) for the Beilinson-Hyodo-Kato complexes and by the universal nature of the trivialization at $\overline{p}$ we obtain that the pullback map $f^*$ is compatible with the maps $\beta\iota_?$.
It remains to show that we have a commutative diagram
$$
\xymatrix{
\R\Gamma^B_{\hk}(U,\overline{U},r)^{N=0}\ar[r]^{f^*}_{\sim}\ar[d]^{\iota^B_{\dr}}  & \R\Gamma^B_{\hk}(T,\overline{T},r)^{N^\prime=0}\ar[d]^{\iota^B_{\dr}}\\
\R\Gamma_{\dr}(U,\overline{U}_{K})/F^r \ar[r]^{f^*} & \R\Gamma_{\dr}(T,\overline{T}_{K^{\prime}})/F^r
}
$$
But this follows since  the Beilinson-Hyodo-Kato map is compatible with base change.
\end{proof}
\subsection{Arithmetic syntomic cohomology}
We are now ready to introduce and study arithmetic syntomic cohomology, i.e., syntomic cohomology over $K$. 
 Let $\sj^{[r]}_{\crr}$, $\sa_{\crr}$, and $\sss(r)$ for $r\geq 0$ be the $h$-sheafifications on $\mathcal{V}ar_{K}$ of the presheaves sending $(U,\overline{U})\in \spp^{ss}_{K}$ to $
\R\Gamma_{\crr}(U,\overline{U},J^{[r]})$,  $
\R\Gamma_{\crr}(U,\overline{U})$, and $\R\Gamma_{\synt}(U,\overline{U},r)$, respectively.  Let $\sj^{[r]}_{\crr,n}$, $\sa_{\crr,n}$, and $\sss_n(r)$ denote the $h$-sheafifications of
the mod-$p^n$ versions of the respective presheaves. We have
$$
\sss_n(r)\simeq \Cone(\sj^{[r]}_{\crr,n}\stackrel{p^r-\phi}{\longrightarrow}\sa_{\crr,n})[-1],\quad \sss(r)\simeq \Cone(\sj^{[r]}_{\crr}\stackrel{p^r-\phi}{\longrightarrow}\sa_{\crr})[-1] .
$$
For $r\geq 0$, define $\sss(r)_{\mathbf Q}$ as the $h$-sheafification of the presheaf sending ss-pairs $(U,\overline{U})$ to $\R\Gamma_{\synt}(U,\overline{U},r)_{\mathbf Q}$. We have
$$\sss(r)_{\mathbf Q}\simeq \Cone(\sj^{[r]}_{\crr,{\mathbf Q}}\lomapr{1-\phi_r}\sa_{\crr,{\mathbf Q}})[-1]
$$

 For $X\in \mathcal{V}ar_{K}$, set $\R\Gamma_{\synt}(X_h,r)_n=\R\Gamma(X_h,\sss_n(r))$, $\R\Gamma_{\synt}(X_h,r):=\R\Gamma(X_h,\sss(r)_{\mathbf Q})$. We have
\begin{align*}
\R\Gamma_{\synt}(X_h,r)_n & \simeq \Cone(\R\Gamma(X_h,\sj^{[r]}_{\crr,n})\stackrel{p^r-\phi}{\longrightarrow}\R\Gamma(X_h,\sa_{\crr,n}))[-1],\\
 \R\Gamma_{\synt}(X_h,r) & \simeq \Cone(\R\Gamma(X_h,\sj^{[r]}_{\crr,{\mathbf Q}})\stackrel{1-\phi_r}{\longrightarrow}\R\Gamma(X_h,\sa_{\crr,{\mathbf Q}}))[-1] .
\end{align*}
We will  often write  $\R\Gamma_{\crr}(X_h)$ for $\R\Gamma(X_h,\sa_{\crr})$ if this does not cause confusion.

  Let $\sa_{\hk}$ be the $h$-sheafification of the presheaf $(U,\overline{U})\mapsto \R\Gamma_{\hk}(U,\overline{U})_{\mathbf Q}$ on $\spp_{K}^{ss}$; this is an $h$-sheaf of $E_{\infty}$ $ K_0$-algebras on ${\mathcal V}ar_{K}$ equipped with a $\phi$-action and a derivation $N$ such that $N\phi=p\phi N$. For $X\in {\mathcal V}ar_{K}$, set $\R\Gamma_{\hk}(X_h):=\R\Gamma(X_h,\sa_{\hk})$.  Similarly, we define $h$-sheaves $\sa^B_{\hk}$ and the complexes
  $\R\Gamma^B_{\hk}(X_h):=\R\Gamma(X_h,\sa^B_{\hk})$.  The maps $\kappa: \R\Gamma^B_{\hk}(U_1,\overline{U}_1)\to \R\Gamma_{\hk}(U,\overline{U})_{\mathbf Q}$ $h$-sheafify and we obtain  functorial quasi-isomorphisms 
  \begin{align*}
  \kappa: \quad \sa^B_{\hk}  \stackrel{\sim}{\to} \sa_{\hk},\quad 
  \kappa:\quad \R\Gamma^B_{\hk}(X_h)  \stackrel{\sim}{\to } \R\Gamma_{\hk}(X_h).
 \end{align*}
 \begin{remark}
 The complexes $\sj^{[r]}_{\crr,n}$ and $\sss_n(r)$ (and their completions) have a concrete description. For the complexes $ \sj^{[r]}_{\crr,n}$: we can represent the presheaves $(U,\ove{U})\mapsto \R\Gamma_{\crr}(U,\ove{U},\sj^{[r]}_n)$ by Godement resolutions (on the crystalline site), sheafify them for the $h$-topology on $\spp^{ss}_K$, and then move them to ${\mathcal V}ar_K$. For the complexes $\sss_n(r)$: the maps $p^r-\phi$ can be lifted to the Godement resolutions and their mapping fiber (defining $\sss_n(r)(U,\ove{U})$) can be computed in the abelian category of complexes of abelian groups. To get $\sss_n(r)$ we $h$-sheafify on $\spp^{ss}_K$ and pass to ${\mathcal V}ar$.
\end{remark}

   Let, for a moment, $K$ be any field of characteristic zero. Consider the presheaf $(U,\overline{U})\mapsto \R\Gamma_{\dr}(U,\overline{U}):=\R\Gamma(\overline{U},\Omega^{\scriptscriptstyle\bullet}_{(U,\overline{U})})$ of filtered dg $K$-algebras on $\spp^{nc}_K$. Let $\sa_{\dr}$ be its $h$-sheafification. It is a sheaf of filtered $K$-algebras on $\mathcal{V}ar_K$. For $X\in \mathcal{V}ar_K$, we have Deligne's de Rham complex of $X$ equipped with Deligne's Hodge filtration: $\R\Gamma_{\dr}(X_h):=\R\Gamma(X_h,\sa_{\dr})$. Beilinson proves the following comparison statement.
 \begin{proposition} (\cite[2.4]{BE1}) 
 \label{deRham1}
 \begin{enumerate}
 \item For $(U,\overline{U})\in \spp^{nc}_K$, the canonical map $\R\Gamma_{\dr}(U,\overline{U})\stackrel{\sim}{\to} \R\Gamma_{\dr}(U_h)$ is a filtered quasi-isomorphism.
 \item The cohomology groups $H^i_{\dr}(X_h):=H^i\R\Gamma_{\dr}(X_h)$ are $K$-vector spaces of dimension equal to the rank of $H^i(X_{\overline{K},\eet},{\mathbf Q}_p)$.
 \end{enumerate}
 \end{proposition}
 \begin{corollary}
 \label{blowup}
 For a geometric pair $(U,\overline{U})$ over $K$ that is saturated and log-smooth, the canonical map $$\R\Gamma_{\dr}(U,\overline{U})\stackrel{\sim}{\to} \R\Gamma_{\dr}(U_h)$$ is a filtered quasi-isomorphism. 
 \end{corollary}
 \begin{proof}
 Recall \cite[Theorem 5.10]{NR} that there is a log-blow-up $(U,\overline{T})\to (U,\overline{U})$ that resolves singularities of $(U,\overline{U})$, i.e., such that $(U,\ove{T}) \in \spp^{\nc}_K$. We have a commutative diagram
$$
\xymatrix{
\R\Gamma_{\dr}(U,\overline{T})\ar[r]^{\sim} &\R\Gamma_{\dr}(U_h)\\
\R\Gamma_{\dr}(U,\overline{U})\ar[u]^{\wr}\ar[ur]&
}
$$
The vertical map is a filtered quasi-isomorphism; the horizontal map is a filtered quasi-isomorphism by the above proposition. Our corollary follows.
 \end{proof}
 \begin{remark}
 Another proof of the above result (and a mild generalization) that does not use resolution of singularities can be found in \cite[1.19]{BE2} (where it is attributed to A.Ogus).
 \end{remark}
  Return now to our $p$-adic field $K$. 
 \begin{remark}
 \label{descent}
By construction the complexes $\R\Gamma(X_h,\sj^{[r]}_{\crr,{\mathbf Q}})$, $\R\Gamma_{\synt}(X_h,r)$, $\R\Gamma_{\hk}(X_h)$, $\R\Gamma^B_{\hk}(X_h)$, and $\R\Gamma_{\dr}(X_h)$ satisfy $h$-descent. In particular, since $h$-topology is finer than the \'etale topology, they satisfy Galois descent for finite extensions. Hence, 
 for any finite Galois extension $K_1/K$, the natural maps
$$
\R\Gamma^*_{?}(X_h)\stackrel{\sim}{\to}\R\Gamma (G,\R\Gamma^*_{?}(X_{K_1,h})),\quad ?={\crr}, {\synt},\hk, \dr; \quad *=B,\emptyset
$$
where $G=\Gal(K_1/K)$, are (filtered) quasi-isomorphisms.  Since $G$ is finite, it follows that the natural maps
$$\R\Gamma^*_{\hk}(X_h)\otimes_{K_0}K_{1,0}\stackrel{\sim}{\to}\R\Gamma^*_{\hk}(X_{K_1,h}),\quad \R\Gamma_{\dr}(X_h)\otimes_{K}K_1\stackrel{\sim}{\to}\R\Gamma_{\dr}(X_{K_1,h})
$$ 
are (filtered) quasi-isomorphisms as well.  
\end{remark}

  Recall from \cite[2.5]{BE2} and
% Proposition
(\ref{isomorphism}) below that for a fine log-scheme $X$, log-smooth over $V^{\times}$, and of Cartier type we have a quasi-isomorphism $\R\Gamma_{\crr}(X_{\overline{V}},\sj^{[r]}_{X_{\overline{V}}/W(k)})_{\mathbf Q}\simeq \R\Gamma(X_{\ovk,h},\sj^{[r]}_{\crr})_{\mathbf Q}$. We can descend this result to $K$ but on the level of rational log-syntomic cohomology; the key observation being that the field extensions introduced by the alterations are harmless since, by Proposition \ref{hypercov11},  log-syntomic cohomology satisfies finite Galois descent. Along the way we will get an analogous comparison quasi-isomorphism for the Hyodo-Kato cohomology. 
  
\begin{proposition}
\label{hypercov}
For any arithmetic pair $(U,\overline{U})$ that is fine, log-smooth over $V^{\times}$, and of Cartier type, and $r\geq 0$, the canonical maps
$$ \R\Gamma_{\hk}^*(U,\overline{U})_{\mathbf Q} \stackrel{\sim}{\to} \R\Gamma_{\hk}^*(U_h),\quad \R\Gamma_{\synt}(U,\overline{U},r)_{\mathbf Q} \stackrel{\sim}{\to} \R\Gamma_{\synt}(U_h,r)
$$
are quasi-isomorphisms.
\end{proposition}
\begin{proof}
It suffices to show that for any $h$-hypercovering $(U_{\scriptscriptstyle\bullet},\overline{U}_{\scriptscriptstyle\bullet})\to (U,\overline{U})$ by pairs from $\spp^{\log}_{K}$ the natural maps $$
\R\Gamma_{\hk}(U,\overline{U})_{{\mathbf Q}} \to \R\Gamma_{\hk}(U_{\scriptscriptstyle\bullet},\overline{U}_{\scriptscriptstyle\bullet})_{{\mathbf Q}},\quad \R\Gamma_{\synt}(U,\overline{U},r)_{{\mathbf Q}} \to \R\Gamma_{\synt}(U_{\scriptscriptstyle\bullet},\overline{U}_{\scriptscriptstyle\bullet},r)_{{\mathbf Q}}
$$ are (modulo taking a refinement of $(U_{\scriptscriptstyle\bullet},\overline{U}_{\scriptscriptstyle\bullet})$)  quasi-isomorphisms.
  For the second map, since we have a canonical quasi-isomorphism
  $$\R\Gamma_{\synt}(U,\overline{U},r)_{\mathbf Q}\stackrel{\sim}{\to}
\Cone(\R\Gamma_{\crr}(U,\overline{U},r)_{\mathbf Q}\to  \R\Gamma_{\crr}(U,\overline{U},\so/\sj^{[r]})_{\mathbf Q})[-1]
  $$
  it suffices to show that, up to a refinement of the hypercovering, we have quasi-isomorphisms
 $$\R\Gamma_{\crr}(U,\overline{U},\so/\sj^{[r]})_{\mathbf Q}\stackrel{\sim}{\to}
  \R\Gamma_{\crr}(U_{\scriptscriptstyle\bullet},\overline{U}_{\scriptscriptstyle\bullet},\so/\sj^{[r]})_{\mathbf Q},\quad \R\Gamma_{\crr}(U,\overline{U},r)_{\mathbf Q}\stackrel{\sim}{\to}\R\Gamma_{\crr}(U_{\scriptscriptstyle\bullet},\overline{U}_{\scriptscriptstyle\bullet},r)_{\mathbf Q} .
 $$
 For the first of these  maps, by Corollary \ref{Langer}  this amounts to showing that the following map is a quasi-isomorphism
$$\R\Gamma(\overline{U}_K,\Omega^{\scriptscriptstyle\bullet}_{(U,\overline{U}_K)})/F^r \stackrel{\sim}{\to}
\R\Gamma(\overline{U}_{\scriptscriptstyle\bullet,K},\Omega^{\scriptscriptstyle\bullet}_{(U_{\scriptscriptstyle\bullet},\overline{U}_{\scriptscriptstyle\bullet,K})})/F^r.
$$
But, by Corollary \ref{blowup} this map is quasi-isomorphic to the map
$$\R\Gamma_{\dr}(U_h)/F^r\to \R\Gamma_{\dr}(U_{\scriptscriptstyle\bullet,h})/F^r,
$$
which is clearly a quasi-isomorphism.

 Hence it suffices to show that, up to a refinement of the hypercovering, we have quasi-isomorphisms
 $$\R\Gamma_{\hk}(U,\overline{U})_{\mathbf Q}\stackrel{\sim}{\to}
  \R\Gamma_{\hk}(U_{\scriptscriptstyle\bullet},\overline{U}_{\scriptscriptstyle\bullet})_{\mathbf Q},\quad \R\Gamma_{\crr}(U,\overline{U},r)_{\mathbf Q}\stackrel{\sim}{\to}\R\Gamma_{\crr}(U_{\scriptscriptstyle\bullet},\overline{U}_{\scriptscriptstyle\bullet},r)_{\mathbf Q} .
 $$
 Fix  $t\geq 0$. To show that $H^t\R\Gamma_{\crr}(U,\overline{U},r)_{\mathbf Q}\stackrel{\sim}{\to}H^t\R\Gamma_{\crr}(U_{\scriptscriptstyle\bullet},\overline{U}_{\scriptscriptstyle\bullet},r)_{\mathbf Q}
$ is a quasi-isomorphism  we will often work with $(t+1)$-truncated $h$-hypercovers. This is
because  $\tau_{\leq t}\R\Gamma_{\crr}(U_{\scriptscriptstyle\bullet},\overline{U}_{\scriptscriptstyle\bullet},r)\simeq
\tau_{\leq t}\R\Gamma_{\crr}((U_{\scriptscriptstyle\bullet},\overline{U}_{\scriptscriptstyle\bullet})_{\leq t+1},r)$, where $(U_{\scriptscriptstyle\bullet},\overline{U}_{\scriptscriptstyle\bullet})_{\leq t+1}$ denotes the $t+1$-truncation.
 Assume first that we have an  $h$-hypercovering $(U_{\scriptscriptstyle\bullet},\overline{U}_{\scriptscriptstyle\bullet})\to (U,\overline{U})$  of arithmetic pairs over $K$, where  each pair $(U_i,\overline{U}_i)$, $i\leq t+1$, is log-smooth  over $V^{\times}$ and of Cartier type.   We claim that then already the maps
\begin{equation}
\label{firsteq}
\tau_{\leq t}\R\Gamma_{\hk}(U,\overline{U})_{\mathbf Q}\stackrel{\sim}{\to}
\tau_{\leq t}\R\Gamma_{\hk}((U_{\scriptscriptstyle\bullet},\overline{U}_{\scriptscriptstyle\bullet})_{\leq t+1})_{\mathbf Q};\quad 
\tau_{\leq t}\R\Gamma_{\crr}(U,\overline{U})_{\mathbf Q}\stackrel{\sim}{\to}
\tau_{\leq t}\R\Gamma_{\crr}((U_{\scriptscriptstyle\bullet},\overline{U}_{\scriptscriptstyle\bullet})_{\leq t+1})_{\mathbf Q}
\end{equation}
are quasi-isomorphisms. To see the second quasi-isomorphism consider the following commutative diagram of distinguished triangles ($R=R_V$)
$$
\begin{CD}
\R\Gamma_{\crr}(U,\overline{U}) @>>> \R\Gamma_{\crr}((U,\overline{U})/R)@> N >> \R\Gamma_{\crr}((U,\overline{U})/R)\\
@VVV @VVV @VVV \\
 \R\Gamma_{\crr}((U_{\scriptscriptstyle\bullet},\overline{U}_{\scriptscriptstyle\bullet})_{\leq t+1})@>>> \R\Gamma_{\crr}((U_{\scriptscriptstyle\bullet},\overline{U}_{\scriptscriptstyle\bullet})_{\leq t+1}/R)@> N >> \R\Gamma_{\crr}((U_{\scriptscriptstyle\bullet},\overline{U}_{\scriptscriptstyle\bullet})_{\leq t+1}/R)
\end{CD}
$$
It suffices to show that
the two right vertical arrows are rational quasi-isomorphisms in degrees less
than or equal to $t$. But we have the $R$-linear quasi-isomorphisms
$$\iota: R\otimes_{W(k)}\R\Gamma_{\hk}(U,\overline{U})_{{\mathbf Q}}\stackrel{\sim}{\to}\R\Gamma((U,\overline{U})/R)_{{\mathbf Q}}, \quad
\iota:  R\otimes_{W(k)}\R\Gamma_{\hk}((U_{\scriptscriptstyle\bullet},\overline{U}_{\scriptscriptstyle\bullet})_{\leq t+1})_{{\mathbf Q}}\stackrel{\sim}{\to}
\R\Gamma((U_{\scriptscriptstyle\bullet},\overline{U}_{\scriptscriptstyle\bullet})_{\leq t+1}/R)_{ {\mathbf Q}}.
$$
Hence to show both quasi-isomorphisms (\ref{firsteq}), it suffices to show that the map
 $$\tau_{\leq t}\R\Gamma_{\hk}(U,\overline{U})_{{\mathbf Q}}\to
\tau_{\leq t}\R\Gamma_{\hk}((U_{\scriptscriptstyle\bullet},\overline{U}_{\scriptscriptstyle\bullet})_{\leq t+1})_{{\mathbf Q}}
 $$
 is a quasi-isomorphism. 
 
   Tensoring over $K_0$ with $K$ and using the Hyodo-Kato quasi-isomorphism (\ref{HKqis}) we reduce to showing that
 the map  $$
 \tau_{\leq t}\R\Gamma(\overline{U}_K,\Omega\kr_{(U,\overline{U}_K)})\to \tau_{\leq t}
 \R\Gamma(\overline{U}_{\scriptscriptstyle\bullet K, \leq t+1},\Omega\kr_{(U_{\scriptscriptstyle\bullet},\overline{U}_{\scriptscriptstyle\bullet,K})_{\leq t+1}})
$$ is a quasi-isomorphism. And this we have done above.

  To treat the general case, set $X=(U,\overline{U}), Y_{\scriptscriptstyle\bullet}=(U_{\scriptscriptstyle\bullet},\overline{U}_{\scriptscriptstyle\bullet})$. We will do a base change to reduce to the case discussed above.
  We may assume that all the fields $K_{n,i}$, $K_{U_n}\simeq \prod K_{n,i}$ are Galois over $K$.
  Choose a finite Galois extension $(V\pri,K\pri)/(V,K)$ for $K\pri$ Galois over all the fields $K_{n,i}$, $n\leq t+1$.
  Write  $N_X(X_{V\pri})$ for the "\v{C}ech nerve" of $X_{V\pri}/X$. The term  $N_X(X_{V\pri}) _n$ is defined as the  $(n+1)$-fold fiber product of $X_{V\pri}$ over $X$: $N_X(X_{V\pri}) _n=(U\times_{K}K^{\prime,n+1},(\overline{U}\times_{V}V^{\prime,n+1})^{\norm})$, where $V^{\prime,n+1},K^{\prime,n+1}$ are defined as the $(n+1)$-fold product of $V\pri$ over $V$ and of $K\pri$ over $K$, respectively. Normalization is taken with respect to the open regular subscheme $U\times_{K}K^{\prime,n+1}$. Note that
  $N_X(X_{V\pri}) _n\simeq (U\times_{K}K\pri \times  G^{n},\overline{U}\times_{V}V\pri \times G^{n})$, $G=\Gal(K\pri/K)$. Hence it is a log-smooth scheme over $V^{\prime,\times}$,
  of Cartier type. The augmentation $N_X(X_{V\pri})\to X$ is an $h$-hypercovering.

   Consider  the bi-simplicial scheme
 $Y_{\scriptscriptstyle\bullet}\times_{X}N_X(X_{V\pri})_{\scriptscriptstyle\bullet} $,
 \begin{align*}
 (Y_{\scriptscriptstyle\bullet}\times_{X}N_X(X_{V\pri})_{\scriptscriptstyle\bullet}) _{n,m}:=Y_n\times_{X}N_X(X_{V\pri}) _m& \simeq (U_n\times_UU\times_{K}K^{\prime,m+1},(\overline{U}_n\times_{\overline{U}}(\overline{U}\times_{V}V^{\prime,m+1})^{\norm})^{\norm})\\
   & \simeq \coprod_i(U_n\times_{K_{n,i}}K_{n,i}\times_{K}K^{\prime,m+1},\overline{U}_n\times_{V_{n,i}}(V_{n,i}\times_{V}V^{\prime,m+1})^{\norm}).
 \end{align*}
 Hence $(Y_{\scriptscriptstyle\bullet}\times_{X}N_X(X_{V\pri})_{\scriptscriptstyle\bullet}) _{n,m}\in \spp^{\log}_{K}$.
  For $n,m\leq t+1$, we have
  $$ (Y_{\scriptscriptstyle\bullet}\times_{X}N_X(X_{V\pri})_{\scriptscriptstyle\bullet})_{n,m} \simeq \coprod_i(U_n\times_{K_{n,i}}K\pri\times G_{n,i}\times G^{m},
  \overline{U}_n\times_{V_{n,i}}V\pri\times G_{n,i}\times G^{m}),
    $$
    where $G_{n,i}=\Gal(K_{n,i}/K)$.
    It is a log-scheme log-smooth over $V^{\prime, \times}$, of Cartier type.

  Consider now its diagonal $Y_{\scriptscriptstyle\bullet}\times_{X}N_X(X_{V\pri}):=\Delta (Y_{\scriptscriptstyle\bullet}\times_{X}N_X(X_{V\pri})_{\scriptscriptstyle\bullet})$.
    It  is an  $h$-hypercovering of $X$ refining $Y_{\scriptscriptstyle\bullet}$ such that, for $n\leq t+1$,  $(Y_{\scriptscriptstyle\bullet}\times_{X}N_X(X_{V\pri}))_n $ is log-smooth over $V^{\prime, \times}$, of Cartier type.  It suffices to show that the compositions
    \begin{align}
\label{secondeq}
\R\Gamma_{\hk}(X)_{\mathbf Q}  & \stackrel{}{\to}
    \R\Gamma_{\hk}(Y_{\scriptscriptstyle\bullet})_{\mathbf Q} \stackrel{\pr_1^*}{\to}\R\Gamma_{\hk}(Y_{\scriptscriptstyle\bullet}\times_{X}N_X(X_{V\pri}))_{\mathbf Q};\\
   \R\Gamma_{\crr}(X,r)_{\mathbf Q}  & \stackrel{}{\to}
    \R\Gamma_{\crr}(Y_{\scriptscriptstyle\bullet},r)_{\mathbf Q} \stackrel{\pr_1^*}{\to}\R\Gamma_{\crr}(Y_{\scriptscriptstyle\bullet}\times_{X}N_X(X_{V\pri}),r)_{\mathbf Q} \notag
    \end{align}
    are quasi-isomorphisms in degrees less than or equal to $t$. Using the commutative diagram of bi-simplicial schemes
  $$
  \begin{CD}
  Y_{\scriptscriptstyle\bullet}\times_{X}N_X(X_{V\pri}) @>\Delta >> Y_{\scriptscriptstyle\bullet}\times_{X}N_X(X_{V\pri})_{\scriptscriptstyle\bullet} @>\pr_1 >>  Y_{\scriptscriptstyle\bullet}\\
  @. @VV\pr_2 V @VVV\\
   @. N_X(X_{V\pri}) @> f >> X
   \end{CD}
  $$
we can write the second composition as
   $$ \R\Gamma_{\crr}(X,r)_{\mathbf Q} \stackrel{f^*}{\to}
    \R\Gamma_{\crr}(N_X(X_{V\pri}),r)_{\mathbf Q} \stackrel{\pr_2^*}{\to}\R\Gamma_{\crr}(Y_{\scriptscriptstyle\bullet}\times_{X}N_X(X_{V\pri})_{\scriptscriptstyle\bullet},r)_{\mathbf Q}
    \stackrel{\Delta^*}{\to}\R\Gamma_{\crr}(Y_{\scriptscriptstyle\bullet}\times_{X}N_X(X_{V\pri}),r)_{\mathbf Q}
    $$

 We claim that all of these maps are quasi-isomorphisms  in degrees less than or equal to $t$.   The map $\Delta^*$ is a quasi-isomorphism (in all degrees) by  \cite[Prop. 2.5]{Fr}.
 For the second map, fix $n\leq t+1$ and consider  the induced map $\pr_2: (Y_{\scriptscriptstyle\bullet}\times_{X}N_X(X_{V\pri})_{\scriptscriptstyle\bullet}) _{\scriptscriptstyle\bullet,n}\to N_X(X_{V\pri}) _n$. It is an $h$-hypercovering whose $(t+1)$-truncation is built from log-schemes, log-smooth  over $(V\pri,K\pri)$, of Cartier type.  It suffices to show that the induced map
 $\tau_{\leq t}\R\Gamma_{\crr}(N_X(X_{V\pri})_n,r)_{\mathbf Q} \stackrel{\pr_2^*}{\to}\tau_{\leq t}\R\Gamma_{\crr}((Y_{\scriptscriptstyle\bullet}\times_{X}N_X(X_{V\pri}))_{\scriptscriptstyle\bullet,n},r)_{\mathbf Q}$ is a quasi-isomorphism. Since all maps are defined over $K\pri$, this follows from the case considered at the beginning of the proof.

  To prove that the map $f^*:  \R\Gamma_{\crr}(X,r)_{\mathbf Q} \to
    \R\Gamma_{\crr}(N_X(X_{V\pri}),r)_{\mathbf Q} $ is a quasi-isomorphism consider first the case when the extension $V^{\prime}/V$ is unramified. Then     $ \R\Gamma_{\crr}(X_{V^{\prime}}) \simeq  \R\Gamma_{\crr}(X)\otimes_{W(k)}W(k^{\prime})$ and the map $f^*$ is a quasi-isomorphism by finite \'etale descent for crystalline cohomology. 

Assume now that the extension $V^{\prime}/V$ is totally ramified and let $\pi$ and  $\pi^{\prime}$ be uniformizers of $V$ and $V^{\prime}$, respectively. Consider the target of $f^*$ 
as a double complex. To show that $f^*$ is a quasi-isomorphism  it suffices to show that, for each $s\geq 0$, the sequence
 $$0\to H^s\R\Gamma_{\crr}(X,r)_{\mathbf Q} \stackrel{f^*}{\to}
    H^s\R\Gamma_{\crr}(N_X(X_{V\pri})_0,r)_{\mathbf Q} \stackrel{d_0^*}{\to}H^s\R\Gamma_{\crr}(N_X(X_{V\pri})_1,r)_{\mathbf Q}
    \stackrel{d_1^*}{\to}H^s\R\Gamma_{\crr}(N_X(X_{V\pri})_2,r)_{\mathbf Q}\ldots $$
    is exact. Embed it into the following diagram
$$
\xymatrix{
0 \ar[r] & H^s\R\Gamma_{\crr}(X,r)_{\mathbf Q} \ar[d]^{\wr}_{\alpha^B_{\synt,\pi}}\ar[r]^-{f^*}  & H^s\R\Gamma_{\crr}(N_X(X_{V\pri})_0,r)_{\mathbf Q}  \ar[d]^{\wr}_{\tilde{\alpha}^B_{\synt,\pi^{\prime}}}\ar[r]^{d_0^*} & H^s\R\Gamma_{\crr}(N_X(X_{V\pri})_1,r)_{\mathbf Q}\ar[d]^{\wr}_{\tilde{\alpha}^B_{\synt,\pi^{\prime}}} \ar[r] & \\
0 \ar[r] & H^s\R\Gamma^B_{\hk}(X,r)_{\mathbf Q}^{N=0}\ar[r]^-{ f^*} & 
    H^s\R\Gamma^B_{\hk}(N_X(X_{V\pri})_0,r)_{\mathbf Q}^{N^{\prime}=0}  \ar[r]^{d_0^*}
     & H^s\R\Gamma^B_{\hk}(N_X(X_{V\pri})_1,r)_{\mathbf Q}^{N^{\prime}=0}    \ar[r] &
 }
$$ 
Note that,  since all the maps $d_i^*$ are induced from automorphisms of $V\pri/V$, by the proof of Proposition \ref{hypercov11} (take the map $f$ used there to be  a given automorphism $g\in G=\Gal(K^{\prime}/K)$ and $\pi^{\prime}$, $g(\pi^{\prime})$ for the uniformizers of $V^{\prime}$) and the proof of Proposition \ref{reduction1}, we get the vertical maps above that make all the squares commute.

   Hence it suffices to show that the following sequence of Hyodo-Kato cohomology groups is exact:
$$0 \to H^s\R\Gamma_{\hk}(X)_{\mathbf Q} \stackrel{f^*}{\to}
    H^s\R\Gamma_{\hk}(N_X(X_{V\pri})_0)_{\mathbf Q}  \stackrel{ d_0^*}{\to}H^s\R\Gamma_{\hk}(N_X(X_{V\pri})_1)_{\mathbf Q}  \stackrel{ d_1^*}{\to}H^s\R\Gamma_{\hk}(N_X(X_{V\pri})_2)_{\mathbf Q} \to
     $$
     But this sequence is isomorphic to the following sequence
     $$0 \to H^s\R\Gamma_{\hk}(X)_{\mathbf Q} \stackrel{f^*}{\to}
    H^s\R\Gamma_{\hk}(X_{V\pri})_{\mathbf Q}  \stackrel{ d_0^*}{\to}H^s\R\Gamma_{\hk}(X_{V\pri})_{\mathbf Q} \times G \stackrel{ d_1^*}{\to}H^s\R\Gamma_{\hk}(X_{V\pri})_{\mathbf Q} \times G^2\to         $$
  representing the (augmented) $G$-cohomology of $H^s\R\Gamma_{\hk}(X)_{\mathbf Q}$. Since $G$ is finite, this complex is exact in degrees at least $1$. It remains to show that 
$H^0(G,H^s\R\Gamma_{\hk}(X_{V\pri})_\bq)\simeq H^s\R\Gamma_{\hk}(X)_\bq$.
 Since $K\pri/K$ is totally ramified, we have $H^s\R\Gamma_{\hk}(X_{V\pri})\simeq H^s\R\Gamma_{\hk}(X)$. Hence the action of $G$ on $H^s\R\Gamma_{\hk}(X_{V\pri})$ is trivial and we get the right $H^0$ as well. We have proved the second quasi-isomorphism from (\ref{secondeq}). Notice that along the way we have actually proved the first quasi-isomorphism.
 \end{proof}

 For $X\in {\mathcal V}ar_K$, we define a canonical $K_0$-linear map ({\em the Beilinson-Hyodo-Kato morphism})
$$\iota^B_{\dr}: \quad \R\Gamma^B_{\hk}(X_h){\to}\R\Gamma_{\dr}(X_h)
$$
as sheafification of the map $\iota^B_{\dr}: \R\Gamma^B_{\hk}(U_1,\overline{U}_1){\to}\R\Gamma_{\dr}(U,\overline{U}_K)
$. It follows from Proposition \ref{HKdR} that we prove in the next section that the cohomology groups $H^i_{\hk}(X_h):=H^i\R\Gamma^B_{\hk}(X_h)$ are finite rank $K_0$-vector spaces and that they vanish for $i>2\dim X$. This implies the following lemma.
\begin{lemma}
\label{dim}
The syntomic cohomology groups $H^i_{\synt}(X_h,r):=H^i\R\Gamma_{\synt}(X_h,r)$ vanish for $i > 2\dim X+2$.
\end{lemma}
\begin{proof}
The map $\iota_{\dr}^{\prime}:\R\Gamma_{\hk}(U,\overline{U},r)^{N=0}\to \R\Gamma_{\dr}(U,\overline{U}_{K})/F^r$ from Remark \ref{reduction21} sheafifies and so does the quasi-isomorphism $\alpha^{\prime}_{\synt}:
 \R\Gamma_{\synt}(U,\overline{U},r)_\bq\stackrel{\sim}{\to}C^{\prime}_{\st}(\R\Gamma_{\hk}(U,\overline{U})\{r\})$. Hence $\R\Gamma_{\synt}(X_h,r)$ is quasi-isomorphic via 
$\alpha^{\prime}_{\synt}$ to the mapping fiber  $$
C^{\prime}_{\st}(\R\Gamma_{\hk}(X_h)\{r\}):= [\R\Gamma_{\hk}(X_h,r)^{N=0}\lomapr{\iota_{\dr}^{\prime}} \R\Gamma_{\dr}(X_h)/F^r]
$$
 The statement of the lemma follows.
\end{proof}

     For $X\in {\mathcal V}ar_{K}$ and $r\geq 0$, define the complex
 $$C_{\st}(\R\Gamma^B_{\hk}(X_h)\{r\}):=
  \left[ \begin{aligned}\xymatrix{\R\Gamma^B_{\hk}(X_h)\ar[rr]^-{(1-\phi_r,\iota^B_{\dr})}\ar[d]^{N}  & & \R\Gamma^B_{\hk}(X_h)
   \oplus \R\Gamma_{\dr}(X_h) /F^r\ar[d]^{(N,0)}\\
\R\Gamma^B_{\hk}(X_h)\ar[rr]^-{1-\phi_{r-1}}  & & \R\Gamma^B_{\hk}(X_h)}\end{aligned}\right]
$$
\begin{proposition}
\label{reduction2}
For $X\in {\mathcal V}ar_{K}$ and $r\geq 0$, there exists a canonical (in the classical derived category) quasi-isomorphism
$$\alpha_{\synt}: \R\Gamma_{\synt}(X_h,r)\stackrel{\sim}{\to}C_{\st}(\R\Gamma^B_{\hk}(X_h)\{r\}).
$$
Moreover, this morphism is compatible with finite base change (of the field $K$).
\end{proposition}
\begin{proof}
To construct the map 
$\alpha_{\synt}$, take  a number $t\geq 2\dim X+2$ and let $Y_{\scriptscriptstyle\bullet}\to X$, $Y_{\scriptscriptstyle\bullet}=(U_{\scriptscriptstyle\bullet},\overline{U}_{\scriptscriptstyle\bullet})$,  be an $h$-hypercovering of $X$ by ss-pairs over $K$. Choose a finite Galois extension $(V\pri,K\pri)/(V,K)$ and a uniformizer $\pi^\prime$ of $V^\prime$ as in the proof of Proposition \ref{hypercov}.  Keeping the notation from that proof, refine our hypercovering  to the $h$-hypercovering $Y_{\scriptscriptstyle\bullet}\times_V {V^{\prime}}\to X_{K^{\prime}}$. Then the  truncation $(Y_{\scriptscriptstyle\bullet}\times_V {V^{\prime}})_{\leq t+1}$ is  built from log-schemes log-smooth over $V^{\prime,\times}$ and of Cartier type. We have
the following sequence of quasi-isomorphisms
\begin{align*}
\gamma_{\pi^\prime}: \R\Gamma_{\synt}(X_{K\pri,h}) &\stackrel{\sim}{\leftarrow} \tau_{\leq t}\R\Gamma_{\synt}(X_{K\pri,h}) \stackrel{\sim}{\to}\tau_{\leq t}\R\Gamma_{\synt}((U_{\scriptscriptstyle\bullet}\times_K {K^{\prime}})_{\leq t+1,h})\stackrel{\sim}{\leftarrow}
\tau_{\leq t}\R\Gamma_{\synt}((Y_{\scriptscriptstyle\bullet}\times_V {V^{\prime}})_{\leq t+1})_{\mathbf Q}\\
 & \stackrel{\sim}{\to}C_{\st}(\tau_{\leq t}\R\Gamma^B_{\hk}((Y_{\scriptscriptstyle\bullet}\times_V {V^{\prime}})_{\leq t+1})\{r\})\stackrel{\sim}{\to}C_{\st}(\tau_{\leq t}\R\Gamma^B_{\hk}((U_{\scriptscriptstyle\bullet}\times_K {K^{\prime}})_{\leq t+1,h})\{r\})\\
 & \stackrel{\sim}{\leftarrow} C_{\st}(\tau_{\leq t}\R\Gamma^B_{\hk}(X_{K^{\prime},h})\{r\})\stackrel{\sim}{\to}C_{\st}(\R\Gamma^B_{\hk}(X_{K^{\prime},h})\{r\})
 \end{align*}
The first quasi-isomorphism follows from Lemma \ref{dim}. The third   and fifth  quasi-isomorphisms follow from Proposition \ref{hypercov}. The fourth quasi-isomorphism (the map $\tilde{\alpha}_{\synt,\pi^\prime}^B$), since all the log-schemes involved are log-smooth over $V^{\prime,\times}$ and  of Cartier type, follows from Proposition \ref{reduction1}. 

   Now, set $G:=\Gal(K\pri/K)$.  Passing from $\gamma_{\pi^\prime}$ to its $G$-fixed points  we obtain the map 
$$\alpha_{\synt}:=\alpha_{\synt,\pi^\prime}:\quad  \R\Gamma_{\synt}(X_{h})\to C_{\st}(\R\Gamma^B_{\hk}(X_{h})\{r\})$$ as the composition
$$
\R\Gamma_{\synt}(X_{h})\to \R\Gamma_{\synt}(X_{K\pri,h})^G\stackrel{\gamma_{\pi^\prime}}{\to}C_{\st}(\R\Gamma^B_{\hk}(X_{K^{\prime},h})\{r\})^G\stackrel{\sim}{\leftarrow}
C_{\st}(\R\Gamma^B_{\hk}(X_{K,h})\{r\})
$$

 It remains to check that so defined map is independent of all choices.  For that, it suffices to check that, in the above construction,  for a  finite Galois extension $(V_1,K_1)$ of $(V^{\prime},K^{\prime})$, $H=\Gal(K_1/K^\prime)$, the corresponding maps
$\alpha_{\synt,?}: \R\Gamma_{\synt}(X_{h})\to C_{\st}(\R\Gamma^B_{\hk}(X_h)\{r\}) $ are the same in the classical derived category (note that this includes trivial extensions).  Easy diagram chase shows that this amounts to checking that the following  diagram commutes
$$
\xymatrix{
 \R\Gamma_{\synt}((Y_{\scriptscriptstyle\bullet}\times_V {V^{\prime}})_{\leq t+1})_{\mathbf Q}\ar[r]^-{\sim}_-{\alpha_{\synt,\pi^{\prime}}}\ar[d]
 & C_{\st}(\R\Gamma^B_{\hk}((Y_{\scriptscriptstyle\bullet}\times_V {V^{\prime}})_{\leq t+1})\{r\}) \ar[d]\\
  \R\Gamma_{\synt}((Y_{\scriptscriptstyle\bullet}\times_V {V_1})_{\leq t+1})_{\mathbf Q}^H\ar[r]^-{\sim}_-{\alpha_{\synt,\pi_1}}
 & C_{\st}(\R\Gamma^B_{\hk}((Y_{\scriptscriptstyle\bullet}\times_V {V_1})_{\leq t+1})\{r\})^H
 }
$$
 But this we have shown in Proposition \ref{hypercov11}.
 
 For the compatibility with finite base change, consider a finite field extension $L/K$. We can choose in the above a Galois extension $K^{\prime}/K$ that works for both fields. We get the same maps $\gamma_{\pi^{\prime}}$ for both $L$ and $K$. Consider now the following commutative diagram. The top and bottom rows define the maps $\alpha^L_{\synt,\pi^{\prime}}$ and $\alpha^K_{\synt,\pi^{\prime}}$, respectively.
$$
\xymatrix{
\R\Gamma_{\synt}(X_{L,h})\ar[r] &  \R\Gamma_{\synt}(X_{K\pri,h})^{G_L}\ar[r]^-{\gamma_{\pi^\prime}} & C_{\st}(\R\Gamma^B_{\hk}(X_{K^{\prime},h})\{r\})^{G_L} & 
C_{\st}(\R\Gamma^B_{\hk}(X_{L,h})\{r\})\ar[l]_-{\sim}\\
\R\Gamma_{\synt}(X_{h})\ar[u]\ar[r]  & \R\Gamma_{\synt}(X_{K\pri,h})^G\ar[u] \ar[r]^-{\gamma_{\pi^\prime}} & C_{\st}(\R\Gamma^B_{\hk}(X_{K^{\prime},h})\{r\})^G\ar[u] & 
C_{\st}(\R\Gamma^B_{\hk}(X_{K,h})\{r\})\ar[l]_-{\sim}\ar[u]
}
$$
This proves the last claim of our proposition.
\end{proof}
 
  \subsection{Geometric syntomic cohomology}
We will now study geometric syntomic cohomology, i.e., syntomic cohomology over $\ovk$. 
 Most of  the  constructions related to syntomic cohomology over $K$ have their analogs over ${\ovk}$. We will summarize them briefly. For details the reader should consult \cite{Ts}, \cite{BE2}. 

  For $(U,\overline{U})\in \spp^{ss}_{\ovk}$, $r\geq 0$,  we have the absolute crystalline cohomology complexes and their completions
\begin{align*}
\R\Gamma_{\crr}(U,\overline{U},\sj^{[r]})_n: &   =\R\Gamma_{\crr}(\overline{U}_{\eet},\R u_{\overline{U}_n/W_n(k)*}\sj^{[r]}_{\overline{U}_n/W_n(k)}),\quad
\R\Gamma_{\crr}(U,\overline{U},\sj^{[r]}):    =
\holim_n\R\Gamma_{\crr}(U,\overline{U},\sj^{[r]})_n,\\
\R\Gamma_{\crr}(U,\overline{U},\sj^{[r]})_\bq: &    =\R\Gamma_{\crr}(U,\overline{U},\sj^{[r]})\otimes\bq_p
\end{align*}
 By \cite[Theorem 1.18]{BE2}, the complex $\R\Gamma_{\crr}(U,\overline{U})$ is a perfect $A_{\crr}$-complex and $\R\Gamma_{\crr}(U,\overline{U})_n\simeq \R\Gamma_{\crr}(U,\overline{U})\otimes^{L}_{A_{\crr}}{A_{\crr}}/p^n\simeq \R\Gamma_{\crr}(U,\overline{U})\otimes^{L}{\mathbf Z}/p^n$. In general, we have
%might have to use filtered derived categories
$\R\Gamma_{\crr}(U,\overline{U},\sj^{[r]})_n\simeq \R\Gamma_{\crr}(U,\overline{U},\sj^{[r]})\otimes^{L}{\mathbf Z}/p^n$.
Moreover,  $J^{[r]}_{\crr}=\R\Gamma_{\crr}(\Spec(\ovk),\Spec(\overline{V}),\sj^{[r]})$ \cite[1.6.3,1.6.4]{Ts}. The absolute  log-crystalline cohomology
 complexes are filtered $E_{\infty}$ algebras over $A_{\crr,n}$, $A_{\crr}$, or $A_{\crr,\bq}$, respectively. Moreover, the rational ones are filtered commutative dg algebras.

 For $r\geq 0$, the mod-$p^n$, completed, and rational log-syntomic complexes $\R\Gamma_{\synt}(U,\overline{U},r)_n$, $\R\Gamma_{\synt}(U,\overline{U},r)$, and $\R\Gamma_{\synt}(U,\overline{U},r)_\bq$ are defined by analogs of formulas (\ref{log-syntomic}). We have $\R\Gamma_{\synt}(U,\overline{U},r)_n\simeq \R\Gamma_{\synt}(U,\overline{U},r)\otimes^{L}{\mathbf Z}/p^n$.
  Let $\sj^{[r]}_{\crr}$, $\sa_{\crr}$, and $\sss(r)$  be the $h$-sheafifications on $\mathcal{V}ar_{\ovk}$ of the presheaves sending $(U,\overline{U})\in \spp^{ss}_{\ovk}$ to $
\R\Gamma_{\crr}(U,\overline{U},\sj^{[r]})$,  $
\R\Gamma_{\crr}(U,\overline{U})$, and $\R\Gamma_{\synt}(U,\overline{U},r)$, respectively.  Let $\sj^{[r]}_{\crr,n}$, $\sa_{\crr,n}$, and $\sss_n(r)$ denote the $h$-sheafifications of
the mod-$p^n$ versions of the respective presheaves; and let $\sj^{[r]}_{\crr,\bq}$, $\sa_{\crr,\bq}$, $\sss(r)_\bq$ be the $h$-sheafification
of the rational versions of the same presheaves.   

  For $X\in \mathcal{V}ar_{\ovk}$, set $\R\Gamma_{\crr}(X_h):=\R\Gamma(X_h,\sa_{\crr})$. It is a filtered (by $\R\Gamma(X_h,\sj^{[r]}_{\crr})$, $r\geq 0$, ) $E_{\infty}$  $A_{\crr}$-algebra equipped with the Frobenius action $\phi$. The Galois group $G_K$ acts on ${\mathcal V}ar_{\ovk}$ and it acts on $X\mapsto \R\Gamma_{\crr}(X_h)$ by transport of structure. If $X$ is defined over $K$ then $G_K$ acts naturally on $\R\Gamma_{\crr}(X_h)$. 

   For $r\geq 0$, set 
$\R\Gamma_{\synt}(X_h,r)_n=\R\Gamma(X_h,\sss_n(r))$, $\R\Gamma_{\synt}(X_h,r):=\R\Gamma(X_h,\sss(r)_\bq)$. We have 
\begin{align*}
\R\Gamma_{\synt}(X_h,r)_n & \simeq \Cone(\R\Gamma(X_h,\sj^{[r]}_{\crr,n})\stackrel{p^r-\phi}{\longrightarrow}\R\Gamma(X_h,\sa_{\crr,n}))[-1],\\
 \R\Gamma_{\synt}(X_h,r) & \simeq \Cone(\R\Gamma(X_h,\sj^{[r]}_{\crr,{\mathbf Q}})\stackrel{1-\phi_r}{\longrightarrow}\R\Gamma(X_h,\sa_{\crr,{\mathbf Q}}))[-1] .
\end{align*}
The direct sum $\bigoplus_{r\geq 0}\R\Gamma_{\synt}(X_h,r)$ is a graded $E_{\infty}$ algebra over ${\mathbf Z}_p$.

  Let $\overline{f}: Z_1\to \Spec(\overline{V}_1)^{\times}$ be  an integral, quasi-coherent log-scheme.  Suppose  that $\overline{f}$ is the base change of $\overline{f}_L:Z_{L,1}\to \Spec(\so_{L,1})^{\times}$ by $\theta_1: \Spec(\overline{\so_{L,1}})^{\times}\to\Spec(\so_{L,1})^{\times}$, for a finite extension $L/K$. That is, we have a map $\theta_{L,1}: Z_1\to Z_{L,1}$ such that the square $(\overline{f},\overline{f}_L,\theta_1,\theta_{L,1})$ is Cartesian. Assume that $\overline{f}_L$ is log-smooth of Cartier type and that the underlying map of schemes is proper. Such data $(L,Z_1,\theta_{L,1})$ form a directed set $\Sigma_1$ and, for a morphism $(L\pri,Z\pri_1,\theta\pri_{L^\prime,1})\to (L,Z_1,\theta_{L,1})$,
 we   have a canonical base change identification compatible with $\phi$-action \cite[1.18]{BE2}
$$\R\Gamma^B_{\hk}(Z_{L,1})\otimes_{L_0}L\pri_0\stackrel{\sim}{\to}
 \R\Gamma^B_{\hk}(Z\pri_{L^\prime,1}).
$$
These identifications can be made compatible with respect to $L$, so we can set
$$ \R\Gamma^B_{\hk}(Z_1):= \dirlim_{\Sigma_1}\R\Gamma^B_{\hk}(Z_{L,1})
$$
 It is a complex of $(\phi,N)$-modules over $K^{\nr}_0$, functorial with respect to morphisms of $Z_1$. 

  Consider the scheme $E_{\crr}:=\Spec(A_{\crr})$. We have $E_{\crr,1}=\Spec(\overline{V}_1)$ and we equip $E_{\crr,1}$ with the induced log-structure. This log-structure extends uniquely to a log-structure on $E_{\crr,n}$ and the PD-thickening $\Spec (\overline{V})^{\times}_1\hookrightarrow E_{\crr,n}$ is universal over ${\mathbf Z}/p^n$. Set $E_{\crr}:=\Spec(A_{\crr})$ with the limit log-structure.
Since we have \cite[1.18.1]{BE2}
$$\R\Gamma_{\crr}(Z_1)\stackrel{\sim}{\to}\R\Gamma_{\crr}(Z_1/E_{\crr}),
$$
Theorem \ref{Bthm} yields a canonical quasi-isomorphism of $B^+_{\crr}$-complexes (called {\em the crystalline Beilinson-Hyodo-Kato quasi-isomorphism})
$$
\iota^B_{\crr}:\quad  \R\Gamma_{\hk}^B(Z_1)^{\tau}_{B^+_{\crr}}\stackrel{\sim}{\to}\R\Gamma_{\crr}(Z_1)_{\bq}
$$ 
compatible with the action of Frobenius. But we have $$\R\Gamma_{\hk}^B(Z_1)^{\tau}_{B^+_{\crr}}=
(\R\Gamma_{\hk}^B(Z_1)\otimes_{K_0^{\nr}}A_{\crr,\bq}^{\tau})^{N=0}$$
and there is a canonical isomorphism $A_{\crr,\bq}^{\tau}\stackrel{\sim}{\to} B_{\st}^+$ that is compatible with Frobenius and monodromy. This implies that the above quasi-isomorphism amounts to a quasi-isomorphism of $B^+_{\crr}$-complexes
$$
\iota^B_{\crr}:\quad  \R\Gamma_{\hk}^B(Z_1)_{B^+_{\st}}\stackrel{\sim}{\to}\R\Gamma_{\crr}(Z_1)\otimes_{A_{\crr}}^{L}B^+_{\st}
$$
compatible with the action of $\phi$ and $N$. The crystalline Beilinson-Hyodo-Kato map can be canonically trivialized at $[\tilde{p}]$, where $\tilde{p}$ is a sequence of $p^n$'th roots of $p$,: 
\begin{align*}
\beta=\beta_{[\tilde{p}]}: \R\Gamma_{\hk}^B(Z_1)\otimes_{K_0^{\nr}}B^+_{\crr} & 
\stackrel{\sim}{\to}
(\R\Gamma_{\hk}^B(Z_1)\otimes_{K_0^{\nr}}B^+_{\crr}[a([\tilde{p}])])^{N=0}\\
x & \mapsto \exp(N(x)a([\tilde{p}]))
\end{align*}
This trivialization is compatible with Frobenius and monodromy.
  
  Suppose now that $\overline{f}_1:Z_1\to \Spec(\overline{V}_1)^{\times}$ is a reduction mod $p$ of a log-scheme $\overline{f}:Z\to \Spec(\overline{V})^{\times}$. Suppose  that $\overline{f}$ is the base change of $\overline{f}_L:Z_{L}\to \Spec(\so_{L})^{\times}$ by $\theta: \Spec(\overline{\so_{L}})^{\times}\to\Spec(\so_{L})^{\times}$, for a finite extension $L/K$. That is, we have a map $\theta_{L}: Z\to Z_{L}$ such that the square $(\overline{f},\overline{f}_L,\theta,\theta_{L})$ is Cartesian. Assume that $\overline{f}_L$ is log-smooth of Cartier type and that the underlying map of schemes is proper. Such data $(L,Z,\theta_{L})$ form a directed set $\Sigma$ and the reduction mod $p$ map $\Sigma\to \Sigma_1$ is cofinal. The Beilinson-Hyodo-Kato quasi-isomorphisms (\ref{HK1}) are compatible with morphisms in $\Sigma$ and their colimit yields a natural quasi-isomorphism (called again the {\em Beilinson-Hyodo-Kato quasi-isomorphism})
$$
\iota^B_{\dr}:\quad \R\Gamma_{\hk}^B(Z_1)_{\ovk}^{\tau}\stackrel{\sim}{\to} \R\Gamma(Z_{\ovk},\Omega^{\scriptscriptstyle\bullet}_{Z/\ovk}).
$$
The trivializations by  $p$ are also compatible with the maps in $\Sigma$ hence we obtain the  Beilinson-Hyodo-Kato maps
$$
\iota^B_{\dr}:=\iota^B_{\dr}\beta_{p}:\quad \R\Gamma_{\hk}^B(Z_1)\to  \R\Gamma(Z_{\ovk},\Omega^{\scriptscriptstyle\bullet}_{Z/\ovk}).
$$

   For  an ss-pair $(U,\overline{U})$ over $\ovk$, set $\R\Gamma^B_{\hk}(U,\overline{U}):=\R\Gamma^B_{\hk}((U,\overline{U})_1)$. Let $\sa^B_{\hk}$ be $h$-sheafification of the presheaf $(U,\overline{U})\mapsto \R\Gamma^B_{\hk}(U,\overline{U})$ on $\spp^{ss}_{\ovk}$. This is an $h$-sheaf of $E_{\infty}$ $K_0^{\nr}$-algebras equipped with a $\phi$-action and locally nilpotent derivation $N$ such that $N\phi=p\phi N$. For $X\in\mathcal{V}ar_{\ovk}$, set $\R\Gamma^B_{\hk}(X_h):=\R\Gamma(X_h,\sa^B_{\hk}). $
\begin{proposition}
\begin{enumerate}
\item For any $(U,\overline{U})\in \spp^{ss}_{\ovk}$, the canonical maps
\begin{equation}
\label{isomorphism}
\R\Gamma_{\crr}(U,\overline{U},\sj^{[r]})_{\mathbf Q}\stackrel{\sim}{\to} \R\Gamma(U_h,\sj^{[r]}_{\crr})_{\mathbf Q},\quad \R\Gamma^B_{\hk}(U,\overline{U})\stackrel{\sim}{\to}\R\Gamma^B_{\hk}(U_h)
\end{equation}
are quasi-isomorphisms.
\item For every $X\in \mathcal{V}ar_{\ovk}$, the cohomology groups $H^n_{\crr}(X_h):=H^n\R\Gamma_{\crr}(X_h)_{\mathbf Q}$, resp.  $H^n_{\hk}(X_h):= H^n\R\Gamma^B_{\hk}(X_h)$, are free $B^{+}_{\crr}$-modules, resp. $K_0^{\nr}$-modules, of rank equal to the rank of $ H^n(X_{\eet},{\mathbf Q}_p)$.
\end{enumerate}
\end{proposition}
\begin{proof}
 Only the filtered statement in part (1) for $r > 0$ requires argument since the rest has been proven by Beilinson  in \cite[2.4]{BE2}. Take $r >0$. To prove that we have a quasi-isomorphism
$\R\Gamma_{\crr}(U,\overline{U},\sj^{[r]})_{\mathbf Q}\stackrel{\sim}{\to} \R\Gamma(U_h,\sj^{[r]}_{\crr})_{\mathbf Q}$ it suffices to show that the map
$\R\Gamma_{\crr}(U,\overline{U},\so/\sj^{[r]})_{\mathbf Q}\stackrel{}{\to} \R\Gamma(U_h,\sa_{\crr}/\sj^{[r]}_{\crr})_{\mathbf Q}$ is a quasi-isomorphism.
Since, for an ss-pair $(T,\ove{T})$ over $K$,  by Corollary \ref{Langer}, $\R\Gamma_{\crr}(T,\overline{T},\so/\sj^{[r]})_{\mathbf Q}\simeq 
\R\Gamma(\overline{T}_K,\Omega^{\scriptscriptstyle\bullet}_{(T,\overline{T}_K)}/F^r)$
this is equivalent to showing that  the map
$\R\Gamma(\overline{U}_K,\Omega^{\scriptscriptstyle\bullet}_{(U,\overline{U}_K)}/F^r)\to \R\Gamma(U_h,\sa_{\dr}/F^r)$ is a quasi-isomorphism. And this follows from Proposition \ref{deRham1}.
\end{proof}
\begin{proposition}
\label{HKdR} Let  $X\in {\mathcal V}ar_{K}$.   The natural projection $\varepsilon: X_{\ovk,h}\to X_h$ defines pullback maps
 \begin{equation}
 \label{qis11}
 \varepsilon^*: \R\Gamma^B_{\hk}(X_h)\to \R\Gamma^B_{\hk}(X_{\ovk,h})^{G_K},\quad  \varepsilon^*: \R\Gamma_{\dr}(X_h)\to \R\Gamma_{\dr}(X_{\ovk,h})^{G_K}.
 \end{equation}
 These are (filtered) quasi-isomorphisms. 
\end{proposition}
\begin{proof}
 Notice that the action of $G_K$ on $\R\Gamma^B_{\hk}(X_{\overline{K},h})\{r\}$ and $  \R\Gamma_{\dr}(X_{\ovk,h})$ is smooth, i.e., the stabilizer of every element is an open subgroup of $G_K$. We will prove only the first quasi-isomorphism - the proof of the second one being analogous. By Proposition \ref{hypercov}, it suffices to show that for any ss-pair over $K$ the natural map
 $$ \R\Gamma^B_{\hk}(U_1,\overline{U}_1)  \to \R\Gamma^B_{\hk}((U,\overline{U})\otimes _K\ovk)^{G_K}
$$
is a quasi-isomorphism. Passing to a finite  extension of $K_U$, if necessary, we may assume that $(U,\overline{U})$ is log-smooth of Cartier type over a finite Galois extension $K_U$ of $K$. Then $$ \R\Gamma^B_{\hk}((U,\overline{U})\otimes _K\ovk)  \simeq \R\Gamma^B_{\hk}(U_1,\overline{U}_1)\otimes_{K_{U,0}}K_0^{\nr}\times H,\quad H=\Gal(K_U/K).$$
Taking  $G_K$-fixed points of this quasi-isomorphism we obtain the first quasi-isomorphism of (\ref{qis11}), as wanted.
 \end{proof}

   Let $(U,\overline{U})$ be an ss-pair over    $\ovk$. Set
\begin{align*}
\R\Gamma_{\dr}^{\natural}(U,\overline{U}):=  & \R\Gamma(\overline{U}_{\eet},\LL\Omega^{\scriptscriptstyle\bullet,\wedge}_{(U,\overline{U})/W(k)}),\quad
\R\Gamma_{\dr}^{\natural}(U,\overline{U})_n: = \R\Gamma_{\dr}^{\natural}(U,\overline{U})\otimes^{{\mathbb L}}{\mathbf Z}/p^n\simeq
\R\Gamma(\overline{U}_{\eet},\LL\Omega^{\scriptscriptstyle\bullet,\wedge}_{(U,\overline{U})_n/W_n(k)}),\\
\R\Gamma_{\dr}^{\natural}(U,\overline{U})\what{\otimes}{\mathbf Z}_p:= & \holim_n \R\Gamma_{\dr}^{\natural}(U,\overline{U})_n,\quad
\R\Gamma_{\dr}^{\natural}(U,\overline{U})\what{\otimes}{\mathbf Q}_p:= (\R\Gamma_{\dr}^{\natural}(U,\overline{U})\what{\otimes}{\mathbf Z}_p)\otimes {\mathbf Q}.
\end{align*}
These are $F$-filtered $E_{\infty}$ algebras. Take the associated presheaves on $\spp^{ss}_{\overline{K}}$.
Denote by $\sa^{\natural}_{\dr}$, $\sa^{\natural}_{\dr,n},\sa^{\natural}_{\dr}\what{\otimes}{\mathbf Z}_p,\sa^{\natural}_{\dr}\what{\otimes}{\mathbf Q}_p$ their sheafifications in the $h$-topology of $\mathcal{V}ar_{\ovk}$. These are sheaves of $F$-filtered  $E_{\infty}$ algebras (viewed as the projective system of quotients modulo $F^i$).
 Set $A_{\dr}:=\LL\Omega^{\scriptscriptstyle\bullet,\wedge}_{\overline{V}/V}$.  By \cite[Lemma 3.2]{BE1} we have 
 $A_{\dr}=\sa^{\natural}_{\dr}(\Spec(\ovk))=\R\Gamma^{\natural}_{\dr}(\ovk,\overline{V})$. The corresponding $F$-filtered algebras $A_{\dr,n}$, $A_{\dr}\what{\otimes}{\mathbf Z}_p$, $A_{\dr}\what{\otimes}{\mathbf Q}_p$ are acyclic in nonzero degrees and the projections $\cdot/F^{m+1}\to \cdot/F^m$ are surjective. 
 Thus (we set $\lim_F:=\holim_F$)
 \begin{align*}
 A^{\diamond}_{\dr,n} & :=\lim_FA_{\dr,n}=\invlim_{m}H^0(A_{\dr,n}/F^m),  \quad A^{\diamond}_{\dr}:=\lim_F(A_{\dr}\what{\otimes}\bz_p)=\invlim_{m}H^0(A_{\dr}\what{\otimes}\bz_p/F^m)\\
   \lim_FA_{\dr}\what{\otimes}{\mathbf Q}_p & =\invlim_{m}H^0(A_{\dr}\what{\otimes}{\mathbf Q}_p/F^m)=B_{\dr}^+, \quad A_{\dr}\what{\otimes}{\mathbf Q}_p/F^m=B^+_{\dr}/F^m
    \end{align*}
For any $(U,\overline{U})$  over    $\ovk$, the complex $\R\Gamma_{\dr}^{\natural}(U,\overline{U})$ is an $F$-filtered $E_{\infty}$ filtered $A_{\dr}$-algebra hence 
$\lim_F\R\Gamma_{\dr}^{\natural}(U,\overline{U})_n$ is an $A_{\dr,n}^{\diamond}$-algebra, 
$\lim_F(\R\Gamma_{\dr}^{\natural}(U,\overline{U})\what{\otimes}\bq_p)$ is a $B^+_{\dr}$-algebra, etc.    We have canonical morphisms
$$\kappa^{\prime}_{r,n}:\quad \R\Gamma_{\crr}(U,\overline{U})_n\to \R\Gamma_{\crr}(U,\overline{U})_n/F^r\stackrel{\sim}{\to }\R\Gamma_{\dr}^{\natural}(U,\overline{U})_n/F^r
$$
In the case of $(\ovk,\overline{V})$, from Theorem \ref{beilinson}, we get isomorphisms $\kappa^{\prime}_{r,n}=\kappa^{-1}_r:A_{\crr,n}/J^{[r]}\stackrel{\sim}{\to} A_{\dr,n}/F^r$. Hence $A_{\dr}^{\diamond}$ is the completion of $A_{\crr}$ with respect to the $J^{[r]}$-topology.

 For $X\in\mathcal{V}ar_{\ovk}$, set $\R\Gamma_{\dr}^{\natural}(X_h):=\R\Gamma(X_h,\sa_{\dr}^{\natural}). $ Since $A_{\dr,{\mathbf Q}}=\ovk$, for any variety $X$ over $\ovk$, we have a filtered quasi-isomorphism of $\ovk$-algebras \cite[3.2]{BE1}
$
\R\Gamma^{\natural}_{\dr}(X_h)_{\mathbf Q}\stackrel{\sim}{\to} \R\Gamma_{\dr}(X_h)
$
obtained by $h$-sheafification of the quasi-isomorphism
\begin{equation}
\label{deRham11}
\R\Gamma^{\natural}_{\dr}(U,\overline{U})_{\mathbf Q}\stackrel{\sim}{\to} \R\Gamma_{\dr}(U,\ove{U}_\bq).
\end{equation}

  Concerning  the $p$-adic coefficients, we have a quasi-isomorphism 
\begin{equation}
\label{gamma}
\gamma_r: (\R\Gamma_{\dr}(X_{h})\otimes_{\ovk}B^+_{\dr}) /F^r\stackrel{\sim}{\to} \R\Gamma(X_{h},\sa^{\natural}_{\dr}\what{\otimes}{\mathbf Q}_p) /F^r
\end{equation}
To define it, consider, for any ss-pair $(U,\overline{U})$ over $\ovk$, the natural map $\R\Gamma^{\natural}_{\dr}(U,\overline{U})\to \R\Gamma^{\natural}_{\dr}(U,\overline{U})\what{\otimes}{\mathbf Z}_p$. It
  yields, by extension to $A_{\dr}\what{\otimes}{\mathbf Q}_p$ and by the quasi-isomorphism (\ref{deRham11}), a quasi-isomorphism  of $F$-filtered $\ovk$-algebras \cite[3.5]{BE2}
  $$
\gamma: \R\Gamma_{\dr}(U,\overline{U})_{\mathbf Q}\otimes_{\ovk}(A_{\dr}\what{\otimes}{\mathbf Q}_p)\stackrel{\sim}{\to}\R\Gamma^{\natural}_{\dr}(U,\overline{U})\what{\otimes}{\mathbf Q}_p
$$
Its mod $F^r$-version $\gamma_r$ after $h$-sheafification yields the quasi-isomorphism
$$\gamma_r: (\sa_{\dr}\otimes_{\ovk}B_{\dr}^+)/F^r\stackrel{\sim}{\to} \sa_{\dr}^{\natural}\what{\otimes}{\mathbf Q}_p/F^r
$$
Passing to $\R\Gamma(X_h,\scriptscriptstyle\bullet)$ we get the quasi-isomorphism (\ref{gamma}).

For $X\in {\mathcal V}ar_{\ovk}$, we have canonical quasi-isomorphisms
$$\iota^B_{\crr}: \R\Gamma^B_{\hk}(X_h)^{\tau}_{B^+_{\crr}}\stackrel{\sim}{\to}\R\Gamma_{\crr}(X_{h})_{\mathbf Q},\quad
\iota^B_{\dr}: \R\Gamma^B_{\hk}(X_h)^{\tau}_{\ovk}\stackrel{\sim}{\to} \R\Gamma_{\dr}(X_h),
$$
compatible with the $\Gal(\ovk/K)$-action.  Here ${}^{\tau}_{B^+_{\crr}}$ and ${}^{\tau}_{\ovk}$ denote the $h$-sheafification of the crystalline and the de Rham Beilinson-Hyodo-Kato twists \cite[2.5.1]{BE2}. Trivializing the first map at $[\tilde{p}]$ and the second map at $p$ we get the Beilinson-Hyodo-Kato maps
\begin{align*}
\iota^B_{\crr}:=\iota^B_{\crr}\beta_{[\tilde{p}]}:   \R\Gamma^B_{\hk}(X_h)\otimes_{K_0^{\nr}}{B^+_{\crr}}\to \R\Gamma_{\crr}(X_{h})_{\mathbf Q},\quad
\iota_{\dr} :=\iota_{\dr}\beta_p:   \R\Gamma^B_{\hk}(X_h)\to  \R\Gamma_{\dr}(X_h).
\end{align*}

 Using the quasi-isomorphism $\kappa_r^{-1}: \sa_{\crr,{\mathbf Q}} /\sj^{[r]}_{\crr,{\mathbf Q}}\stackrel{\sim}{\to} (\sa^{\natural}_{\dr}\what{\otimes}{\mathbf Q}_p)/F^r$ from Theorem \ref{beilinson}, we obtain  the following  quasi-isomorphisms of complexes of sheaves on $X_{\ovk,h}$ 
\begin{align*}
\sss(r)_{\mathbf Q}  & \stackrel{\sim}{\to}
\xymatrix{[\sj^{[r]}_{\crr,{\mathbf Q}}\ar[r]^-{1-\phi_r} & \sa_{\crr,{\mathbf Q}}]} \stackrel{\sim}{\to}
 \xymatrix{[\sa_{\crr,{\mathbf Q}}\ar[rr]^-{(1-\phi_r,\can)} & & \sa_{\crr,{\mathbf Q}}\oplus  \sa_{\crr,{\mathbf Q}} /\sj^{[r]}_{\crr,{\mathbf Q}}]}\\
   & 
\stackrel{\sim}{\leftarrow}\xymatrix{[\sa_{\crr,{\mathbf Q}}\ar[rr]^-{(1-\phi_r,\kappa_r^{-1})} & & \sa_{\crr,{\mathbf Q}}\oplus  (\sa^{\natural}_{\dr}\what{\otimes}{\mathbf Q}_p)/F^r ]}   
\end{align*}
Applying $\R\Gamma(X_{h},\scriptscriptstyle\bullet)$ and the quasi-isomorphism $\gamma_r^{-1}:\R\Gamma(X_{h},\sa^{\natural}_{\dr}\what{\otimes}{\mathbf Q}_p) /F^r\stackrel{\sim}{\to} (\R\Gamma_{\dr}(X_{h})\otimes_{\ovk}B^+_{\dr}) /F^r$ from (\ref{gamma})  we obtain the following quasi-isomorphisms 
\begin{align}
\label{kwak1}
\R\Gamma_{\synt}(X_{h},r)
 & \stackrel{\sim}{\to}
   \xymatrix{[\R\Gamma_{\crr}(X_{h})_{\mathbf Q}\ar[rr]^-{(1-\phi_r,\kappa_r^{-1})} & & \R\Gamma_{\crr}(X_{h})_{\mathbf Q}\oplus  \R\Gamma(X_{h},\sa^{\natural}_{\dr}\what{\otimes}{\mathbf Q}_p) /F^r]}\\
& \stackrel{\sim}{\to}
   \xymatrix{[\R\Gamma_{\crr}(X_{h})_{\mathbf Q}\ar[rr]^-{(1-\phi_r,\gamma_{r}^{-1}\kappa_r^{-1} )} & & \R\Gamma_{\crr}(X_{h})_{\mathbf Q}\oplus  (\R\Gamma_{\dr}(X_{h})\otimes_{\ovk}B_{\dr}^+) /F^r]}\notag
\end{align}
\begin{corollary}
For any $(U,\overline{U})\in\spp^{ss}_{\ovk}$, the canonical map 
$$\R\Gamma 
_{\synt}(U,\overline{U},r)_\bq\stackrel{\sim}{\to} \R\Gamma_{\synt} (U_h,r)
$$
is a quasi-isomorphism.
\end{corollary}
\begin{proof}
Arguing as above we find quasi-isomorphisms
\begin{align*}
\R\Gamma_{\synt}(U,\overline{U},r)_\bq
 & \stackrel{\sim}{\to}
   \xymatrix{[\R\Gamma_{\crr}(U,\overline{U})_{\mathbf Q}\ar[rr]^-{(1-\phi_r,\kappa_r^{-1})} & & \R\Gamma_{\crr}(U,\overline{U})_{\mathbf Q}\oplus  (\R\Gamma^{\natural}(U,\overline{U})\what{\otimes}{\mathbf Q}_p) /F^r]}\\
& \stackrel{\sim}{\to}
   \xymatrix{[\R\Gamma_{\crr}(U,\overline{U})_{\mathbf Q}\ar[rr]^-{(1-\phi_r,\gamma_{r}^{-1}\kappa_r^{-1})} & & \R\Gamma_{\crr}(U,\overline{U})_{\mathbf Q}\oplus  (\R\Gamma_{\dr}(U,\overline{U})\otimes_{\ovk}B_{\dr}^+) /F^r]}
\end{align*}
Comparing them with quasi-isomorphisms (\ref{kwak1})  we see that it suffices to check that the natural maps
$$
\R\Gamma_{\crr}(U,\overline{U})_{\mathbf Q}  \stackrel{\sim}{\to} \R\Gamma_{\crr}(U_h)_{\mathbf Q},\quad
 \R\Gamma_{\dr}(U,\overline{U}) \stackrel{\sim}{\to}\R\Gamma_{\dr}(U_h),
$$
are (filtered) quasi-isomorphisms. 
But this is known by
the isomorphisms
% Proposition
(\ref{isomorphism}) and Proposition \ref{deRham1}.
\end{proof}

   Consider  the following composition of morphisms
 \begin{align}
\R\Gamma_{\synt}(X_{h},r)
 & \stackrel{\sim}{\to}
  \left[\xymatrix@C=36pt{\R\Gamma_{\crr}(X_{h})_{\mathbf Q}\ar[rr]^-{(1-\phi_r,\gamma^{-1}_r\kappa_r^{-1})} & & \R\Gamma_{\crr}(X_{h})_{\mathbf Q}\oplus (\R\Gamma_{\dr}(X_{h})\otimes_{\ovk}B^+_{\dr}) /F^r}\right]\notag\\
\label{Cccc} &   \stackrel{\sim}{\leftarrow}
   \left[\begin{aligned}\xymatrix@C=50pt{\R\Gamma^B_{\hk}(X_{h})\otimes_{K_0^{\nr}}B_{\st}^+\ar[r]^-{(1-\phi_r,\iota^B_{\dr}\otimes\iota)}\ar[d]^{N}  & \R\Gamma^B_{\hk}(X_{h})\otimes_{K_0^{\nr}}B_{\st}^+\oplus (\R\Gamma_{\dr}(X_{h})\otimes_{\ovk}B^+_{\dr}) /F^r\ar[d]^{(N,0)}\\
\R\Gamma^B_{\hk}(X_{h})\otimes_{K_0^{\nr}}B_{\st}^+\ar[r]^{1-\phi_{r-1}}  & \R\Gamma^B_{\hk}(X_{h})\otimes_{K_0^{\nr}}B_{\st}^+}\end{aligned}\right]
 \end{align}
The second quasi-isomorphism 
 uses  the map $$
 (\R\Gamma^B_{\hk}(X_{h})\otimes_{K_0^{\nr}}B_{\st}^+)^{N=0}=\R\Gamma^B_{\hk}(X_{h})^{\tau}_{B^+_{\crr}}\stackrel{\iota^B_{\crr}}{\to}\R\Gamma_{\crr}(X_{h})_{\mathbf Q}
 $$ (that is compatible with the action of $N$ and $\phi$) and  the following lemma.
\begin{lemma}
\label{HK-compatibility}
The following diagrams commute
 $$ \xymatrix{
 \R\Gamma_{\crr}(X_{h})_{\mathbf Q}\otimes_{B^+_{\crr}}B_{\st}^+\ar[rr]^-{\gamma^{-1}_r\kappa_r^{-1}\otimes \iota} & & (\R\Gamma_{\dr}(X_{h})\otimes_{\ovk}B^+_{\dr}) /F^r
  & \R\Gamma_{\crr}(X_{h})_{\mathbf Q}\otimes_{A_{\crr}}B_{\dr} \ar[r]^{\gamma_{\dr}}_{\sim}  & \R\Gamma_{\dr}(X_{h})\otimes_{\ovk}B_{\dr} 
\\
 \R\Gamma^B_{\hk}(X_{h})\otimes_{K_0^{\nr}}B_{\st}^+\ar[rru]_-{\iota^B_{\dr}\otimes\iota}\ar[u]^{\iota^B_{cr}}_{\wr} & & &\R\Gamma^B_{\hk}(X_{h})\otimes_{K_0^{\nr}}B_{\st}\ar[u]^{\iota^B_{\crr}\otimes\iota}\ar[ur]_{\iota^B_{\dr}\otimes\iota}& 
 }
  $$
  Here $\gamma_{\dr}$ is the map defined by Beilinson in \cite[3.4.1]{BE2}.
  \end{lemma}
\begin{proof} We will start with the left diagram. It suffices to show that it canonically commutes with $X_h$ replaced by  any  ss-pair $\overline {Y}=(U,\ove{U})$ over $\ovk$ -- a base change of an ss-pair $Y$ split over $(V,K)$.
  Proceeding as in Example \ref{crucial}, we obtain  the following  diagram in which all squares but the one in the left bottom clearly commute.  
   $$
   \xymatrix{     \R\Gamma_{\hk}^B(Y_1)^{\tau}_K\ar[d]^{\id\otimes 1} \ar[r]^{\iota^B_{K}} & \R\Gamma_{\crr}(Y_1/V^{\times})_\bq/F^r\ar[d] & \R\Gamma(Y_{\eet},\LL\Omega^{\scriptscriptstyle\bullet,\wedge}_{Y/V^{\times}})\wh{\otimes}\bq_p/F^r \ar[l]_{\kappa_r}^{\sim}\ar[d]  & \R\Gamma_{\dr}(Y_K)/F^r\ar[l]_-{\gamma_r}^-{\sim}\ar[d]\\
     \R\Gamma_{\hk}^B(\ove{Y}_1)^{\tau}_{\ovk}\otimes_{\ovk} B_{\dr}^+\ar[r]^{\iota_{\ovk}^B\otimes \kappa_r}  & \R\Gamma_{\crr}(\overline{Y}_1/V^{\times})_\bq/F^r & \R\Gamma(\overline{Y}_{\eet},\LL\Omega^{\scriptscriptstyle\bullet,\wedge}_{\overline{Y}/V^{\times}})\wh{\otimes}\bq_p/F^r \ar[l]_{\kappa_r}^{\sim}  &  (\R\Gamma_{\dr}(\overline{Y}_K)\otimes _{\ovk}B_{\dr}^+)/F^r\ar[l]_{\gamma_r}^-{\sim}\ar[ld]_-{\gamma_r}^{\sim}\\
     \R\Gamma_{\hk}^B(\ove{Y}_1)^{\tau}_{B^+_{\crr}}\otimes_{B^+_{\crr}}B^+_{\st}\ar[u]^{\delta}\ar[r]^{\iota^B_{\crr}\otimes \kappa_r\iota} & \R\Gamma_{\crr}(\overline{Y}_1/A_{\crr})_\bq/F^r\ar[u]^{\wr} & (\R\Gamma^{\natural}_{\dr}(\overline{Y})\wh{\otimes}\bq_p)/F^r  \ar[l]_{\kappa_r}^{\sim}\ar[u]^{\wr} 
       }
       $$
Here we have $B^+_{\dr}/F^m= (\R\Gamma^{\natural}_{\dr}(\ove{K},\ovv)\wh{\otimes}\bq_p)/F^m$ and the map $\delta$ is defined as the composition   
   $$
    \delta:\quad    \R\Gamma_{\hk}^B(\ove{Y}_1)^{\tau}_{B^+_{\crr}}\otimes_{B^+_{\crr}}B^+_{\st}=  (\R\Gamma_{\hk}^B(\ove{Y}_1)\otimes_{K_0^{\nr}} B^+_{\st})^{N=0}\otimes_{B^+_{\crr}}B^+_{\st}\stackrel{\sim}{\to}  \R\Gamma_{\hk}^B(\ove{Y}_1)\otimes_{K_0^{\nr}}B^+_{\st}\lomapr{\beta_p\otimes \iota}  \R\Gamma_{\hk}^B(\ove{Y}_1)^{\tau}_{\ovk}\otimes_{\ovk} B_{\dr}^+  
   $$
   Recall that for the map $\iota_{\dr}^B:    \R\Gamma_{\hk}^B(Y_1)^{\tau}_K\to \R\Gamma_{\dr}(Y_K)/F^r$ we have $\iota^B_{\dr}=\gamma_r^{-1} \kappa_r^{-1}\iota^B_K$.
   Everything in sight being compatible with change of the ss-pairs $Y$ - more specifically with maps in the directed system $\Sigma$ - if this diagram commutes so does its $\Sigma$ colimit and  the left diagram in the lemma for the pair $(U,\ove{U})$. 

   It remains to show that the left bottom square in the above diagram commutes.    To do that consider the ring $\wh{A}_n$ defined as the PD-envelope of the closed immersion
$$\ovv_1^{\times}\hookrightarrow A_{\crr,n}\times _{W_n(k)}V^{\times}_n
$$
That is, $\wh{A}_n$  is the product of the PD-thickenings $(\ovv_1^{\times}\hookrightarrow A_{\crr,n})$ and $(V^{\times}_1 \hookrightarrow V^{\times}_n)$ over $(W_1(k)\hookrightarrow W_n(k))$.  By \cite[Lemma 1.17]{BE2},  this makes $\overline{V}_1^{\times}\hookrightarrow \wh{A}_{\crr,n}$ into the universal PD-thickening in the  log-crystalline site of
   $\overline{V}_1^{\times}$ over $V_n^{\times}$.  Let $\wh{A}:=\injlim_n\wh{A}_{\crr,n}$ with the limit log-structure. Set $\wh{B}^+_{\crr}:=\wh{A}_{\crr}[1/p]$.

    Using Theorem \ref{Bthm} , we obtain a canonical quasi-isomorphism $$
\iota^B_{\wh{B}^+_{\crr}}: \R\Gamma^B_{\hk}(\overline{Y}_1)^{\tau}_{\wh{B}^+_{\crr}}\stackrel{\sim}{\to} \R\Gamma_{\crr}(\ove{Y}_1/\wh{A}_{\crr})_\bq$$ 
   By construction, we have the  maps  of PD-thickenings
   $$
   \xymatrix{  (V^{\times}_1\hookrightarrow V^{\times})  &  (  \overline{V}^{\times}_1\hookrightarrow \wh{A}_{\crr})\ar[l]_{\pr_1}\ar[r]^{\pr_2} & (\overline{V}^{\times}_1\hookrightarrow A_{\crr})  }
   $$
Consider the following  diagram 
   $$
\xymatrix{ 
 \R\Gamma_{\hk}^B(\overline{Y}_1)^{\tau}_{\wh{B}^+_{\crr}}\ar[ddd]_{\iota^B_{\wh{B}^+_{\crr}}} &   & \R\Gamma_{\hk}^B(\overline{Y}_1)_{\ovk}^{\tau}\otimes_{\ovk} B_{\dr}^+/F^r
\ar[ll]_{\pr^*_1\otimes\pr^*_1\kappa_r}\ar[ddd]^{\iota_{\ovk}^B\otimes \kappa_r}\\
   & \R\Gamma_{\hk}^B(\overline{Y}_1)^{\tau}_{B^+_{\crr}}\ar[d]^{\iota^B_{\crr}} \ar[lu]^{\pr_2^*}\ar[ru]^{\delta}  & \\
  & \R\Gamma_{\crr}(\overline{Y}_1/A_{\crr})_\bq/F^r\ar[ld]_{\pr^*_2} ^{\sim}\ar[rd]^{\sim}  &\\
    \R\Gamma_{\crr}(\overline{Y}_1/\wh{A}_{\crr})_\bq/F^r   & &   \R\Gamma_{\crr}(\overline{Y}_1/V^{\times})_\bq/F^r \ar[ll]^{\pr_1^*}_{\sim}
         }
   $$
   The bottom triangle commutes since $   \R\Gamma_{\crr}(\overline{Y}_1/A_{\crr})=   \R\Gamma_{\crr}(\overline{Y}_1/W(k))$. 
 The pullback maps 
\begin{align*}
\pr_1^*:\quad   &  \R\Gamma_{\crr}(\overline{Y}_1/V^{\times}) \stackrel{\sim}{\to} \R\Gamma_{\crr}(\overline{Y}/\wh{A}_{\crr}),\\
\pr_2^*: \quad  & \R\Gamma_{\crr}(\overline{Y}/A_{\crr})_\bq/F^r\stackrel{\sim}{\to} \R\Gamma_{\crr}(\overline{Y}/\wh{A}_{\crr})_\bq/F^r
\end{align*}
are quasi-isomorphisms. Indeed, in the case of the first pullback this follows from the universal property of $\wh{A}_{\crr}$; in the case of the second one - it follows from the commutativity of the bottom triangle since the right slanted map is a quasi-isomorphism as shown by the first diagram in our proof. 

  The left trapezoid and the big square commute by the definition of the Beilinson-Bloch-Kato maps. To see that the top triangle commutes it suffices to show that for  an element
$$x\in    \R\Gamma_{\hk}^B(\overline{Y}_1)^{\tau}_{B^+_{\crr}}=   (\R\Gamma_{\hk}^B(\overline{Y}_1)\otimes_{K_0^{\nr}}B^+_{\st})^{N=0},\quad x=b\sum_{i\geq 0} N^i(m) a([\tilde{p}])^{[i]}, m\in \R\Gamma^B_{\hk}(\ove{Y}_1),b\in B^+_{\crr},$$
we have $\pr^*_2(x)=\pr^*_1\delta(x)$.
Since $\iota(a([\tilde{p}]))=\log([\tilde{p}]/p)$ \cite[4.2.2]{F1}, we calculate
\begin{align*}
\delta(x)=\delta(b\sum_{i\geq 0} N^i(m) a([\tilde{p}])^{[i]}) & = b\sum_{i\geq 0}(\sum_{j\geq 0} N^{i+j}(m)a(p)^{[j]})\log([\tilde{p}]/p)^{[i]}\\
 & =b\sum_{k\geq 0} N^{k}(m)(a(p)+\log([\tilde{p}]/p))^{[k]}
\end{align*}
Since in $\wh{B}^+_{\crr}$ we have  $[\tilde{p}]=([\tilde{p}]/p)p$ and $[\tilde{p}]/p\in 1+J_{\wh{B}^+_{\crr}}$,  it follows that  $a([\tilde{p}])=\log([\tilde{p}]/p)+a(p)$ and $$
    \pr^*_1\delta(x)=\pr^*_1(b\sum_{k\geq 0} N^{k}(m)(a(p)+\log([\tilde{p}]/p))^{[k]})=    b\sum_{k\geq 0} N^{k}(m)a([\tilde{p}])^{[k]}=\pr_2^*   (b\sum_{k\geq 0} N^{k}(m)a([\tilde{p}])^{[k]})=\pr_2^*(x),
   $$
   as wanted. It follows now that the right trapezoid in the above diagram commutes as well and that so does the left diagram in our lemma.
  
   To check the commutativity of the right diagram, consider the following map obtained from the maps $\kappa^{\prime}_{r,n}$ by passing to $F$-limit
   $$
   \kappa^{\prime}_n: \quad \R\Gamma_{\crr}(\overline{Y})_n\otimes^L_{A_{\crr,n}}A_{\dr,n}\stackrel{\sim}{\to}\invlim_F \R\Gamma_{\crr}(\overline{Y})_n/F^r
   $$
   By \cite[3.6.2]{BE2}, this is a quasi-isomorphism. Beilinson \cite[3.4.1]{BE2} defines the map
   $$
   \gamma_{\dr}:\quad \R\Gamma_{\crr}(\overline{Y})_\bq\otimes_{A_{\crr}}B_{\dr}^+\stackrel{\sim}{\to} \R\Gamma_{\dr}(\overline{Y}_K)\otimes_{\ovk}B^+_{\dr}
   $$ by $B^+_{\dr}$-linearization of the composition $\invlim_r(\gamma_r^{-1}\kappa_r^{-1})\holim_n\kappa^{\prime}_n$. We have 
   $$
   \gamma_{\dr}=\gamma_r^{-1}\kappa^{-1}_r:\R\Gamma_{\crr}(\overline{Y})_\bq\to (\R\Gamma_{\dr}(\overline{Y}_K)\otimes_{\ovk}B^+_{\dr})/F^r   
   $$
   Hence the commutativity of the right diagram follows from that of the left one.
   \end{proof}
Let $C^+(\R\Gamma^B_{\hk}(X_{h})\{r\})$ denote the second homotopy limit in the diagram (\ref{Cccc}); denote by $C(\R\Gamma^B_{\hk}(X_{h})\{r\})$ the complex $C^+(\R\Gamma^B_{\hk}(X_{h})\{r\})$ with all the pluses removed. We have defined a map $\alpha_{\synt}:
\R\Gamma_{\synt}(X_{h},r)\to C^+(\R\Gamma^B_{\hk}(X_{h})\{r\})$ and  proved the following proposition.
\begin{proposition}
There is a functorial $G_K$-equivariant quasi-isomorphism
$$\alpha_{\synt}:\quad \R\Gamma_{\synt}(X_{h},r)=\R\Gamma(X_{h},\sss(r)_{\mathbf Q})\simeq C^+(\R\Gamma^B_{\hk}(X_{h})\{r\}).
$$
\end{proposition}
\begin{corollary}
For  $(U,\overline{U})\in \spp^{ss}_{K}$, we have a long exact sequence
\begin{align*}
\to H^i_{\synt}((U,\overline{U})_{\ovk},r)\to (H^i_{\hk}(U,\overline{U})_{\bq}\otimes_{K_0} B_{\st}^+)^{\phi=p^r,N=0}\to (H^i_{\dr}(U,\overline{U})\otimes_K B_{\dr}^+)/F^r\to H^{i+1}_{\synt}((U,\overline{U})_{\ovk},r)\to
\end{align*}
\end{corollary}
\begin{proof}By diagram (\ref{Cccc}), it suffices to show  that 
\begin{align*}
H^i[\R\Gamma^B_{\hk}((U,\overline{U})_1)\otimes_{K_0}B_{\st}^+]^{\phi=p^r,N=0}
 & \simeq (H^i_{\hk}(U,\overline{U})_{\bq}\otimes_{K_0}B^+_{st})^{\phi=p^r,N=0},\\
H^i(\R\Gamma_{\dr}(U,\overline{U})\otimes_K B_{\dr}^+)/F^r) & \simeq (H^i_{\dr}(U,\overline{U})\otimes_K B_{\dr}^+)/F^r
\end{align*}
The second isomorphism is a consequence of the degeneration of the Hodge-de Rham spectral sequence. Keeping in mind that the Beilinson-Hyodo-Kato complexes $\R\Gamma^B_{\hk}((U,\overline{U})_1)$ are built from $(\phi,N)$-modules,  the first isomorphism  follows  from   the following short exact sequences
 (for a $(\phi, N)$-module $M$) 
\begin{align*}
 0\to & M\otimes_{K_0}B^+_{\crr}\to M\otimes_{K_0}B^+_{\st}\stackrel{N}{\to} M\otimes_{K_0}B^+_{\st}\to 0,\\
0\to & (M\otimes_{K_0}B^+_{\crr})^{\phi=p^r}\to M\otimes_{K_0}B^+_{\crr}\stackrel{1-\phi_r}{\to} M\otimes_{K_0}B^+_{\crr}\to 0.
\end{align*}
The first one follows, by induction on $m$ such that $N^m=0$ on $M$, from the exact sequence (\ref{kwak11}) and the fact that $(M\otimes_{K_0}B^+_{\st})^{N=0}\simeq M\otimes_{K_0}B^+_{\crr}$. The second one follows from
\cite[Remark 2.30]{CN}.
\end{proof}
\section{Relation between syntomic cohomology and \'etale cohomology}
In this section we will study the relationship between syntomic and \'etale cohomology in both the geometric and the arithmetic situation.
\subsection{Geometric case}  We start with the geometric case. In this subsection, we will construct the geometric syntomic period map from syntomic to \'etale cohomology. We will prove that in the torsion case, on the level of $h$-sheaves it is a quasi-isomorphism modulo a universal constant; in the rational case it induces an isomorphism on cohomology groups in a stable range. Finally, we will construct the syntomic descent spectral sequence.

We will first recall the de Rham and Crystalline Poincar\'e Lemmas of Beilinson and Bhatt \cite{BE1}, \cite{BE2}, \cite{BH}.
\begin{theorem}(de Rham Poincar\'e Lemma \cite[3.2]{BE1})
\label{derham}
The maps $A_{\dr}\otimes ^{{L}}{\mathbf Z}/p^n \to \sa_{\dr}^{\natural}\otimes ^{{L}}{\mathbf Z}/p^n $ are filtered quasi-isomorphisms of $h$-sheaves on ${\mathcal V}ar_{\overline{K}}$.
\end{theorem}
\begin{theorem}(Filtered Crystalline Poincar\'e Lemma \cite[2.3]{BE2}, \cite[Theorem 10.14]{BH})
The map $J^{[r]}_{\crr,n}\to \sj^{[r]}_{\crr,n}$ is a quasi-isomorphism of $h$-sheaves on ${\mathcal V}ar_{\ovk}$.
\end{theorem}
\begin{proof}  We have the following map of distinguished triangles
$$
\begin{CD}
J^{[r]} _{\crr,n}@>>> A_{\crr,n} @>>> A_{\crr,n}/J^{[r]}_{\crr,n}  \\
@VVV @VV\wr V @VV\wr V\\
\sj^{[r]}_{\crr,n}@>>> \sa_{\crr,n}
 @>>> \sa_{\crr,n}/\sj^{[r]}_{\crr,n}
\end{CD}
$$
The middle map is a quasi-isomorphism by the Crystalline Poincar\'e Lemma proved in \cite[2.3]{BE2}. Hence it suffices to show that so is the rightmost map. But, by \cite[1.9.2]{BE2}, this map is quasi-isomorphic to the map $A_{\dr,n}/F^r\to \sa^{\natural}_{\dr,n}/F^r$. Since the last map is a quasi-isomorphism by the de Rham Poincar\'e Lemma (\ref{derham}) we are done.
\end{proof}

%Beilinson period morphisms here

We will now recall the definitions of the crystalline, Beilinson-Hyodo-Kato, and de Rham period maps \cite[3.1]{BE2}, \cite[3.5]{BE1}.
Let $X\in {\mathcal V}ar_{\ovk}$. To define the crystalline period map $$\rho_{\crr}: \R\Gamma_{\crr}(X_h)\to \R\Gamma(X_{\eet},{\mathbf Z}_p)\what{\otimes} A_{\crr},$$ consider the natural map
$\alpha_n: \R\Gamma_{\crr}(X_h)\to \R\Gamma(X_h,\sa_{\crr,n})$ and the composition
$$\beta_n: \quad \R\Gamma(X_{\eet},{\mathbf Z}_p(r))\otimes^{L}_{{\mathbf Z}_p}A_{\crr,n}\stackrel{\sim}{\to}\R\Gamma(X_{\eet},A_{\crr,n})
\stackrel{\sim}{\to}\R\Gamma(X_{h},A_{\crr,n})
\stackrel{\sim}{\to} \R\Gamma(X_{h},\sa_{\crr,n}).
$$ 
 Set $\rho_{\crr,n}:=\beta_n^{-1}\alpha_n$ and $\rho_{\crr}:=\holim_n\rho_{\crr,n}$.
  The Hyodo-Kato period map $$\rho_{\hk}:\R\Gamma^B_{\hk}(X_h)^{\tau}_{B^+_{\crr}}\to \R\Gamma(X_{\eet},{\mathbf Q}_p)\otimes B^+_{\crr},\quad  \rho_{\hk}=\rho_{\crr,{\mathbf Q}}\iota^B_{\crr},
$$
 is obtained by composing the map $\rho_{\crr, {\mathbf Q}}$ with the quasi-isomorphism $\iota^B_{\crr}: \R\Gamma^B_{\hk}(X_h)^{\tau}_{B^+_{\crr}}\stackrel{\sim}{\to} \R\Gamma_{\crr}(X_h)_{{\mathbf Q}}$. The maps $\rho_{\crr}, \rho_{\hk}$ are morphisms of $E_{\infty}$ $A_{\crr}$- and $B^+_{\crr}$-algebras equipped with a Frobenius action; they are compatible with the action of the Galois group $G_K$. 
  
  To define the de Rham period map $\rho_{\dr}:\R\Gamma_{dR}(X_h)\otimes_{\ovk}B_{\dr}^+\to \R\Gamma(X_{\eet},{\mathbf Q}_p)\otimes B_{\dr}^+$ consider the compositions
 \begin{align*}
 \alpha:\,  & \R\Gamma_{\dr}(X_h)\stackrel{\sim}{\to}\R\Gamma^{\natural}_{\dr}(X_h)\otimes {\mathbf Q}\to \R\Gamma^{\natural}_{\dr}(X_h)\what{\otimes} {\mathbf Q}_p,\\
\beta:\, & \R\Gamma(X_{\eet},{\mathbf Z})\otimes^{\mathbf L}A_{\dr}\stackrel{\sim}{\to}\R\Gamma(X_{\eet},A_{\dr})\to \R\Gamma(X_h,A_{\dr})\to \R\Gamma(X_h,\sa_{\dr}^{\natural})=\R\Gamma^{\natural}_{\dr}(X_h).
 \end{align*}
 After tensoring the map $\beta$ with ${\mathbf Z}/p^n$ and using the de Rham Poincar\'e Lemma we get a quasi-isomorphism
 $$\beta_n: \R\Gamma(X_{\eet},{\mathbf Z}/p^n)\otimes^{\mathbf L}A_{\dr}\stackrel{\sim}{\to} \R\Gamma^{\natural}_{\dr}(X_h)\otimes^{\mathbf L}{\mathbf Z}/p^n.
   $$
  Set $\beta_{\mathbf Q}:=\holim_n\beta_n\otimes {\mathbf Q}$ and $\rho_{\dr}:=\beta^{-1}\alpha$. This is  a morphism of filtered $E_{\infty}$ $B^+_{\dr}$-algebras, compatible with $G_K$-action.
\begin{theorem}(\cite[3.2]{BE2}, \cite[3.6]{BE1})For  $X\in {\mathcal V}ar_{\ovk}$,   we have canonical quasi-isomorphisms
\begin{align*}
\rho_{\crr}:  \,& \R\Gamma_{\crr}(X_h)\otimes_{A_{\crr}}B_{\crr}\stackrel{\sim}{\to}  \R\Gamma(X_{\eet},{\mathbf Q}_p)\otimes B_{\crr},\quad \rho_{\hk}:\, \R\Gamma^B_{\hk}(X_h)^{\tau}_{B_{\crr}}\stackrel{\sim}{\to}\R\Gamma(X_{\eet},{\mathbf Q}_p)\otimes B_{\crr},\\
\rho_{\dr}: \,& \R\Gamma_{\dr}(X_h)\otimes_{\ovk}{B_{\dr}}\stackrel{\sim}{\to}\R\Gamma(X_{\eet},{\mathbf Q}_p)\otimes B_{\dr}.
\end{align*}
\end{theorem}
Pulling back $\rho_{\hk}$ to the Fontaine-Hyodo-Kato ${\mathbb G}_a$-torsor $\Spec(B_{\st})/\Spec(B_{\crr})$ we get a canonical quasi-isomorphism of $B_{\st}$-complexes
\begin{equation}
\rho_{\hk}: \R\Gamma^B_{\hk}(X_h)\otimes_{K_0^{\nr}}B_{\st}\stackrel{\sim}{\to}\R\Gamma(X_{\eet},{\mathbf Q}_p)\otimes B_{\st}
\end{equation}
compatible with the $(\phi,N)$-action and with the $G_K$-action on ${\mathcal V}ar_{\ovk}$.
\begin{corollary}
\label{period-compatibility}
The period morphisms  are compatible, i.e., the following diagrams commute.
$$
\xymatrix{
\R\Gamma^B_{\hk}(X_h)\otimes_{K_0^{\nr}}B_{\st}\ar[r]^-{\iota^B_{\dr}\otimes\iota}\ar[d]^{\rho_{\hk}} & \R\Gamma_{\dr}(X_h)\otimes_{\ovk}B_{\dr}\ar[d]^{\rho_{\dr}} &
 \R\Gamma_{\crr}(X_h)\otimes_{A_{\crr}}B_{\dr}\ar[d]_{\rho_{\crr}\otimes\id_{B_{\dr}}} & \R\Gamma_{\dr}(X_h)\otimes_{\ovk}B_{\dr} \ar[dl]^{\rho_{\dr}}\ar[l]_{\gamma_{\dr}}^{\sim} \\ 
 \R\Gamma(X_{\eet},{\mathbf Q}_p)\otimes B_{\st}\ar[r]^-{1\otimes \iota} & \R\Gamma(X_{\eet},{\mathbf Q}_p)\otimes B_{\dr} & 
 \R\Gamma(X_{\eet},{\mathbf Q}_p)\otimes B_{\dr}\\
}
$$
\end{corollary}
\begin{proof}
 The second diagram commutes by \cite[3.4]{BE2}.  The commutativity of the first one can be reduced, by 
the equality $\rho_{\hk}=\rho_{\crr}\iota^B_{\crr}$ and the second diagram above, to the commutativity of the right diagram in Lemma \ref{HK-compatibility}.
\end{proof}

  We will now define the syntomic period map $$\rho_{\synt}:\, \R\Gamma_{\synt}(X_h,r)_{\mathbf Q}\to \R\Gamma(X_{\eet},{\mathbf Q}_p(r)),\quad r\geq 0.
$$
 Set
 ${\mathbf Z}/p^n(r)^{\prime}:=(1/(p^aa!){\mathbf Z}_p(r))\otimes{\mathbf Z}/p^n$, where $a$ is the largest integer $\leq r/(p-1)$.  Recall that we have 
  the fundamental exact sequence \cite[Theorem 1.2.4]{Ts}
 $$0\to {\mathbf Z}/p^n(r)^{\prime}\to J_{\crr,n}^{<r>}\lomapr{1-\phi_r}A_{\crr,n}\to 0,
 $$
 where  $$J_n^{<r>}:= \{x\in J_{n+s}^{[r]}\mid \phi(x)\in p^rA_{\crr,n+s}\}/p^n ,$$
for some $s\geq r$.
Set
 $S_n(r):=\Cone(J^{[r]}_{\crr,n}\lomapr{p^r-\phi} A_{\crr,n})[-1]$.   There is a natural morphism of complexes 
 $S_n(r)\to{\mathbf Z}/p^n(r)^{\prime}$ (induced by $p^r $ on $J_{\crr,n}^{[r]}$ and $\id$ on $A_{\crr,n}$) , whose kernel and cokernel are annihilated by $p^r$.

  The Filtered Crystalline Poincar\'e Lemma implies easily the following Syntomic Poincar\'e Lemma.
 \begin{corollary}
 \begin{enumerate}
 \item For $0\leq r\leq p-2$, there is a  unique quasi-isomorphism ${\mathbf Z}/p^n(r)\stackrel{\sim}{\longrightarrow}\sss_n(r)$ of complexes of sheaves on ${\mathcal V}ar_{\ovk,h}$  that is compatible with the Crystalline Poincar\'e Lemma.
 \item
There is a  unique quasi-isomorphism $S_n(r)\stackrel{\sim}{\to}\sss_n(r)$ of complexes of sheaves on ${\mathcal V}ar_{\ovk,h}$ that is compatible with the Crystalline Poincar\'e Lemma.
\end{enumerate}
\end{corollary}
\begin{proof}
We will prove the second claim - the first one is proved in an analogous way. Consider the following map of distinguished triangles
$$
\xymatrix{
\sss_n(r)\ar[r] & \sj^{[r]}_{\crr,n}\ar[r]^{p^r-\phi} & \sa_{\crr,n}\\
S_n(r)\ar[r]\ar@{-->}[u]& J^{[r]}_{\crr,n}\ar[u]^{\wr}\ar[r]^{p^r-\phi} & A_{\crr,n}\ar[u]^{\wr}
}
$$
The  triangles are distinguished by definition.
The vertical continuous arrows are quasi-isomorphisms by the Crystalline Poincar\'e Lemma. They induce the dash arrow  that is clearly a quasi-isomorphism.
\end{proof}

  Consider the natural map
$\alpha_n: \R\Gamma(X_h,\sss(r))\to \R\Gamma(X_h,\sss_n(r))$ and the zig-zag
$$\beta_n:\, \R\Gamma(X_h,\sss_n(r))\leftarrow \R\Gamma(X_{h},S_n(r))\to \R\Gamma(X_{\eet},{\mathbf Z}/p^n(r)^{\prime})\stackrel{\sim}{\leftarrow}
\R\Gamma(X_{h},{\mathbf Z}/p^n(r)^{\prime}).$$ Set $\beta:=(\holim_n\beta_{n})\otimes {\mathbf Q}$; note that this  is a quasi-isomorphism. 
 Set $$
 \rho_{\synt}:=p^{-r}\beta\alpha:\quad  \R\Gamma_{\synt}(X_h,r)\to \R\Gamma(X_{\eet},{\mathbf Q}_p(r)),
 $$
 where $\alpha:=(\holim_n\alpha_{n})\otimes {\mathbf Q}$.  The period map $\rho_{\synt}$  induces a map of graded $E_{\infty}$ algebras over $\bq_p$ compatible with the action of the Galois group $G_K$.

  The syntomic period map has a different, more global definition that we find very useful. Define the map $\rho_{\synt}^{\prime}$ by the following diagram.
$$
\xymatrix@C=40pt{
\R\Gamma_{\synt}(X_h,r)\ar[r]^{\sim} \ar[d]^{\rho_{\synt}^{\prime}}& [\R\Gamma_{\crr}(X_h)_\bq\ar[r]^-{(1-\phi_r,\gamma^{-1}_r\kappa^{-1}_r)}\ar[d]^{\rho_{\crr}} & \R\Gamma_{\crr}(X_h)_\bq\oplus \R\Gamma_{\dr}(X_h)/F^r]\ar[d]^{\rho_{\crr}+\rho_{\dr}}\\
  \R\Gamma_{\eet}(X,\bq_p(r))\ar[r]^-{\sim} & [\R\Gamma_{\eet}(X,\bq_p(r))\otimes B_{\crr}\ar[r]^-{(1-\phi_r, \can) } & \R\Gamma_{\eet}(X,\bq_p(r))\otimes B_{\crr}\oplus \R\Gamma_{\eet}(X,\bq_p(r))\otimes B_{\dr}/F^r]
}
$$
This definition makes sense since the following diagram commutes.
$$
\xymatrix{
\R\Gamma_{\crr}(X_h)_\bq\ar[r]^{\gamma^{-1}_r\kappa^{-1}_r}\ar[d]^{\rho_{\crr}} & \R\Gamma_{\dr}(X_h)/F^r\ar[d]^{\rho_{\dr}}\\
 \R\Gamma_{\eet}(X,\bq_p(r))\otimes B_{\crr}\ar[r]^-{\can } & \R\Gamma_{\eet}(X,\bq_p(r))\otimes B_{\dr}/F^r
}
$$
The  syntomic period morphisms $\rho_{\synt}$ and $\rho_{\synt}^{\prime}$ are homotopic by a  homotopy compatible with the $G_K$-action (and, unless necessary, we will not distinguish them in what follows). These two facts follow easily from the definitions.

   For $X\in {\mathcal V}ar_K$, we have a quasi-isomorphism
  
  \begin{equation}
  \label{definition1}
  \alpha_{\eet}: \R\Gamma_{}(X_{\ovk,\eet},{\mathbf Q}_p(r))\stackrel{\sim}{\to} C(\R\Gamma^B_{\hk}(X_{\ovk,h})\{r\})
\end{equation}
 that we define  as the inverse of the following composition of  quasi-isomorphisms (square brackets denote complex)
\begin{align*}
C(\R\Gamma^B_{\hk}(X_{\ovk,h})\{r\}) & \twomapr{\rho}{\sim}\R\Gamma_{}(X_{\ovk,\eet},{\mathbf Q}_p)\otimes_{{\mathbf Q}_p}
[B_{\st}\verylomapr{(N,1-\phi_r,\iota)}B_{\st}\oplus B_{\st}\oplus B_{\dr}/F^r\veryverylomapr{(1-\phi_{r-1})-N}  B_{\st}]\\
& \stackrel{\sim}{\leftarrow}\R\Gamma_{}(X_{\ovk,\eet},{\mathbf Q}_p)\otimes_{{\mathbf Q}_p}C(D_{\st}({\mathbf Q}_p(r)))
\stackrel{\sim}{\leftarrow}\R\Gamma_{}(X_{\ovk,\eet},{\mathbf Q}_p(r)).
\end{align*}
The last quasi-isomorphism is by Remark \ref{basics}. The map $\rho$ is defined using the period morphisms $\rho_{\hk}$ and $\rho_{\dr}$ and their compatibility (Corollary \ref{period-compatibility}). The map $\alpha_{\eet}$ is compatible with the action of $G_K$.

\begin{proposition}
\label{BK2}For a variety  $X\in {\mathcal V}ar_{K}$, we have a canonical, compatible with the action of  $G_K$,  quasi-isomorphism
$$\rho_{\synt}: \tau_{\leq r}\R\Gamma_{\synt}(X_{\ovk,h},r)\stackrel{\sim}{\to} \tau_{\leq r}\R\Gamma(X_{\ovk,\eet},{\mathbf Q}_p(r)).
$$
\end{proposition}
\begin{proof}
The Bousfield-Kan spectral sequences associated to the homotopy limits defining the complexes $C^+(H^j_{\hk}(X_{\ovk,h})\{r\})$ and $C(H^j_{\hk}(X_{\ovk,h})\{r\})$ form the following commutative diagram
$$
\xymatrix{
^+E^{i,j}_2=H^i(C^+(H^j_{\hk}(X_{\ovk,h})\{r\}))\ar[d]^{\can}\ar@{=>}[r] & H^{i+j}(C^+(\R\Gamma^B_{\hk}(X_{\ovk,h})\{r\}))\ar[d]^{\can}\\
E^{i,j}_2=H^i(C(H^j_{\hk}(X_{\ovk,h})\{r\}))\ar@{=>}[r] & H^{i+j}(C(\R\Gamma^B_{\hk}(X_{\ovk,h})\{r\}))
}
$$
We have  $D_j=H^j_{\hk}(X_{\ovk,h})\{r\}\in MF_K^{\ad}(\phi, N,G_K)$. For $j\leq r$, $F^{1}D_{j,K}=F^{1-(r-j)}H^j_{\dr}(X_{h})\{r\}=0$. Hence, by  Corollary \ref{resolution3},  
we have $^+E^{i,j}_2\stackrel{\sim}{\to} E^{i,j}_2$. This implies that  $\tau_{\leq r}C^+(\R\Gamma^B_{\hk}(X_{\ovk,h})\{r\})\stackrel{\sim}{\to} \tau_{\leq r}C(\R\Gamma^B_{\hk}(X_{\ovk,h})\{r\})$.

    Since $\rho_{\hk}=\rho_{\crr}\iota^B_{\crr}$, we check easily  that we have the following commutative diagram
\begin{equation}
\label{compatibility0}
\begin{CD}
\R\Gamma_{\synt}(X_{\ovk,h},r)@>\sim >\alpha_{\synt}> C^+(\R\Gamma^B_{\hk}(X_{\ovk,h})\{r\})\\
@VV\rho_{\synt}V @VV\can V\\
\R\Gamma_{}(X_{\ovk,\eet},{\mathbf Q}_p(r))@>\sim >\alpha_{\eet}> C(\R\Gamma^B_{\hk}(X_{\ovk,h})\{r\})
\end{CD}
\end{equation}
It follows that $\rho_{\synt}: \tau_{\leq r}\R\Gamma_{\synt}(X_{\ovk,h},r)\stackrel{\sim}{\to} \tau_{\leq r}\R\Gamma(X_{\ovk,\eet},{\mathbf Q}_p(r))
$, as wanted.
\end{proof}

 Let $X\in {\mathcal V}ar_K$. The natural projection $\varepsilon: X_{\ovk,h}\to X_h$ defines pullback maps
 $$\varepsilon^*: \R\Gamma^B_{\hk}(X_h)\to \R\Gamma^B_{\hk}(X_{\ovk,h}),\quad  \varepsilon^*: \R\Gamma_{\dr}(X_h)\to \R\Gamma_{\dr}(X_{\ovk,h}).
 $$
 By construction they are compatible with the monodromy operator, Frobenius, the  action of the Galois group $G_K$, and filtration. It is also clear that they are compatible with the Beilinson-Hyodo-Kato morphisms, i.e., that
the following diagram commutes
$$
\xymatrix{
  \R\Gamma^B_{\hk}(X_h)\ar[r]^{\iota^B_{\dr}}\ar[d]^{\varepsilon^*} & \R\Gamma_{\dr}(X_h)\ar[d]^{\varepsilon^*}\\
   \R\Gamma^B_{\hk}(X_{\ovk,h})\ar[r]^{\iota^B_{\dr}} & \R\Gamma_{\dr}(X_{\ovk,h}).
}
$$
It follows that we can define a canonical pullback map $$
\varepsilon^*:\quad   C_{\st}(\R\Gamma^B_{\hk}(X_h)\{r\})\to  C^+(\R\Gamma^B_{\hk}(X_{\overline{K},h})\{r\}).
$$  
\begin{lemma}
\label{compatibilit01}
Let $r\geq 0$.
The following diagram commutes in the derived category.
$$
\xymatrix{
\R\Gamma_{\synt}(X_h,r)\ar[r]^-{\alpha_{\synt}}\ar[d]^{\varepsilon^*}  & C_{\st}(\R\Gamma^B_{\hk}(X_h)\{r\})\ar[d]^{\varepsilon^*}\\
\R\Gamma_{\synt}(X_{\ovk,h},r)\ar[r]^-{\alpha_{\synt}} & C^+(\R\Gamma^B_{\hk}(X_{\overline{K},h})\{r\}).
}
$$
\end{lemma}
\begin{proof}
Take a number $t\geq 2\dim X +2$ and choose a finite Galois extension $(V^{\prime},K^{\prime})/(V,K)$ (see the proof of Proposition \ref{hypercov}) such that we have an $h$-hypercovering $Z_{\scriptscriptstyle\bullet}\to X_{K^{\prime}}$ with $(Z_{\scriptscriptstyle\bullet})_{\leq t+1}$ built from log-schemes log-smooth over $V^{\prime, \times}$ and of Cartier type. Since the top map $\alpha_{\synt}$ is  compatible with base change (cf. Proposition \ref{reduction2}) it suffices to show that the diagram in the lemma commutes with $X$ replaced by $(Z_{\scriptscriptstyle\bullet})_{\leq t+1}$. 
By (\ref{isomorphism}) and
Propositions
%\ref{isomorphism},
\ref{hypercov} and \ref{deRham1}, this reduces to showing that, for  an ss-pair $(U,\ove{U})$ split over $V$,  the following diagram commutes canonically in the $\infty$-derived category (we 
set $Y:=(U,\ove{U}), \ove{Y}:=Y_{\ovv}$, $\pi$ - a fixed uniformizer of $V$).
$$
\xymatrix{
\R\Gamma_{\synt}(Y,r)_\bq\ar[r]^-{\alpha^B_{\synt,\pi}}\ar[d]^{\varepsilon^*}  & C_{\st}(\R\Gamma^B_{\hk}(Y)\{r\})\ar[d]^{\varepsilon^*}\\
\R\Gamma_{\synt}(Y_{\ovk},r)_\bq \ar[r]^-{\alpha_{\synt}} & C^+(\R\Gamma^B_{\hk}(Y_{\overline{K}})\{r\}).
}
$$

  From the uniqueness property of the homotopy fiber functor,  it suffices to show that the following diagram commutes canonically in the $\infty$-derived category.
$$
\xymatrix{
\R\Gamma_{\crr}(Y)_\bq\ar[r] \ar[d] & \R\Gamma_{\crr}(Y/R)^{N=0} _\bq & \R\Gamma_{\hk}^B(Y_1)^{\tau,N=0}_{R_\bq}\ar[l]_{\iota_{\pi}}^{\sim} & 
\R\Gamma_{\hk}^B(Y_1)^{N=0}\ar[l]_{\beta}^{\sim}\ar[ld]\\
\R\Gamma_{\crr}(\ove{Y})_\bq
 & \R\Gamma_{\hk}^B(\ove{Y}_1)^{\tau,N=0}_{B^+_{\crr}}\ar[l]_-{\iota^B_{\crr}} ^-{\sim}& (\R\Gamma_{\hk}^B(\ove{Y}_1)\otimes _{K_0^{\nr}}B^+_{\st})^{N=0}\ar@{=}[l]
 }
$$
To do that we will need the  ring of periods $\wh{A}_{\st}$ \cite[p.253]{Ts}.
   Set
$$
\wh{A}_{\st,n}=H^0_{\crr}(\overline{V}_{n}^{\times}/R_{n}), \quad \wh{A}_{\st}=\invlim_nH^0_{\crr}(\overline{V}_{n}^{\times}/R_{n}).
$$
The ring $\wh{A}_{\st,n}$ has a natural action of $G_K$, Frobenius $\phi $,
and a monodromy
operator $N$. It is also equipped with a
PD-filtration $F^i\wh{A}_{\st,n}=H^0_{\crr}(\overline{V}_{n}^{\times}/R_{n},\sj_{\crr,n}^{[i]})$.
We have a  morphism $A_{\crr,n}\to \wh{A}_{\st,n}$ induced by the map
$H^0_{\crr}(\overline{V}_{n}/W_n(k))\to H^0_{\crr}(\overline{V}_{n}^{\times}/R_{n})$. It is compatible with the Galois action, the Frobenius,
and the filtration.
The natural map $R_{n}\to \wh{A}_{\st,n}$ is compatible with all the structures.
  We can view $\wh{A}_{\st,n}$
   as the PD-envelope of the closed immersion
   $$
   \overline{V}_n^{\times}\hookrightarrow A_{\crr,n}\times_{W_n(k)}W_n(k)[X]^{\times}
   $$
   defined by the map $\theta: A_{\crr,n}\to \overline{V}_n$ and the projection $W_n(k)[X] \to \overline{V}_n$, $X\mapsto \pi$.
   This makes $\overline{V}_1^{\times}\hookrightarrow \wh{A}_{\st,n}$ into a PD-thickening in the crystalline site of
   $\overline{V}_1$.  Set $\wh{B}^+_{\st}:=\wh{A}_{\st}[1/p]$.

Commutativity of the last diagram will follow from  the following commutative diagram 
$$
\xymatrix{
\R\Gamma_{\crr}(Y)_\bq\ar[d] \ar[rr]  & &\R\Gamma_{\crr}(\ove{Y})_\bq\ar[dl]^{\sim}\\
\R\Gamma_{\crr}(Y/R)^{N=0} _\bq\ar[r]  &  \R\Gamma_{\crr}(\ove{Y}/\wh{A}_{\st})^{N=0}_\bq\\
\R\Gamma_{\hk}^B(Y_1)^{\tau,N=0}_{R_\bq}\ar[u]^{\iota_{\pi}}_{\wr}\ar[r]& \R\Gamma_{\hk}^B(Y_1)^{\tau,N=0}_{\wh{B}^+_{\st}} \ar[u]^{\iota^B_{\wh{B}^+_{\st}}}_{\wr}
&\R\Gamma_{\hk}^B(Y_1)^{\tau,N=0}_{B^+_{\crr}}\ar[l]\ar[uu]^{\iota_{\crr}^B}_{\wr}\\
& \R\Gamma_{\hk}^B(Y_1)^{N=0} \ar[ur]\ar[ul]^{\beta}_{\sim}
}
$$
as soon as we show that 
 the map $\R\Gamma_{\crr}(\ove{Y})_\bq  \to \R\Gamma_{\crr}(\ove{Y}/\wh{A}_{\st})^{N=0}_\bq$ is a quasi-isomorphism.  Notice that the map $\iota^B_{\wh{B}^+_{\st}}$ is a quasi-isomorphism by Theorem \ref{Bthm}.  Hence using the Beilinson-Hyodo-Kato maps  $\iota^B_{\wh{B}^+_{\st}}$   and $ \iota_{\crr}^B$ this reduces to proving that the canonical map
$ \R\Gamma_{\hk}^B(Y_1)^{\tau,N=0}_{B^+_{\crr}}\to \R\Gamma_{\hk}^B(Y_1)^{\tau,N=0}_{\wh{B}^+_{\st}} $ is a quasi-isomorphism.  In fact, we claim that for any $(\phi, N)$-module $M$ we have an isomorphism 
$M_{B_{\crr}^{+}}^{\tau,N=0}\stackrel{\sim}{\to} M_{\wh{B}_{\st}^{+}}^{\tau,N=0}.$ Indeed, assume first that the monodromy $N_M$ is trivial. We calculate
\begin{align*}
M^{\tau}_{B^+_{\crr}} & =(M\otimes _{K_0}B^{+,\tau}_{\crr})^{N^{\prime}=0}=M\otimes _{K_0}(B^{+,\tau}_{\crr})^{N_{\tau}=0}=M\otimes_{K_0} B^+_{\crr},\quad N^{\prime}=
N_M\otimes 1 +1\otimes N_{\tau}=1\otimes N_{\tau},\\
M^{\tau}_{\wh{B}^+_{\st}} & =(M\otimes _{K_0}\wh{B}^{+,\tau}_{\st})^{N^{\prime}=0}=M\otimes_{K_0} (\wh{B}^{+,\tau}_{\crr})^{N_{\tau}=0}=M\otimes _{K_0}\wh{B}^+_{\st}
\end{align*}
Hence $M_{B_{\crr}^{+}}^{\tau,N=0}=M\otimes _{K_0}B^+_{\crr}$ and $M_{\wh{B}_{\st}^{+}}^{\tau,N=0}=M\otimes _{K_0}(\wh{B}^+_{\st})^{N=0}=M\otimes _{K_0}B^+_{\crr}$, where the last equality is proved in \cite[Lemma 1.6.5]{Ts}.  We are done in this case.

  In general, we can write $M\otimes _{K_0}B^+_{\st}\stackrel{\sim}{\leftarrow} M^{\prime}\otimes _{K_0}B^+_{\st}$ for a $(\phi,N)$-module $M^{\prime}$ such that $N_{M^{\prime}}=0$ (take for $M^{\prime}$ the image of the map $M\to M\otimes _{K_0}B^+_{\st}$, $m\mapsto \exp(N_M(m)u)$, for $u\in B^+_{\st}$ such that $B^+_{\st}=B^+_{\crr}[u]$, $N_{\tau}(u)=-1$). Similarly, using the fact that the  ring $B^+_{\st}$ is canonically (and compatibly with all
the structures) isomorphic to the elements of $\wh{B}^+_{\st}$ annihilated by a power of 
the monodromy operator \cite[3.7]{Kas}, we can write in a compatible way $M\otimes _{K_0}{B}^+_{\st}\stackrel{\sim}{\leftarrow} M^{\prime}\otimes _{K_0}\wh{B}^+_{\st}$  for the same module $M^{\prime}$.
  We obtained a commutative diagram
  $$
  \begin{CD}
  M_{B_{\crr}^{+}}^{\tau,N=0}@>>> M_{\wh{B}_{\st}^{+}}^{\tau,N=0}  \\
  @VV\wr V @VV\wr V\\
  M_{B_{\crr}^{+}}^{\prime \tau,N=0}@>\sim>> M_{\wh{B}_{\st}^{+}}^{\prime \tau,N=0}
  \end{CD}
  $$
  that reduces the general case to the case of trivial monodromy on $M$ that we treated above.
  \end{proof}

 Let $X\in {\mathcal V}ar_K$, $r\geq 0$. 
 Set 
 \begin{align*}
 C_{\pst}(\R\Gamma^B_{\hk}(X_{\overline{K},h})\{r\}):=
 \left[\begin{aligned}
 \xymatrix@C=50pt{\R\Gamma^B_{\hk}(X_{\ovk,h})^{G_K}\ar[r]^-{(1-\phi_r,\iota^B_{\dr})}\ar[d]^{N}  & \R\Gamma^B_{\hk}(X_{\ovk,h})^{G_K}\oplus (\R\Gamma_{\dr}(X_{\ovk,h})/F^r)^{G_K}\ar[d]^{(N,0)}\\
\R\Gamma^B_{\hk}(X_{\ovk,h})^{G_K}\ar[r]^{1-\phi_{r-1}}& \R\Gamma^B_{\hk}(X_{\ovk,h})^{G_K}}\end{aligned}\right]
 \end{align*}
 The above makes sense since the action of $G_K$ on $\R\Gamma^B_{\hk}(X_{\overline{K},h})\{r\}$ and $  \R\Gamma_{\dr}(X_{\ovk,h})$ is smooth. In particular, we have
 \begin{align*}
H^j(\R\Gamma^B_{\hk}(X_{\overline{K},h})\{r\}^{G_K})  \simeq H^j(\R\Gamma^B_{\hk}(X_{\overline{K},h})\{r\})^{G_K},\quad H^j(\R\Gamma_{\dr}(X_{\ovk,h})^{G_K}) &  \simeq  H^j(\R\Gamma_{\dr}(X_{\ovk,h}))^{G_K}. \end{align*}

Consider the canonical pullback map $$
\varepsilon^*:  C_{\st}(\R\Gamma^B_{\hk}(X_h)\{r\})\stackrel{\sim}{\to} C_{\pst}(\R\Gamma^B_{\hk}(X_{\overline{K},h})\{r\}).
$$  
By Proposition \ref{HKdR}, this is a quasi-isomorphism. 
This allows us to construct a canonical spectral sequence (the {\em syntomic descent spectral sequence})
 \begin{equation}
 \label{kwak2}
 \xymatrix{
^{\synt}E^{i,j}_2=H^i_{\st}(G_K,H^j(X_{\ovk,\eet},{\mathbf Q}_p(r)))\ar@{=>}[r] & H^{i+j}_{\synt}(X_{h},r)
}
\end{equation}Indeed, the Bousfield-Kan spectral sequences associated to the homotopy limits defining complexes  $C_{\pst}(\R\Gamma^B_{\hk}(X_{\ovk,h})\{r\}))$  and $C_{\st}(\R\Gamma^B_{\hk}(X_{h})\{r\}))$ give us  the following commutative diagram
$$
\xymatrix{
^{\pst} E^{i,j}_2=H^i(C_{\pst}(H^j_{\hk}(X_{\ovk,h})\{r\}))\ar@{=>}[r] & H^{i+j}(C_{\pst}(\R\Gamma^B_{\hk}(X_{\ovk,h})\{r\}))\\
^{\synt} E^{i,j}_2=H^i(C_{\st}(H^j_{\hk}(X_{h})\{r\}))\ar[u]^{\wr}_{\varepsilon^*}\ar@{=>}[r] & H^{i+j}(C_{\st}(\R\Gamma^B_{\hk}(X_{h})\{r\}))\ar[u]^{\wr}_{\varepsilon^*}
}
$$
Since, by Proposition \ref{reduction2}, we have $\alpha_{\synt}: H^{i+j}_{\synt}(X_h,r)
\stackrel{\sim}{\to}H^{i+j}(C_{\st}(\R\Gamma^B_{\hk}(X_h)\{r\}))$, we have obtained a spectral sequence
$$
\xymatrix{
E^{i,j}_2=H^i(C_{\pst}(H^j_{\hk}(X_{\ovk,h})\{r\}))\ar@{=>}[r] & H^{i+j}_{\synt}(X_{h},r)
}
$$

  It remains to show that there is a canonical isomorphism 
\begin{equation}
\label{eq1}
H^i(C_{\pst}(H^j_{\hk}(X_{\ovk,h})\{r\}))\simeq H^i_{\st}(G_K,H^j(X_{\ovk,\eet},{\mathbf Q}_p(r))).
\end{equation}
But, we have $D_j=H^j_{\hk}(X_{\ovk,h})\{r\}\in MF_K^{\ad}(\phi, N,G_K)$, $V_{\pst}(D_j)\simeq H^j(X_{\ovk,\eet},\bq(r))$, and $D_{\pst}(H^j(X_{\ovk,\eet},\bq(r)))\simeq D_j$.   Hence isomorphism (\ref{eq1}) follows from Remark \ref{pst=st} and we have obtained the spectral sequence (\ref{kwak2}).
\subsection{Arithmetic case} In this subsection, we define the arithmetic syntomic period map by Galois descent from the geometric case. Then we show that, via this period map, the syntomic descent spectral sequence and the \'etale Hochschild-Serre spectral sequence are compatible. Finally, we show that this implies that the arithmetic syntomic cohomology and \'etale cohomology are isomorphic in a stable range.

Let $X\in {\mathcal V}ar_K$. For $r\geq 0$, we define the canonical syntomic period map 
$$
\rho_{\synt}: \quad \R\Gamma_{\synt}(X_h,r)\to \R\Gamma(X_{\eet},{\mathbf Q}_p(r)),
$$
 as  the following composition 
\begin{align*}
\R\Gamma_{\synt}(X_h,r) & =\R\Gamma(X_h,\sss(r))_\bq \to \holim_n \R\Gamma(X_{h},\sss_n(r))_\bq \stackrel{\varepsilon^*}{\to}  \holim_n \R\Gamma (G_K, \R\Gamma(X_{\ovk,h},\sss_n(r)))_\bq\\
 & \lomapr{p^{-r}\beta} \holim_n \R\Gamma (G_K, \R\Gamma(X_{\ovk,\eet},{\mathbf Z}/p^n(r)^{\prime}))_\bq \stackrel{\sim}{\leftarrow}\holim_n \R\Gamma(X_{\eet},{\mathbf Z}/p^n(r)^{\prime})_\bq =\R\Gamma(X_{\eet},{\mathbf Q}_p(r)).
\end{align*}
It induces a morphism of graded $E_{\infty}$ algebras over $\bq_p$.

The syntomic period map $\rho_{\synt}$ is compatible with the syntomic descent and the Hochschild-Serre spectral sequences.
\begin{theorem}
\label{stHS}
For $X\in \sv ar_K$, $r\geq 0$, there is  a canonical map of spectral sequences
$$
\xymatrix{
^{\synt}E^{i,j}_2=H^i_{\st}(G_K,H^j(X_{\ovk,\eet},{\mathbf Q}_p(r)))\ar[d]^{\can}\ar@{=>}[r] & H^{i+j}_{\synt}(X_{h},r)\ar[d]^{\rho_{\synt}}\\
^{\eet}E^{i,j}_2=H^i(G_K,H^j(X_{\ovk,\eet},{\mathbf Q}_p(r)))\ar@{=>}[r] & H^{i+j}(X_{\eet},{\mathbf Q}_p(r))
}
$$
\end{theorem}
\begin{proof}We work in the (classical) derived category. The Bousfield-Kan spectral sequences associated to the homotopy limits defining complexes $C(\R\Gamma^B_{\hk}(X_{\ovk,h})\{r\})$ and $C_{\pst}(\R\Gamma^B_{\hk}(X_{\ovk,h})\{r\})$, and Theorem \ref{speccomp}  give us  the following commutative diagram of spectral sequences
$$
\xymatrix{ ^{II}E^{i,j}_2=H^i(G_K,C(H^j_{\hk}(X_{\ovk,h})\{r\}))\ar@{=>}[r] & H^{i+j}(G_K,C(\R\Gamma^B_{\hk}(X_{\ovk,h})\{r\}))\\
^{\pst} E^{i,j}_2=H^i(C_{\pst}(H^j_{\hk}(X_{\ovk,h})\{r\}))\ar[u]^{\delta}\ar@{=>}[r] & H^{i+j}(C_{\pst}(\R\Gamma^B_{\hk}(X_{\ovk,h})\{r\}))\ar[u]^{\delta}
}
$$

    More specifically, in the language of Section \ref{Jan}, set $X=C(\R\Gamma^B_{\hk}(X_{\ovk,h})\{r\})$ (hopefully, the notation will not be too confusing). Filtering complex $X$  in the direction of the homotopy limit we obtain a Postnikov system (\ref{postnikov}) with $Y^i=0$, $i\geq 3$, and 
\begin{align*}
Y^0 & =\R\Gamma^B_{\hk}(X_{\ovk,h})\{r\}\otimes_{K_0^{\nr}} B_{\st},\\
 Y^1 & =\R\Gamma^B_{\hk}(X_{\ovk,h})\{r-1\}\otimes_{K_0^{\nr}} B_{\st}\oplus (\R\Gamma^B_{\hk}(X_{\ovk,h})\{r\}\otimes_{K^{\nr}_0} B_{\st}\oplus (\R\Gamma_{\dr}(X_{\ovk})\otimes_{\ovk} B_{\dr})/F^r),\\
 Y^2 & =\R\Gamma^B_{\hk}(X_{\ovk,h})\{r-1\}\otimes _{K_0^{\nr}}B_{\st}.
\end{align*}
Still in the setting of Section \ref{Jan}, take for $A$ the abelian category of sheaves of abelian groups on the pro-\'etale site  $\Spec(K)_{\proeet}$ of Scholze \cite[3]{Sch}. 
\begin{remark}
We work with the pro-\'etale site to make sense of the continuous cohomology $\R\Gamma(G_K,\cdot)$. If the reader is willing to accept that this is possible then he can skip the tedious parts of the proof involving passage to the pro-\'etale site (and existence of continuous sections). 
\end{remark}
Recall that there is a projection map $\nu: \Spec(K)_{\proeet}\to \Spec(K)_{\eet}$ such that, for an \'etale sheaf $\sff$, we have the quasi-isomorphism  $\nu^*: \sff\stackrel{\sim}{\to}\R\nu_*\nu^*\sff$ \cite[5.2.6]{BhS}. More generally,  for a topological $G_K$-module $M$, we get a sheaf $\nu M$ on $\Spec(K)_{\proeet}$ by setting, for a profinite $G_K$-set $S$, 
$\nu M(S)=\Hom_{\cont,G_K}(S,M)$, and Scholze showed that there is a canonical quasi-isomorphism $H^*(\Spec(K)_{\proeet},\nu M)\simeq H^*_{\cont}(G_K,M)$ 
\cite[3.7 (iii)]{Sch}, \cite{ScE}. In this proof we will need this kind of quasi-isomorphisms for complexes $M$ as well and this will require extra arguments. For that observe that the functor $\nu$ is left exact. To study  right exactness, it suffices to look 
at the global sections  on profinite sets $S$ with a free $G_K$-action of the form $S=S^{\prime}\times G_K$ for a profinite set $S^{\prime}$ with trivial $G_K$-action\footnote{To see this, for a profinite $G_K$-set $S^{\prime}$, use the covering $S^{\prime}\times G_K\to S^{\prime}$, where the first $S^{\prime}$ has trivial $G_K$-action, induced from the $G_K$-action on $S^{\prime}$.}. Then, for any $G_K$-module $T$,  we have $\Gamma(S,\nu T)=\Hom_{\cont}(S^{\prime},T)$. It follows that, for  a surjective map $T_1\to T_2$ of $G_K$-modules, the pullback map $\nu T_1\to \nu T_2$ is also surjective if the original map had a continuous set-theoretical section. This is a criterion familiar from continuous cohomology and we will use it often.

We will see the complex $X$ as a complex of sheaves on the site  $\Spec(K)_{\proeet}$ in the following way: represent  $\R\Gamma^B_{\hk}(X_{\ovk,h})$ and 
$\R\Gamma_{\dr}(X_{\ovk})$ by (filtered) perfect complexes of $K_0^{\nr}$- and $\ovk$-modules, respectively,   think of $X$ as $\nu X$, and work on the pro-\'etale site. This makes sense, i.e., functor $\nu$  transfers (filtered) quasi-isomorphisms of representatives of $\R\Gamma^B_{\hk}(X_{\ovk,h})$ and $\R\Gamma_{\dr}(X_{\ovk})$ to quasi-isomorphisms of the corresponding sheaves $\nu X$. To see this look at the Postnikov system of sheaves on $\Spec(K)_{\proeet}$ obtained by pulling back by $\nu$ the above Postnikov system. Now, look 
at the global sections on profinite sets $S=S^{\prime}\times G_K$ as above and note that we have $\Gamma(S, \nu Y^0)=\Hom_{\cont}(S^{\prime},Y^0)$. Conclude that, by  perfectness of the Beilinson-Hyodo-Kato complexes, quasi-isomorphisms of representatives of $\R\Gamma^B_{\hk}(X_{\ovk,h})$ yield   quasi-isomorphisms of the  sheaves $\nu Y^0$. By a similar argument, we get the analogous statement for $Y^2$. For $Y^1$, we just have to show that filtered quasi-isomorphisms of representatives of $\R\Gamma_{\dr}(X_{\ovk})$ yield   quasi-isomorphisms of the sheaves $\nu ((\R\Gamma_{\dr}(X_{\ovk})\otimes_{\ovk}B_{\dr})/F^r)$. Again, we look at the global section on $S=S^{\prime}\times G_K$ as above. By compactness of $S^{\prime}$ we may replace $(\R\Gamma_{\dr}(X_{\ovk})\otimes_{\ovk}B_{\dr})/F^r$ by  
$(t^{-i}\R\Gamma_{\dr}(X_{\ovk})\otimes_{\ovk}B^+_{\dr})/F^r$, for some $i\geq 0$, where, using devissage, we can again argue by (filtered) perfection of $\R\Gamma_{\dr}(X_{\ovk})$. Observe that the same argument shows that  $\sh^j(\nu Y^i)\simeq \nu H^j(Y^i)$, for $i=0,1,2$.

   The above Postnikov system gives rise to an exact couple
$$
D_1^{i,j}=\sh^j(X^i),\quad E_1^{i,j}=\sh^j(Y^i) \Longrightarrow \sh^{i+j}(X)
$$
This is the Bousfield-Kan spectral sequence associated to $X$.

   Consider now the complex $X_{\pst}:=C_{\pst}(\R\Gamma^B_{\hk}(X_{\ovk,h})\{r\})$. We claim that
 the canonical map
 \begin{align*}
   C_{\pst}(\R\Gamma^B_{\hk}(X_{\overline{K},h})\{r\}) \stackrel{\sim}{\to}
  C(\R\Gamma^B_{\hk}(X_{\overline{K},h})\{r\})^{G_K}
 \end{align*}
 is a quasi-isomorphism (recall that taking $G_K$-fixed points corresponds to taking global sections on the pro-\'etale site). In particular, that the term on the right hand side makes sense. To see this, 
 it suffices to show that the canonical maps
 \begin{align*}
  (\R\Gamma_{\dr}(X_{\overline{K},h})/F^r)^{G_K} & \stackrel{\sim}{\to}  ((\R\Gamma_{\dr}(X_{\overline{K},h})\otimes_{\ovk}B_{\dr})/F^r)^{G_K},\\
  \R\Gamma^B_{\hk}(X_{\overline{K},h})^{G_K} &  \stackrel{\sim}{\to}  (\R\Gamma^B_{\hk}(X_{\overline{K},h})\otimes_{K_0^{\nr}}B_{\st})^{G_K}
\end{align*}
 are quasi-isomorphisms and to use the fact that the action of $G_K$ on $\R\Gamma_{\hk}^B(X_{\ovk,h})$ is smooth. The fact that the first map is a quasi-isomorphism  follows from the filtered quasi-isomorphism 
 $\R\Gamma_{\dr}(X)\otimes_K\ovk\stackrel{\sim}{\to}\R\Gamma_{\dr}(X_{\overline{K},h})$ and the fact that $B_{\dr}^{G_K}=K$. Similarly, the second map is a quasi-isomorphism because, by \cite[4.2.4]{F1}, $\R\Gamma^B_{\hk}(X_{\overline{K},h})$ is the subcomplex of those elements of $\R\Gamma^B_{\hk}(X_{\overline{K},h})\otimes_{K^{\nr}_0}B_{\st}$ whose stabilizers in $G_K$ are open. 
 
  Taking  the $G_K$-fixed points of the above Postnikov system  we get an exact couple 
$$
{}^{\pst} D_1^{i,j}=H^j(X^i_{\pst}),\quad {}^{\pst}E_1^{i,j}=H^j(Y^i_{\pst}) \Longrightarrow H^{i+j}(X_{\pst})
$$
corresponding to the Bousfield-Kan filtration of the complex $X_{\pst}$. On the other hand, applying $\R\Gamma(\Spec(K)_{\proeet},\cdot)$
to the same Postnikov system we obtain an exact couple
$$
{}^{I}D_1^{i,j}=H^j(\Spec(K)_{\proeet},X^i),\quad {}^{I}E_1^{i,j}=H^j(\Spec(K)_{\proeet},Y^i) \Longrightarrow H^{i+j}(\Spec(K)_{\proeet},X)
$$
together with  a natural map of exact couples $({}^{\pst} D_1^{i,j}, {}^{\pst}E_1^{i,j})\to ({}^{I}D_1^{i,j}, {}^{I}E_1^{i,j})$.

  We also have the hypercohomology exact couple
$$
{}^{II}D_2^{i,j}=H^{i+j}(\Spec(K)_{\proeet},\tau_{\leq j-1}X),\quad {}^{II}E_2^{i,j}=H^i(\Spec(K)_{\proeet},\sh^j(X)) \Longrightarrow H^{i+j}(\Spec(K)_{\proeet},X)
$$  
Theorem \ref{speccomp} gives us  a natural morphism of exact couples $({}^{I}D^{i,j}_2,{}^{I}E^{i,j}_2)\to ({}^{II}D^{i,j}_2,{}^{II}E^{i,j}_2)$ -- hence a natural morphism of spectral sequences ${}^{I}E_2^{i,j}\to {}^{II}E_2^{i,j}$ compatible with the identity map on the common abutment -- if our original Postnikov system satisfies the equivalent conditions (\ref{ass}). We will check the condition (4), i.e., that the following long sequence is exact for all $j$ 
$$
0\to \sh^j(X)\to \sh^j(Y^0)\to \sh^j(Y^1)\to \sh^j(Y^2)\to 0
$$
For that it is enough to show that 
\begin{enumerate}
\item $\sh^j(\nu Y^i)\simeq \nu H^j(Y^i)$, for $i=0,1,2$;
\item $\sh^j(\nu X)\simeq \nu H^j(X)$;
\item   the following long sequence of $G_K$-modules
$$
0\to H^j(X)\to H^j(Y^0)\to H^j(Y^1)\to H^j(Y^2)\to 0
$$
is exact;
\item the pullback $\nu$ preserves its exactness. 
\end{enumerate}
The assertion in (1) was shown above.
The sequence in (3) is equal to the top sequence in the following commutative diagram (where we set $M=H^j_{\hk}(X_{\ovk,h})$, $M_{\dr}=H^j_{\dr}(X_{\ovk,h})$, $E= H^j(X_{\ovk,\eet},\bq_p)$).
$$
\xymatrix@C=40pt{
 H^j(X)\ar[d]^{\alpha_{\eet}^{-1}}_{\wr}\ar[r] & M\otimes_{K_0^{\nr}} B_{\st}\ar[r]^-{(N,1-\phi_r,\iota)}\ar[d]^{\rho_{\hk}}_{\wr} & M\otimes_{K_0^{\nr}}( B_{\st} \oplus  B_{\st})\oplus (M_{\dr}\otimes_{\ovk} B_{\dr})/F^r\ar[r]^-{(1-\phi_{r-1})-N}\ar[d]^{\rho_{\hk}+\rho_{\hk}+\rho_{\dr}}_{\wr}& M\otimes_{K_0^{\nr}} B_{\st}\ar[d]^{\rho_{\hk}}_{\wr} \\
E(r)\ar@{^{(}->}[r] & E\otimes B_{\st}\ar[r]^-{(N,1-\phi_r,\iota)} & E\otimes (B_{\st}\oplus B_{\st})\oplus E\otimes B_{\dr}/F^r\ar@{->>}[r]^-{(1-\phi_{r-1})-N} &
E\otimes B_{\st}
}
$$
Since the bottom sequence is just a fundamental exact sequence of $p$-adic Hodge Theory, the top sequence is exact, as wanted.

  To prove assertion (4), we pass to the bottom exact sequence above and apply $\nu$ to it. It is easy to see that it enough now to show that the following surjections have continuous $\bq_p$-linear sections
\begin{align*}
B_{\st}\stackrel{N}{\to}B_{\st},\quad B_{\crr}\verylomapr{(1-\phi_r,\can)}B_{\crr}\oplus B_{\dr}/F^r.
\end{align*}
For the monodromy, write $B_{\st}=B_{\crr}[u_s]$ and take for a continuous section the map induced by $bu_s^i\mapsto -(b/(i+1))u_s^{i+1}$, $b\in B_{\crr}$. For the second map, the existence of continuous section was proved in \cite[1.18]{BK}. For a different argument: observe that an analogous statement was proved in \cite[Prop. II.3.1]{Col} with $B_{\max}$ in place of $B_{\crr}$ as a consequence of the general theory of $p$-adic Banach spaces. We will just modify it here. Write $A_i=t^{-i}B^+_{\crr}$ and $B_i=t^{-i}B^+_{\crr}\oplus t^{-i}B^+_{\dr}/t^r$ for $i\geq 1$. These are $p$-adic Banach spaces. Observe  that $B_i\subset B_{i+1}$ is closed. Indeed, it is enough to show that $tB^+_{\crr}\subset B^+_{\crr}$ is closed. But we have
$tB^+_{\crr}=\bigcap_{n\geq 0}\ker(\theta\circ \phi^n)$.

   It follows \cite[Prop. I.1.5]{Col} that we can find a closed complement $C_{i+1}$ of $B_i$ in $B_{i+1}$. Set $f=(1-\phi_r,\can):B_{\crr}\to B_{\crr}\oplus B_{\dr}/F^r$. We know that $f$ maps $A_i$  onto $B_i$. Write 
$t^{-i}B^+_{\crr}\oplus t^{-i}B^+_{\dr}/t^r=B_1\oplus (\oplus_{j=2}^{i-1}C_j)$. By \cite[Prop. I.1.5]{Col}, we can find a continuous section $s_1: B_1\to A_1$ of $f$ and, if $i\geq 2$, a continuous section $s_i: C_i\to A_i$ of $f$.  Define the map $s: t^{-i}B^+_{\crr}\oplus t^{-i}B^+_{\dr}/t^r \to B_{\crr}$ by  $s_1$ on $B_1$ and by $s_i$ on $C_i$ for $i\geq 2$. Taking inductive limit over $i$ we get our section of $f$.

  To prove assertion (2), take a perfect representative of the complex $\R\Gamma(X_{\ovk, \eet},{\mathbf Z}_p(r))$. Consider the complex $Z=\R\Gamma(X_{\ovk, \eet},\bq_p(r))$ as a complex of sheaves on $\Spec(K)_{\proeet}$. As before, we see that this makes sense  and we easily find that (canonically)
$\sh^j(Z)\simeq \nu H^j(X_{\ovk,\eet},\bq_p(r))$. To prove (2), it is enough to show  that we can also pass with the map $\alpha_{\eet}: \R\Gamma(X_{\ovk,\eet},\bq_p(r))\stackrel{\sim}{\to} C(\R\Gamma^B_{\hk}(X_{\ovk,h})\{r\})$ to the site $\Spec(K)_{\proeet}$. Looking at its definition (cf. (\ref{definition1})) we see that we need to show that the period quasi-isomorphisms $\rho_{\crr}, \rho_{\hk}, \rho_{\dr}$ as well as the quasi-isomorphism
$$
\bq_p(r)\stackrel{\sim}{\to}[ B_{\st}\verylomapr{(N,1-\phi_r,\iota)}  B_{\st}\oplus B_{\st}\oplus B_{\dr}/F^r\veryverylomapr{(1-\phi_{r-1})-N}B_{\st}]
$$
can be lifted to the pro-\'etale site. The last fact we have just shown. For the crystalline period map $\rho_{\crr}$ this follows from the fact that it is defined integrally and all the relevant complexes are perfect. For the Hyodo-Kato period map $\rho_{\hk}$ - it follows from the case of $\rho_{\crr}$ and from perfection of complexes involved in the definition of the Beilinson-Hyodo-Kato map. For the de Rham period map $\rho_{\dr}$ this follows from perfection of the involved complexes as well as from the exactness of $\holim_n$ (in the definition of $\rho_{\dr}$) on the pro-\'etale site of $K$ (cf. \cite[3.18]{Sch}).

 We define the map of spectral sequences $\delta:=(\delta_D,\delta):=({}^{\pst} D_2^{i,j}, {}^{\pst}E_2^{i,j})\to ({}^{II}D^{i,j}_2,{}^{II}E^{i,j}_2)$ -- that we stated at the beginning of the proof -- as the composition of the  two maps constructed above $$\delta:\quad 
({}^{\pst} D_2^{i,j}, {}^{\pst}E_2^{i,j})\to({}^{I}D^{i,j}_2,{}^{I}E^{i,j}_2)\to ({}^{II}D^{i,j}_2,{}^{II}E^{i,j}_2).$$
To get the spectral sequence from the theorem we need to pass from $^{II}E_2$ to  the Hochschild-Serre spectral sequence. To do that consider the hypercohomology exact couple
$$
{}^{\eet}D_2^{i,j}=H^{i+j}(\Spec(K)_{\proeet},\tau_{\leq j-1}Z),\quad {}^{\eet}E_2^{i,j}=H^i(\Spec(K)_{\proeet},\sh^j(Z)) \Longrightarrow H^{i+j}(\Spec(K)_{\proeet},Z)
$$  
and, via $\alpha_{\eet}^{-1}$,   a natural morphism of exact couples $({}^{II}D^{i,j}_2,{}^{II}E^{i,j}_2)\to ({}^{\eet}D^{i,j}_2,{}^{\eet}E^{i,j}_2)$ -- hence a natural morphism of spectral sequences ${}^{II}E_2^{i,j}\to {}^{\eet}E_2^{i,j}$ compatible with the map $\alpha_{\eet}^{-1}$ on the abutment. We have a quasi-isomorphism $\psi: \R\Gamma(\Spec(K)_{\proeet},Z)\stackrel{\sim}{\to} \R\Gamma(X_{\eet},\bq_p(r))$ defined as the composition
\begin{align*}
\psi:\quad \R\Gamma(\Spec(K)_{\proeet},\R\Gamma(X_{\ovk,\eet},\bq_p(r))) & \stackrel{\sim}{\to}\bq\otimes\holim_n
\R\Gamma(G_K,\R\Gamma (X_{\ovk,\eet},{\mathbf Z}/p^n(r)))\\ & =\bq\otimes \holim_n \R\Gamma(X_{\eet},{\mathbf Z}/p^n(r))=\R\Gamma (X_{\eet},\bq(r))
\end{align*}

  We have obtained  the following natural maps of spectral sequences
$$
\xymatrix{^{\synt}E^{i,j}_2=H^i_{\st}(G_K,H^j(X_{\ovk,\eet},{\mathbf Q}_p(r)))\ar[d]_{\wr}\ar@{=>}[r] & H^{i+j}_{\synt}(X_{h},r)\ar[d]^{\alpha_{\synt}}_{\wr}\\
 E^{i,j}_2=H^i(C_{\st}(H^j_{\hk}(X_{h})\{r\}))\ar[d]^{\alpha_{\eet}^{-1}\delta\varepsilon^*}\ar@{=>}[r] & H^{i+j}(C_{\st}(\R\Gamma^B_{\hk}(X_{h})\{r\}))\ar[d]^{\psi \alpha_{\eet}^{-1}\delta \varepsilon^*}\\
^{\eet}E^{i,j}_2=H^i(G_K,H^j(X_{\ovk,\eet},{\mathbf Q}_p(r)))\ar@{=>}[r] & H^{i+j}(X_{\eet},{\mathbf Q}_p(r))
}
$$  
It remains to show that the right vertical composition $\gamma: H^{i+j}_{\synt}(X_{h},r)\to H^{i+j}(X_{\eet},{\mathbf Q}_p(r))$ is equal to the map $\rho_{\synt}$. Since we have the equality $\alpha_{\synt}=\rho_{\synt}\alpha_{\eet}$ (in the derived category) from (\ref{compatibility0}) and, by Lemma \ref{compatibilit01}, $\varepsilon^*\alpha_{\synt}=\alpha_{\synt}\varepsilon^*$, the map $\gamma$  can be written as the composition
\begin{align*}
\tilde{\rho}_{\synt}:\quad H^{i+j}_{\synt}(X_{h},r)\stackrel{\varepsilon^*}{\to}H^{i+j}(\Spec(K)_{\proeet},\nu \R\Gamma_{\synt}(X_{\ovk,h},r)) & \stackrel{\rho_{\synt}}{\to}
H^{i+j}(\Spec(K)_{\proeet},\nu\R\Gamma (X_{\ovk,\eet},\bq_p(r)))\\
 & \stackrel{\psi}{\to} H^{i+j}(X_{\eet},\bq_p(r)),
\end{align*}
where the period map $\rho_{\synt}$ is understood to be on sheaves on $\Spec(K)_{\proeet}$. There is no problem with that since we care only about the induced map on cohomology groups.  
It is easy now to see that $\tilde{\rho}_{\synt}=\rho_{\synt}$, as wanted.
\end{proof}
\begin{remark}
If $X$ is proper and smooth, it is known that the \'etale Hochschild-Serre spectral sequence degenerates, i.e., that ${}^{\eet}E_2={}^{\eet}E_{\infty}$. 
It is very likely that so does the syntomic descent spectral sequence in this case, i.e., that ${}^{\synt}E_2={}^{\synt}E_{\infty}$\footnote{This was, in fact, shown in \cite{DN}.}. 
\end{remark}
\begin{corollary}
\label{BK1}
For  $X\in {\mathcal V}ar_{K}$, we have a canonical quasi-isomorphism
$$\rho_{\synt}: \quad\tau_{\leq r}\R\Gamma_{\synt}(X_{h},r)_{\mathbf Q}\stackrel{\sim}{\to} \tau_{\leq r}\R\Gamma(X_{\eet},{\mathbf Q}_p(r)).
$$
\end{corollary}
\begin{proof} By Theorem \ref{stHS},  the syntomic descent and the Hochschild-Serre spectral sequence are compatible.
We have  $D_j=H^j_{\hk}(X_{\ovk,h})\{r\}\in MF_K^{\ad}(\phi, N,G_K)$. For $j\leq r$, $F^{1}D_{j,K}=F^{1-(r-j)}H^j_{\dr}(X_{h})=0$. Hence, by  Proposition \ref{resolution33},  
we have $^{\synt}E^{i,j}_2\stackrel{\sim}{\to} {}^{\eet}E^{i,j}_2$. This implies that 
 $\rho_{\synt}: \tau_{\leq r}\R\Gamma_{\synt}(X_h,r)\stackrel{\sim}{\to} \tau_{\leq r}\R\Gamma(X_{\eet},{\mathbf Q}_p(r))$, as wanted.
\end{proof}
\begin{remark}
All of the above automatically extends to finite diagrams of $K$-varieties, hence to essentially finite diagrams of $K$-varieties (i.e., the diagrams for which every truncation of their cohomology $\tau_{\leq n}$ is computed by truncating the cohomology of some finite diagram). This includes, in particular, simplicial and cubical varieties.
\end{remark}

\begin{proposition}
\label{compBK}
Let $X\in{\mathcal V}ar_K$ and $i\geq 0$. The composition
\begin{align*}
H^q_{\dr}(X)/F^r \stackrel{\partial}{\to} H^{q+1}_{\synt}(X_h,r)
\stackrel{\rho_{\synt}}{\longrightarrow} H^{q+1}_{\eet}(X,\Qp(r))\to
H^{q+1}_{\eet}(X_{\ovk},\Qp(r))
\end{align*}
is the zero map. The map induced by the syntomic descent spectral sequence
$$
H^q_{\dr}(X)/F^r\to H^1(G_K,H^q_{\eet}(X_{\ovk},\Qp(r)))
$$
is equal to the Bloch-Kato exponential associated with the Galois
representation $V^q(r) = H^q_{\eet}(X_{\ovk},\Qp(r))$.
\end{proposition}

\begin{proof}In what follows we will omit the passage to the pro-\'etale site. 
Consider the Postnikov system from the proof of Theorem \ref{stHS},
which arises from the complex $X=C(\R\Gamma^B_{\hk}(X_{\ovk,h})\{r\})$; then  $Y^p = C^p(\R\Gamma^B_{\hk}(X_{\ovk,h})\{r\})$. The discussion
from Example \ref{filtered} then applies to the functor $f(-) = (-)^{G_K}$
and yields the following four exact couples.

\noindent
(1) $D_1^{p,q} = H^q(X^p)$, $E_1^{p,q} = H^q(Y^p) =
C^p(H^q_{\hk}(X_{\ovk,h})\{r\})) = C^p(H^q_{\hk}\{r\})$. The corresponding
quasi-isomorphism
$H^q(X) \stackrel{\sim}{\to} E_1^{{\scriptscriptstyle\bullet},q}$
is then identified, via the various period maps, with
$$ V^q(r) \stackrel{\sim}{\to} C(H^q_{\hk}\{r\}) = C(D_{\pst}(V^q(r))).$$
(2) ${}^f D_1^{p,q} = H^q(f(X^p))$, ${}^f E_1^{p,q} = H^q(f(Y^p)) =
f(H^q(Y^p)) = C^p_{\st}(H^q_{\hk}\{r\}) = f(E_1^{p,q})$.

\noindent
(3) ${}^I D_1^{p,q} = (\R^q f)(X^p)$, ${}^I E_1^{p,q} = (\R^q f)(Y^p)$.

\noindent
(4) ${}^{II} D_2^{p,q} = (\R^{p+q}f)(\tau_{\leq q-1} X)$,
${}^{II} E_2^{p,q} = (\R^p f)(H^q(X)) = H^p(G_K, V^q(r))$.

\noindent
There is a canonical morphism of exact couples (2) $\to$ (3) and
a morphism (3) $\to$ (4) given by the maps $(u, v)$ from the proof
of Theorem \ref {speccomp}. As observed in \ref {BKexp}, the
Bloch-Kato exponential for $V = V^q(r)$ is obtained by applying
$R^0 f$ to
\begin{align*}
Z^1 C(H^q_{\hk}\{r\})  & = Z^1(E_1^{{\scriptscriptstyle\bullet},q})
\stackrel{\can}{\longrightarrow}
(\sigma_{\geq 1} C(H^q_{\hk}\{r\}))[1] =
(\sigma_{\geq 1} C(E_1^{{\scriptscriptstyle\bullet},q})) [1]
\stackrel{-\can}{\longrightarrow} C(H^q_{\hk}\{r\}))[1] =
E_1^{{\scriptscriptstyle\bullet},q} [1]\\
 & \stackrel{\sim}{\leftarrow} V^q(r) [1] = H^q(X) [1],
\end{align*}
hence is equal to the composite map
$$
f(Z^1(E_1^{{\scriptscriptstyle\bullet},q})) =
Z^1({}^f E_1^{{\scriptscriptstyle\bullet},q}) \to
{}^f E_2^{1,q} \stackrel{\can}{\longrightarrow}
{}^I E_2^{1,q} \stackrel{- v^\prime = v}{\longrightarrow}
(R^1 f)(E_1^{{\scriptscriptstyle\bullet},q}) = {}^{II} E_2^{p,q},
$$
which coincides, in turn, with
$$ Z^1 C_{\st}(H^q_{\hk}\{r\}) \stackrel{\can}{\longrightarrow}
H^1_{\st}(G_K, V^q(r)) \to H^1(G_K, V^q(r)).$$
After restricting to the de Rham part of $Z^1 C(H^q_{\hk}\{r\})$
we obtain the desired statement about $H^q_{\dr}(X)/F^r$.
\end{proof}
In more concrete terms, the above proposition says that the following diagram commutes
$$
\xymatrix{
H^{q+1}_{\synt}(X_h,r)_0\ar[r]^{\rho_{\synt}} &  H^{q+1}_{\eet}(X,\Qp(r))_0\ar[d]\\
H^q_{\dr}(X)/F^r\ar[u]^{\partial}\ar[r]^-{\exp_{\bk}} & H^1(G_K,H^q_{\eet}(X_{\ovk},\Qp(r))),
}
$$
where the subscript $0$ refers to the classes that vanish in $H^{q+1}_{\synt}(X_{\ovk,h},r)$ and $H^{q+1}_{\eet}(X_{\ovk},\Qp(r))$, respectively.
\begin{remark}
Assume that  $r>q$. Then in the above diagram all the maps are isomorphisms. Indeed,  we have $F^rH^q_{\dr}(X)=0$. By \cite[Theorem 6.8]{BER}, the map $\exp_{\bk}$ is  an isomorphism. By Proposition \ref{BK2} and Corollary \ref{BK1}, so is the period map $\rho_{\synt}$. Since, by Theorem \ref{main}, $H^2(G_K,H^{q}_{\eet}(X_{\ovk},\Qp(r)))=H^2(G_K,H^{q-1}_{\eet}(X_{\ovk},\Qp(r)))=0$, the vertical map is an isomorphism as well. Hence so is the map $\partial$.
\end{remark}
\section{Syntomic regulators}
In this section we prove that Soul\'e's \'etale regulators land in the semistable Selmer groups. This will be done by constructing syntomic regulators that are compatible with the \'etale ones via the period map and by exploiting the syntomic descent spectral sequence. 
\subsection{Construction of syntomic Chern classes}
We start with the construction of syntomic Chern classes. This will be standard once we prove that syntomic cohomology satisfies projective space theorem and homotopy property. 

  In this subsection we will work in the (classical) derived category. For a fine log-scheme $(X,M)$, log-smooth over $V^{\times}$, we have the  log-crystalline and log-syntomic  first Chern class maps of complexes of sheaves on $X_{\eet}$ \cite[2.2.3]{Ts}
\begin{align*}
c_1^{\crr}:   j_*\so^*_{X_{\tr}}\stackrel{\sim}{\to} M^{\gp}\to M^{\gp}_n\to R\varepsilon_*\sj^{[1]}_{X_n/W_n(k)} [1], & \quad
 c_1^{\st}:   j_*\so^*_{X_{\tr}}\stackrel{\sim}{\to} M^{\gp}\to M^{\gp}_n\to R\varepsilon_*\sj^{[1]}_{X_n/R_n}[1],\\
 c_1^{\hk}:   j_*\so^*_{X_{\tr}}\stackrel{\sim}{\to} M^{\gp}\to M^{\gp}_0\to R\varepsilon_*\sj^{[1]}_{X_0/W_n(k)^0}[1], & \quad c_1^{\synt}:   j_*\so^*_{X_{\tr}}\stackrel{\sim}{\to} M^{\gp}\to \sss(1)_{X,\bq} [1].
\end{align*}
Here $\varepsilon$ is the projection from the corresponding crystalline site to the \'etale site. The maps $c_1^{\crr}, c_1^{\st},$ and $c_1^{\synt}$ are clearly compatible. So are the maps $c_1^{\st}$ and $ c_1^{\hk}$.
For ss-pairs $(U,\overline{U})$ over $K$, we get the induced  functorial maps
\begin{align*}
c_1^{\crr}:\Gamma(U,\so_U^*)  \stackrel{\sim}{\leftarrow}\Gamma(\overline{U},j_*\so_U^*)\to \R\Gamma_{\crr}(U,\overline{U},
\sj^{[1]})[1], & \quad c_1^{\st}:\Gamma(U,\so_U^*)  \to \R\Gamma_{\crr}((U,\overline{U})/R,
 \sj^{[1]})[1],\\
c_1^{\hk}:\Gamma(U,\so_U^*) \to \R\Gamma_{\crr}((U,\overline{U})_0/W_n(k)^0,\sj^{[1]})[1], &
\quad c_1^{\synt}:   \Gamma(U,\so_U^*)  \to \R\Gamma_{\synt}(U,\overline{U},1)_{\bq}[1].
\end{align*}

  For $X\in {\mathcal V}ar_K$ we can glue the absolute log-crystalline and log-syntomic classes to obtain the absolute crystalline and syntomic first Chern class maps
$$c_1^{\crr}: \so^*_{X_h}\to \sj_{\crr,X}[1],\quad
c_1^{\synt}: \so^*_{X_h}\to \sss(1)_{X,\bq}[1].
$$
They  induce  (compatible) maps
\begin{align*}
c_1^{\crr}: \Pic(X) & =H^1(X_{\eet},\so^*_X)\to H^1(X_{h},\so^*_X)\stackrel{c_1^{\crr}}{\to }H^2(X_h,\sj_{\crr}),\\
c_1^{\synt}: \Pic(X) & =H^1(X_{\eet},\so^*_X)\to H^1(X_{h},\so^*_X)\stackrel{c_1^{\synt}}{\to }H^2_{\synt}(X_h,1).
\end{align*}
 Recall that, for a log-scheme $(X,M)$ as above,  we also have the log de Rham first Chern class  map $$c_1^{\dr}: j_*\so^*_{X_{\tr}}\stackrel{\sim}{\to} M^{\gp}\to M^{\gp}_n\stackrel{\dlog}{\to} \Omega^{\scriptscriptstyle\bullet}_{(X,M)_n/V^{\times}_n} [1].$$
 For ss-pairs $(U,\overline{U})$ over $K$, it induces  maps
 $$c_1^{\dr}:\Gamma(U,\so_U^*)  \stackrel{\sim}{\leftarrow}\Gamma(\overline{U},j_*\so_U^*)\to \R\Gamma(\overline{U},\Omega^{\scriptscriptstyle\bullet}_{(U,\overline{U})/V^{\times}} )[1].\\
 $$
 By the map $\R\Gamma_{\crr}(U,\overline{U},\sj^{[1]})\to \R\Gamma_{\crr}(U,\overline{U})\to \R\Gamma(\overline{U},\Omega^{\scriptscriptstyle\bullet}_{(U,\overline{U})/V^{\times}} )$ they are compatible with the absolute  log-crystalline and log-syntomic classes  \cite[2.2.3]{Ts}.
 \begin{lemma}
 \label{compatibility}
 For strict ss-pairs $(U,\overline{U})$ over $K$, the Hyodo-Kato map and the Hyodo-Kato isomorphism
 $$
 \iota: H^2_{\hk}(U,\overline{U})_{\mathbf Q}\to H^2_{\crr}((U,\overline{U})/R)_{\mathbf Q},\quad
  \iota_{\dr,\pi}: H^2_{\hk}(U,\overline{U})_{\mathbf Q}\otimes_{K_0}K\stackrel{\sim}{\to} H^2(\overline{U}_{K},\Omega^{\scriptscriptstyle\bullet}_{(U,\overline{U}_{K})/K})
  $$
 are compatible with   first Chern class maps.
  \end{lemma}
 \begin{proof}
 Since $\iota_{\dr,\pi}=i^*_{\pi}\iota\otimes\id$ and the map $i^*_{\pi}$ is compatible with first Chern classes, it suffices to show the compatibility for the Hyodo-Kato map $\iota$. Let $\sll$ be a line bundle on $U$. Since the map $\iota$ is a section of the map $i^*_0:H^2_{\crr}((U,\overline{U})/R)_{\mathbf Q}\to H^2_{\hk}(U,\overline{U})_{\mathbf Q}$ and the  map $i^*_0$ is compatible with first Chern classes, we have that the element  $\zeta\in H^2_{\crr}((U,\overline{U})/R)_{\mathbf Q}$ defined as  $\zeta=\iota(c_1^{\hk}(\sll))-c_1^{\st}(\sll)$ lies in  $ TH^2_{\crr}((U,\overline{U})/R)_{\mathbf Q}$. Hence $\zeta=T\gamma$. Since the map $\iota$ is compatible with Frobenius and $\phi(c_1^{\hk}(\sll))=pc_1^{\hk}(\sll)$, $\phi(c_1^{\st}(\sll))=pc_1^{\st}(\sll)$, we have $\phi(\zeta)=p\zeta$. Since $\phi(T\gamma)=T^p\phi(\gamma)$ this implies that $\gamma\in\bigcap_{n=1}^{\infty}T^nH^2_{\crr}((U,\overline{U})/R)_{\mathbf Q}$, which  is not possible unless $\gamma$ (and hence $\zeta$) are zero. But this is what we wanted to show.
 \end{proof}

  We  have the following projective space theorem for syntomic cohomology.
\begin{proposition}
\label{projective}
Let $\se$ be a locally free sheaf of rank $d+1$, $d\geq 0$, on a scheme $X\in {\mathcal V}ar_{K}$. Consider the associated projective bundle $\pi:{\mathbb P}(\se)\to X$.  Then we have the following quasi-isomorphism of complexes of sheaves on $X_h$
\begin{align*}
\bigoplus_{i=0}^d{c}_1^{\synt}(\so(1))^i\cup\pi^*: \quad
\bigoplus_{i=0}^d \sss(r-i)_{X,{\mathbf Q}}[-2i] \stackrel{\sim}{\to} R\pi_*
\sss(r)_{{\mathbb P}(\se),{\mathbf Q}}, \quad 0\leq d \leq r.
\end{align*}
Here, the class ${c}_1^{\synt}(\so(1))\in H^2_{\synt}({\mathbb P}(\se)_h, 1)$ refers to the class of the tautological bundle on ${\mathbb P}(\se)$.
\end{proposition}
\begin{proof}
By (tedious) checking of many compatibilities we will reduce the above projective space theorem to the projective space theorems for the Hyodo-Kato and the filtered de Rham cohomologies.

 To prove our proposition it suffices to show that for any ss-pair $(U,\overline{U})$ over $K$ and the projective space $\pi: {\mathbb P}^d_{\overline{U}}\to\overline{U} $ of dimension $d$ over $\overline{U}$ we have a  projective space theorem for syntomic cohomology ($a\geq 0$)
$$
\bigoplus_{i=0}^d{c}_1^{\synt}(\so(1))^i\cup\pi^*: \quad
\bigoplus_{i=0}^d H^{a-2i}_{\synt}(U_h,r-i)\stackrel{\sim}{\to} H^a_{\synt}({\mathbb P}^d_{U,h},r), \quad 0\leq d \leq r.
$$
By Proposition \ref{hypercov} and the compatibility of the maps $ H^{*}_{\synt}(U,\overline{U},j)_{{\mathbf Q}}\stackrel{\sim}{\to} H^{*}_{\synt}(U_h,j)_{{\mathbf Q}}$ with products and first Chern classes, this reduces to proving a projective space theorem for log-syntomic cohomology, i.e.,  a quasi-isomorphism of complexes
\begin{align*}
\bigoplus_{i=0}^d{c}_1^{\synt}(\so(1))^i\cup\pi^*: \quad
\bigoplus_{i=0}^d H^{a-2i}_{\synt}(U,\overline{U},r-i)_{{\mathbf Q}}\stackrel{\sim}{\to} H^a_{\synt}({\mathbb P}^d_{U},{\mathbb P}^d_{\overline{U}},r)_{\mathbf Q}, \quad 0\leq d \leq r,
\end{align*}
where the class ${c}_1^{\synt}(\so(1))\in H^2_{\synt}({\mathbb P}^d_{U},{\mathbb P}^d_{\overline{U}}, 1)$ refers to the class of the tautological bundle on ${\mathbb P}^d_{\overline{U}}$.

  By the distinguished triangle 
 \begin{equation*}
{\mathrm R}\Gamma_{\synt}(U,\overline{U},r)_{{\mathbf Q}} \to
\R\Gamma_{\crr}(U,\overline{U},r)_{{\mathbf Q}}\stackrel{}{\to} \R\Gamma_{\dr}(U,\overline{U}_{K})/F^r
\end{equation*} and its compatibility with the action of $c_1^{\synt}$, it suffices to prove the following two quasi-isomorphisms for the twisted absolute log-crystalline complexes and for the filtered log de Rham complexes ($0\leq d \leq r$)
\begin{align*}
\bigoplus_{i=0}^d{c}_1^{\crr}(\so(1))^i\cup\pi^*:  &\quad
\bigoplus_{i=0}^d H^{a-2i}_{\crr}(U,\overline{U},r-i)_{{\mathbf Q}}  \stackrel{\sim}{\to} H^a_{\crr}({\mathbb P}^d_{U},{\mathbb P}^d_{\overline{U}},r)_{\mathbf Q}, \\
\bigoplus_{i=0}^d{c}_1^{\dr}(\so(1))^i\cup\pi^*: & \quad
\bigoplus_{i=0}^d F^{r-i}H^{a-2i}_{\dr}(U,\overline{U}_K) \stackrel{\sim}{\to} F^{r}H^a_{\dr}({\mathbb P}^d_{U},{\mathbb P}^d_{\overline{U}_K}).
\end{align*}
 For the log de Rham cohomology, notice that the above map is quasi-isomorphic to the map \cite[3.2]{BE1}
 $$\bigoplus_{i=0}^d{c}_1^{\dr}(\so(1))^i\cup\pi^*:  \quad
\bigoplus_{i=0}^d F^{r-i}H^{a-2i}_{\dr}(U) \stackrel{\sim}{\to} F^{r}H^a_{\dr}({\mathbb P}^d_{U}). $$
 Hence well-known to be a quasi-isomorphism.

  For the twisted log-crystalline cohomology, notice that since Frobenius behaves well with respect to ${c}_1^{\crr}$, it suffices to prove a projective space theorem for the absolute log-crystalline cohomology $H^{*}_{\crr}(U,\overline{U})_{{\mathbf Q}} $.
 \begin{equation*}
\bigoplus_{i=0}^d{c}_1^{\crr}(\so(1))^i\cup\pi^*:  \quad
\bigoplus_{i=0}^d H^{a-2i}_{\crr}(U,\overline{U})_{{\mathbf Q}}  \stackrel{\sim}{\to} H^a_{\crr}({\mathbb P}^d_{U},{\mathbb P}^d_{\overline{U}})_{\mathbf Q}
\end{equation*}
Without loss of generality we may assume that the pair $(U,\overline{U})$ is split over $K$. By the  distinguished triangle  
\begin{equation*}
\R\Gamma_{\crr}(U,\overline{U})\to \R\Gamma_{\crr}((U,\overline{U})/R)\stackrel{N}{\to}
\R\Gamma_{\crr}((U,\overline{U})/R))
\end{equation*}
and its compatibility with the action of $c^{\crr}_1(\so(1))$ (cf. \cite[Lemma 4.3.7]{Ts}),
 it suffices to prove a projective space theorem for the  log-crystalline cohomology $H^{*}_{\crr}((U,\overline{U})/R)_{{\mathbf Q}} $. Since the $R$-linear isomorphism  $\iota:  H^{*}_{\hk}(U,\overline{U})_{{\mathbf Q}} \otimes R_{\mathbf Q}\stackrel{\sim}{\to}H^{*}_{\crr}((U,\overline{U})/R)_{{\mathbf Q}} $ is compatible with products \cite[Prop. 4.4.9]{Ts} and first Chern classes (cf. Lemma \ref{compatibility}) we reduce the problem to showing the projective space theorem for the Hyodo-Kato cohomology.
\begin{equation*}
\bigoplus_{i=0}^d{c}_1^{\hk}(\so(1))^i\cup\pi^*: \quad
\bigoplus_{i=0}^d H^{a-2i}_{\hk}(U,\overline{U})_{{\mathbf Q}}  \stackrel{\sim}{\to} H^a_{\hk}({\mathbb P}^d_{U},{\mathbb P}^d_{\overline{U}})_{\mathbf Q}
\end{equation*}
Tensoring by $K$ and using the isomorphism $ \iota_{\dr,\pi}: H^{*}_{\hk}(U,\overline{U})_{{\mathbf Q}} \otimes_{K_0} K\stackrel{\sim}{\to}H^{*}_{\dr}(U,\overline{U}_K)$ that is compatible with products \cite[Cor. 4.4.13]{Ts} and first Chern classes (cf. Lemma \ref{compatibility}) we reduce to checking the projective space theorem for the log de Rham cohomology $ H^{*}_{\dr}(U,\overline{U}_K) $. And we have done this above.
\end{proof}
The above proof proves also the projective space theorem for the absolute crystalline cohomology.
\begin{corollary}
\label{projective1}
Let $\se$ be a locally free sheaf of rank $d+1$, $d\geq 0$, on a scheme $X\in {\mathcal V}ar_{K}$. Consider the associated projective bundle $\pi:{\mathbb P}(\se)\to X$.  Then we have the following quasi-isomorphism of complexes of sheaves on $X_h$
\begin{align*}
\bigoplus_{i=0}^d{c}_1^{\crr}(\so(1))^i\cup\pi^*: \quad
\bigoplus_{i=0}^d \sj^{[r-i]}_{X,{\mathbf Q}}[-2i] \stackrel{\sim}{\to} R\pi_*
\sj^{[r]}_{{\mathbb P}(\se),{\mathbf Q}}, \quad 0\leq d \leq r.
\end{align*}
Here, the class ${c}_1^{\crr}(\so(1))\in H^2({\mathbb P}(\se)_h, \sj_{\crr})$ refers to the class of the tautological bundle on ${\mathbb P}(\se)$.
\end{corollary}

  For $X\in {\mathcal V}ar_K$, using the projective space theorem (cf. Theorem \ref{projective}) and the Chern classes
$$
c_0^{\synt}: \bq_p\stackrel{\can}{\to} \sss(0)_{X_{\bq}},\quad c_1^{\synt}: \so_{X_h}^*\to \sss(1)_{X_{\bq}}[1],
$$
we obtain syntomic Chern classes $c_i^{\synt}(\se)$, for any locally free sheaf $\se$ on $X$.

  Syntomic cohomology has  homotopy invariance property. 
\begin{proposition}
\label{homotopy}
 Let $X\in {\mathcal V}ar_K$ and $f: {\mathbb A}^1_X \to X$ be the natural projection from the affine line over $X$ to $X$. Then, 
 for all $r\geq 0$, the pullback map 
$$f^*:\,\R\Gamma_{\synt}(X_h,r)\lomapr{\sim}\R\Gamma_{\synt}({\mathbb A}^1_{X,h},r)$$
is a quasi-isomorphism.
\end{proposition}
\begin{proof}
Localizing in the $h$-topology of $X$ we may assume that $X=U$ - the open set of an ss-pair $(U,\overline{U})$ over $K$.  Consider the following commutative diagram.
$$
\begin{CD}
\R\Gamma_{\synt}(U,\overline{U},r)_{\mathbf Q}@> f^*>>\R\Gamma_{\synt}({\mathbb A}^1_{U},{\mathbb P}^1_{\overline{U}},r)_{\mathbf Q}\\
@VV \wr V @VV\wr V\\
\R\Gamma_{\synt}(U_h,r) @> f^*>> \R\Gamma_{\synt}({\mathbb A}^1_{U,h},r)
\end{CD}
$$
The vertical maps are quasi-isomorphisms by Proposition \ref{hypercov}. It suffices thus to  show that the top horizontal map  is a quasi-isomorphism. By Proposition  \ref{reduction1}, this reduces to showing that 
the map $$C_{\st}(\R\Gamma_{\hk}(U,\overline{U})_{\mathbf Q}\{r\})\stackrel{ f^*}{\to}C_{\st}(\R\Gamma_{\hk}({\mathbb A}^1_{U},{\mathbb P}^1_{\overline{U}})_{\mathbf Q}\{r\})$$ is a quasi-isomorphism. Or, that the map $f: ({\mathbb A}^1_{U},{\mathbb P}^1_{\overline{U}})\to (U,\overline{U})$ induces a quasi-isomorphism on the Hyodo-Kato cohomology and a filtered quasi-isomorphism on the log de Rham cohomology:
$$\R\Gamma_{\hk}(U,\overline{U})_{\mathbf Q}\stackrel{ f^*}{\to}\R\Gamma_{\hk}({\mathbb A}^1_{U},{\mathbb P}^1_{\overline{U}})_{\mathbf Q},\quad 
\R\Gamma_{\dr}(U,\overline{U}_K)\stackrel{ f^*}{\to}\R\Gamma_{\dr}({\mathbb A}^1_{U},{\mathbb P}^1_{\overline{U}_K})
$$
Without loss of generality we may assume that the pair $(U,\overline{U})$ is split over $K$. Tensoring with $K$ and using the Hyodo-Kato quasi-isomorphism we reduce the Hyodo-Kato case to the log de Rham one. The latter  follows easily from the projective space theorem and the existence of the Gysin sequence in log de Rham cohomology. 
\end{proof}
\begin{remark}
The above implies that syntomic cohomology is a Bloch-Ogus theory. A proof of this fact was kindly communicated to us by Fr\'ed\'eric D\'eglise and is contained in Appendix~B, Proposition \ref{Bloch-Ogus}.
\end{remark}
\begin{proposition}For a scheme $X$, let $K_*(X)$ denote  Quillen's higher  $K$-theory  groups of $X$.  
 For $X\in {\mathcal V}ar_K$, $i,j\geq 0$, there are
 functorial syntomic Chern class maps
$$
c_{i,j}^{\synt}:  K_j(X)  \rightarrow H^{2i-j}_{\synt}(X_h,i).
$$
\end{proposition}
\begin{proof}
 Recall the
construction  of the  classes $c_{i,j}^{\synt}$.  
 First, one constructs universal classes
$C^{\synt}_{i,l}\in H^{2i}_{\synt}(B_{\scriptscriptstyle\bullet}GL_{l,h},i)$.
By a standard argument, projective space theorem and homotopy property show that  $$H^*_{\synt}(B_{\scriptscriptstyle\bullet}GL_{l,h},*)\simeq H^*_{\synt}(K,*)[x^{\synt}_1,\ldots,x^{\synt}_l],$$
 where the classes
$x^{\synt}_i\in H^{2i}_{\synt}(B_{\scriptscriptstyle\bullet}GL_{l,h},i)$ are the  syntomic Chern classes of the universal locally free sheaf on $B_{\scriptscriptstyle\bullet}GL_{l}$ (defined via a projective space theorem). 
For $l\geq i$, we define $$C^{\synt}_{i,l}=x^{\synt}_i\in
H^{2i}_{\synt}(B_{\scriptscriptstyle\bullet}GL_{l,h},i).$$ 
 The classes $C^{\synt}_{i,l}\in H^{2i}_{\synt}(B_{\scriptscriptstyle\bullet}GL_{l,h},i)$  yield compatible universal classes (see \cite[p. 221]{Gi})
$C^{\synt}_{i,l}\in H^{2i}_{\synt}(X,GL_l(\so_X),i)$, hence
 a natural map of pointed simplicial sheaves on $X_{\ZAR}$,
$C^{\synt}_i:B_{\scriptscriptstyle\bullet}GL(\so_X)\to \sk(2i,\sss^{\prime}(i){}_{X})$,
where $\sk $ is the Dold--Puppe functor of
$\tau_{\geq 0}\sss^{\prime}(i){}_{X}[2i]$
and $\sss^{\prime}(i){}_{X}$ is an
injective resolution of $\sss(i){}_{X}:=R\varepsilon_*\sss(i)_{\bq}$, $\varepsilon: X_h\to X_{\ZAR}$.
The characteristic classes $c^{\synt}_{i,j}$ are now
defined \cite[2.22]{Gi}
as the composition
\begin{align*}
K_j(X) & \to
H^{-j}(X,{\mathbf Z}\times B_{\scriptscriptstyle\bullet}GL(\so_X)^+)
\to H^{-j}(X,B_{\scriptscriptstyle\bullet}GL(\so_X)^+)\\
&  \stackrel{C^{\synt}_i}{\longrightarrow}  H^{-j}(X,
\sk(2i,\sss^{\prime}(i){}_{X}))
 \stackrel{h_j}{\rightarrow}  H^{2i-j}_{\synt}(X_h,i),
\end{align*}
where  $B_{\scriptscriptstyle\bullet}GL(\so_X)^+$ is the (pointed) simplicial sheaf on
$X$ associated to the $+\,$- construction \cite[4.2]{S}. Here, for a
(pointed) simplicial sheaf $\se_{\scriptscriptstyle\bullet} $ on $X_{\ZAR}$,
$H^{-j}(X,\se_{\scriptscriptstyle\bullet} ) =\pi_j(\R\Gamma
(X_{\ZAR},\se_{\scriptscriptstyle\bullet} ))$ is the generalized sheaf
cohomology of  $\se_{\scriptscriptstyle\bullet} $ \cite[1.7]{Gi}. The map $h_j$ is the Hurewicz map: 
\begin{align*} H^{-j}(X, \sk(2i,\sss^{\prime}(i){}_{X}
))  &  = \pi_j(\sk(2i,\sss^{\prime}(i)(X))) \!\stackrel{h_j}{\rightarrow}\!
H_j(\sk(2i,\sss^{\prime}(i)(X)))\\
 &  = H_j(\sss^{\prime}(i)(X)[2i])= H^{2i-j}_{\synt}(X_h,i).
\end{align*}
\end{proof}
\begin{proposition}
\label{syn-et}
The syntomic and the \'etale Chern classes are compatible, i.e.,
for $X\in {\mathcal V}ar_K$, $j\geq 0, 2i-j\geq 0$, the following diagram commutes
$$
\xymatrix{
& K_j(X)\ar[ld]_{c^{\synt}_{i,j}}\ar[rd]^{c^{\eet}_{i,j}} &\\
H^{2i-j}_{\synt}(X_h,i)\ar[rr]^{\rho_{\synt}} & & H^{2i-j}_{\eet}(X,\bq_p(i))}
$$
\end{proposition}
\begin{proof}
We can pass to the universal case ($X=B_{\scriptscriptstyle\bullet}GL_l:=B_{\scriptscriptstyle\bullet}GL_l/K$, $l\geq 1$). We have
\begin{align*}
H^*_{\synt}(B_{\scriptscriptstyle\bullet}GL_{l,h},*) & \simeq H^*_{\synt}(K,*)[x^{\synt}_1,\ldots,x^{\synt}_l],\\ H^*_{\eet}(B_{\scriptscriptstyle\bullet}GL_l,*) & \simeq H^*_{\eet}(K,*)[x^{\eet}_1,\ldots,x^{\eet}_l]
\end{align*}
By the projective space theorem and the fact that the syntomic period map commutes with products it suffices to check that $\rho_{\synt}(x_1^{\synt})=x_1^{\eet}$ and that the syntomic period map $\rho_{\synt}$ commutes with the classes $c_0^{\synt}: \bq_p\to \sss(0)_\bq$ and $c_0^{\eet}: \bq_p\to \bq_p(0)$. The statement about $c_0$ is clear from the definition of $\rho_{\crr}$; for $c_1$ consider   the canonical map $f: B_{\scriptscriptstyle\bullet}GL_l\to B_{\scriptscriptstyle\bullet}GL_{l,\ovk}$ and the induced pullback map
$$
f^*_{\eet}:\quad H^*_{\eet}(B_{\scriptscriptstyle\bullet}GL_l,*)=H^*_{\eet}(K,*)[x_1,\ldots,x_l]\to H^*_{\eet}(B_{\scriptscriptstyle\bullet}GL_{l,\ovk},*)=\bq_p[\overline{x}_1,\ldots,\ove{x}_l]
$$
that sends the Chern classes $x^{\eet}_i$ of the universal vector bundle to the classes $\overline{x}^{\eet}_i$ of its pullback. It suffices to show that $f^*_{\eet}\rho_{\synt}(C_{1,1}^{\synt})=C_{1,1}^{\eet}$. But, by definition,   $f^*_{\eet}\rho_{\synt}=\rho_{\synt}f^*_{\synt}$ and,  by construction, we have the following commutative diagram
$$
\xymatrix{
H^2_{\synt}(B_{\scriptscriptstyle\bullet}{\mathbb G}_{m,h},1)\ar[r]^{\can}\ar[d]^{\rho_{\synt}} & H^2_{\crr}(B_{\scriptscriptstyle\bullet}{\mathbb G}_{m,\ovk,h})\ar[d]^{\rho_{cr}}\\
  H^2_{\eet}(B_{\scriptscriptstyle\bullet}{\mathbb G}_{m,\ovk},\bq_p(1))\ar[r] & H^2_{\eet}(B_{\scriptscriptstyle\bullet}{\mathbb G}_{m,\ovk}, B^+_{\crr})= H^2_{\eet}(B_{\scriptscriptstyle\bullet}{\mathbb G}_{m,\ovk},\bq_p(1))\otimes B^+_{\crr}
 }
$$
where the bottom map sends the generator of $\bq_p(1)$ to the element $t\in B^+_{\crr}$ associated to it.
Since the syntomic and the crystalline Chern classes are compatible, it suffices to show that, for a line bundle $\sll$, $\rho_{\crr}(c_1^{\crr}(\sll))=c_1^{\eet}(\sll)\otimes t$. But this is \cite[3.2]{BE2}.
\end{proof}
\begin{remark}
If $\sx$ is a scheme over $V$ and $X=\sx_K$, we can consider the syntomic Chern classes $c_{i,j}^{\synt}: K_j(\sx)\to H^{2i-j}_{\synt}(X_h,i)$ defined as the composition
$$
K_j(\sx)\to K_j(X) \lomapr{c_{i,j}^{\synt}} H^{2i-j}_{\synt}(X_h,i).
$$
By the above proposition, these classes are compatible with the \'etale Chern classes. Recall that analogous results were proved earlier for $\sx$ smooth and projective \cite{NM}, for $\sx$ -- a complement of a divisor with relative normal crossings in such, and for $\sx$ - a semistable scheme over $V$ \cite{NT}.
\end{remark}
\subsection{Image of \'etale regulators}In this subsection we show that
Soul\'e's \'etale regulators factor through the semistable Selmer groups.

   Let $X\in {\mathcal V}ar_K$. For $2r-i-1\geq 0$, set 
$$
K_{2r-i-1}(X)_0:=\ker(K_{2r-i-1}(X)\lomapr{c^{\eet}_{r,i+1}} H^0(G_K,H^{i+1}_{\eet}(X_{\ovk},\bq_p(r))))
$$
Write $r^{\eet}_{r,i}$ for the map
$$r^{\eet}_{r,i}:K_{2r-i-1}(X)_0\to  H^1(G_K,H^{i}_{\eet}(X_{\ovk},\bq_p(r)))
$$
induced by the Chern class map $c^{\eet}_{r,i+1}$ and the Hochschild-Serre spectral sequence map $\delta: H^{i+1}_{\eet}(X,\bq_p(r))_0\to H^1(G_K,H^i_{\eet}(X_{\ovk},\bq_p(r)))$,
where we set $H^{i+1}_{\eet}(X,\bq_p(r))_0:=\ker( H^{i+1}_{\eet}(X,\bq_p(r))\to H^{i+1}_{\eet}(X_{\ovk},\bq_p(r))) $.
\begin{theorem}
\label{Tony}
The map $r^{\eet}_{r,i}$ factors through the subgroup $$   H^1_{\st}(G_K,H^{i+1}_{\eet}(X_{\ovk},\bq_p(r)))\subset H^1(G_K,H^{i+1}_{\eet}(X_{\ovk},\bq_p(r))).$$
\end{theorem}
\begin{proof}
By  Proposition \ref{syn-et}, we have the following commutative diagram
$$
\xymatrix{K_{2r-i-1}(X)\ar[d]^{c^{\synt}_{r,i+1}}\ar[dr]^{c^{\eet}_{r,i+1}}\\
H^{i+1}_{\synt}(X_h,r)\ar[r]^-{\rho_{\synt}} & H^{i+1}_{\eet}(X,\bq_p(r))\ar[r] & H^{i+1}_{\eet}(X_{\ovk},\bq_p(r))
}
$$
Hence the Chern class map $c^{\synt}_{r,i+1}:K_{2r-i-1}(X)_0\to H^{i+1}_{\synt}(X_h,r)$ factors through $H^{i+1}_{\synt}(X_h,r)_0:=\ker(H^{i+1}_{\synt}(X_h,r)\stackrel{\rho_{\synt}}{\to} H^{i+1}_{\eet}(X_{\ovk},\bq_p(r)))$. Compatibility of the syntomic descent  and the Hochschild-Serre spectral sequences (cf. Theorem \ref{stHS}) yields the following commutative diagram 
$$
\xymatrix{
K_{2r-i-1}(X)_0\ar[d]^{c^{\synt}_{r,i+1}}\ar[dr]^{c^{\eet}_{r,i+1}}\\
H^{i+1}_{\synt}(X_h,r)_0\ar[d]^{\delta}\ar[r]^{\rho_{\synt}} & H^{i+1}_{\eet}(X,\bq_p(r))_0\ar[d]^{\delta}\\
H^1_{\st}(G_K,H^{i}_{\eet}(X_{\ovk},\bq_p(r)))\ar[r]^{\can} & H^1(G_K,H^{i}_{\eet}(X_{\ovk},\bq_p(r)))
}
$$
Our theorem follows. 
\end{proof}
\begin{remark}The question of the image of Soul\'e's regulators $r_{r,i}^{\eet}$ was raised by Bloch-Kato in \cite{BK} in connection with their Tamagawa Number Conjecture. Theorem \ref{Tony} is known to follow from the constructions of Scholl \cite{Sc}. The argument goes as follows. Recall  that for a class $y\in K_{2r-i-1}(X)_0$ he constructs  an explicit extension $E_y\in \Ext^1_{\sm\sm_K}({\mathbf Q}(-r),h^{i}(X))$ in the category of mixed motives over $K$. The association $y\mapsto E_y$ is compatible with the \'etale cycle class and realization maps. By the de Rham Comparison Theorem, the \'etale realization $r^{\eet}_{r,i}(y)$ of the extension class $E_y$ in  $$ \Ext^1_{G_K}({\mathbf Q}_p(-r),H^{i}(X_{\ovk},\bq_p))=H^1(G_K,H^i_{\eet}(X_{\ovk},\bq_p(r)))$$ is de Rham, hence potentially semistable by \cite{BER}, as wanted.
\end{remark}
\appendix
\section {Vanishing of $H^2(G_K,V)$ by Laurent Berger}
%%
%% Version de : 9 octobre 2012
%%

Let $V$ be a $\Qp$-linear representation of $G_K$. In this appendix we prove
the following theorem.

\begin{theorem}
\label{main}
If $V$ is semistable and all its Hodge-Tate weights are $\geq 2$, then
$H^2(G_K,V)=0$.
\end{theorem}

Let $\dfont(V)$ be Fontaine's $(\phi,\Gamma)$-module attached to $V$ \cite{Fon}.
It comes with a Frobenius map $\phi$ and an action of $\Gamma_K$.
Let $H_K = \Gal(\ovk/K(\mu_{p^\infty}))$ and let $I_K =
\Gal(\ovk/K^{\nr})$. The injectivity of the restriction map
$H^2(G_K,V) \to H^2(G_L,V)$ for $L/K$ finite allows us to replace
$K$ by a finite extension, so that we can assume that $H_K I_K = G_K$
and that $\Gamma_K \simeq \Zp$. Let $\gamma$ be a topological generator
of $\Gamma_K$. Recall (\S I.5 of \cite{CC99}) that we have a map
$\psi : \dfont(V) \to \dfont(V)$.

 Ideally, our  proof of this theorem would go as follows. We use the Hochschild-Serre spectral sequence
$$
H^i(G_K/I_K,H^j(I_K,V|_{I_K})) \Rightarrow H^{i+j}(G_K,V)
$$
and, interpreting Galois cohomology in terms of  $(\phi,\Gamma)$-modules, we  compute that $H^2(I_K,V|_{I_K})=0$ and $H^1(I_K,V|_{I_K})=\hat{K}^{\nr}\otimes_KD_{\dr}(V)$.
We conclude since, by Hilbert 90,  $H^1(G_K/I_K,H^1(I_K,V|_{I_K})) =0$. However, we do not, in general, have Hochschild-Serre spectral sequences for continuous cohomology. We mimic thus the above argument with direct computations on continuous cocycles (again using $(\phi, \Gamma)$-modules). Laurent Berger is grateful to Kevin Buzzard for discussions related to the above spectral sequence.
\begin{lemma}
\label{cc}
\begin{enumerate}
\item If $V$ is a representation of $G_K$, then there is an exact sequence
$$0 \to \dfont(V)^{\psi=1} / (\gamma-1) \to H^1(G_K,V) \to
(\dfont(V)/(\psi-1))^{\Gamma_K} \to 0;$$
\item We have $H^2(G_K,V) = \dfont(V)/(\psi-1,\gamma-1)$.
\end{enumerate}
\end{lemma}

\begin{proof}
See I.5.5 and II.3.2 of \cite{CC99}.
\end{proof}

\begin{lemma}
\label{psimsur}
We have $\dfont(V|_{I_K})/(\psi-1)=0$
\end{lemma}

\begin{proof}
Since $V|_{I_K}$ corresponds to the case when $k$ is algebraically closed,
see the proof of Lemma VI.7 of \cite{L01}.
\end{proof}

Let $\gamma_I$ denote a generator of $\Gamma_{\widehat{K}^{\nr}}$.

\begin{lemma}
\label{psigam}
The natural map $\dfont(V|_{I_K})^{\psi=1} / (\gamma_I-1) \to
(\dfont(V|_{I_K}) / (\gamma_I-1))^{\psi=1}$ is an isomorphism
if $V^{I_K}=0$.
\end{lemma}

\begin{proof}
This map is part of the six term exact sequence that comes from the map
$\gamma_I-1$ applied to $0 \to \dfont(V|_{I_K})^{\psi=1} \to \dfont(V|_{I_K})
\xrightarrow{\psi-1} \dfont(V|_{I_K}) \to 0$. Its kernel is included in
$\dfont(V|_{I_K})^{\gamma_I=1}$ which is $0$, since $V^{I_K}=0$ (note that
the inclusion $(\widehat{K}^{\nr} \otimes V)^{G_K} \subseteq
(\widehat{\mathcal E}^{\nr} \otimes V)^{G_K} = \dfont(V)^{G_K}$ is an
isomorphism).
\end{proof}

Suppose that $x \in \dfont(V)/(\psi-1,\gamma-1)$. If $\tilde{x} \in \dfont(V)$
lifts $x$, then Lemma \ref{psimsur} gives us an element $y\in\dfont(V |_{I_K})$
such that $(\psi-1)y = \tilde{x}$. 
%Recall that $H_K I_K = G_K$ so that the map $I_K \to \Gamma_K$ is surjective;
%we identify the generators $\gamma_G$ and $\gamma_I$ for $G_K$ and for $I_K$. 
Define a cocycle $\delta(x) \in Z^1(G_K/I_K, \dfont(V |_{I_K})^{\psi=1} /
(\gamma_I-1))$ by $\delta(x) : \overline{g} \mapsto (g-1)(y)$ if $g \in G_K$
lifts $\overline{g} \in G_K/I_K$.

\begin{proposition}
\label{hsss}
If $V^{I_K}=0$, then the map
\[ \delta : \dfont(V)/(\psi-1,\gamma-1) \to H^1(G_K/I_K, (\dfont(V|_{I_K}) /
(\gamma_I-1))^{\psi=1}) \] 
is well-defined and injective.
\end{proposition}

\begin{proof}
We first check that $\delta(x)(g) \in (\dfont(V|_{I_K}) /
(\gamma_I-1))^{\psi=1}$. We have $(\psi-1)(g-1)(y) = (g-1)(x)$. If we write
$g = ih \in I_KH_K$, then $(g-1)x  = (ih-1)x=(i-1)x \in (\gamma_I-1)
\dfont(V|_{I_K})$ since $\gamma_I-1$ divides the image of $i-1$ in
$\Zp\dcroc{\Gamma_{\widehat{K}^{\nr}}}$. This implies that $\delta(x)(g)
\in (\dfont(V|_{I_K}) / (\gamma_I-1))^{\psi=1}$.

We now check that $\delta(x)$ does not depend on the choices. If we choose
another lift $g' \in G_K$ of $\overline{g} \in G_K/I_K$, then $g'=ig$ for some
$i \in I_K$ and $(g'-1)y-(g-1)y = (i-1)gy \in (\gamma_I-1) \dfont(V|_{I_K})$
since $\gamma_I-1$ divides the image of $i-1$ in
$\Zp\dcroc{\Gamma_{\widehat{K}^{\nr}}}$. If we choose another $y'$ such that
$(\psi-1)y'=\tilde{x}$, then $y-y' \in \dfont(V|_{I_K})^{\psi=1}$ so that
$\delta$ and $\delta'$ are cohomologous. Finally, if $\tilde{x}'$ is another
lift of $x$, then $\tilde{x}' - \tilde{x} = (\gamma-1)a + (\psi-1)b$ with
$a,b \in \dfont(V)$. We can then take $y' = y + b + (\gamma_G-1)c$ where
$(\psi-1)c=a$. We then have $(g-1)y'=(g-1)y + (g-1)b + (\gamma_G-1)(g-1)c$.
Since $G_K=I_K H_K$, we can write $g=ih$ and $(g-1)b=(i-1)b$. Using $G_K=I_K
H_K$ once again, we see that $I_K \to G_K/H_K$ is surjective, so that we can
identify $\gamma_I$ and $\gamma_G$. The resulting cocycle is then cohomologous
to $\delta(x)$. This proves that $\delta$ is well-defined.

We now prove that $\delta$ is injective. If $\delta(x)=0$, then using Lemma
\ref{psigam} there exists $z \in \dfont(V|_{I_K})^{\psi=1}$ such that
$\delta(x)(\overline{g})$ is the image of $(g-1)(z)$ in
$\dfont(V|_{I_K})^{\psi=1} / (\gamma_I-1)$. This implies that $(g-1)(y-z)
\in (\gamma_I-1) \dfont(V|_{I_K})^{\psi=1}$. Applying $\psi-1$ gives
$(g-1)\tilde{x} = 0$ so that $\tilde{x} \in \dfont(V)^{G_K} \subset V^{I_K}=0$.
The map $\delta$ is therefore injective.
\end{proof}

\begin{lemma}
\label{expbij}
If $V$ is semistable and the weights of $V$ are all $\geq 2$, then $\exp_V :
\ddr(V|_{I_K}) \to H^1(I_K,V)$ is an isomorphism.
\end{lemma}

\begin{proof}
Apply Thm. 6.8 of \cite{BER} to $V|_{I_K}$.
\end{proof}

\begin{proof}[Proof of Theorem \ref{main}]
We can replace $K$ by $K_n$ for $n \gg 0$ and use the fact that if
$H^2(G_{K_n},V) = 0$, then $H^2(G_K,V) = 0$ since the restriction map is
injective. In particular, we
can assume that $H_K I_K = G_K$ and that $\Gamma_K$ is isomorphic to $\Zp$. By
item (2) of Lemma \ref{cc}, we have $H^2(G_K,V) = \dfont(V)/(\psi-1,\gamma-1)$,
and so by Proposition \ref{hsss} above, it is enough to prove that 
\[ H^1(G_K/I_K, (\dfont(V|_{I_K}) / (\gamma_I-1))^{\psi=1}) = 0. \]
Lemma \ref{psigam} tells us that $(\dfont(V|_{I_K}) / (\gamma_I-1))^{\psi=1} =
\dfont(V|_{I_K})^{\psi=1} / (\gamma_I-1)$. Since $\dfont(V|_{I_K})/(\psi-1)=0$
by Lemma \ref{psimsur}, item (1) of Lemma \ref{cc} tells us that
$\dfont(V|_{I_K})^{\psi=1} / (\gamma-1) = H^1(I_K,V)$.

The map $\exp_V : \ddr(V|_{I_K}) \to H^1(I_K,V)$ is an isomorphism by Lemma
\ref{expbij}, and this isomorphism commutes with the action of $G_K$ since it
is a natural map. We therefore have $H^1(I_K,V) = \widehat{K}^{\nr} \otimes_K
\ddr(V)$ as $G_K$-modules. It remains to observe that the cocycle
$\delta(x) \in Z^1(G_K/I_K, \widehat{K}^{\nr} \otimes_K \ddr(V))$
is continuous and that $H^1(G_K/I_K, \widehat{K}^{\nr})=0$ by taking
a lattice, reducing modulo a uniformizer of $K$, and applying Hilbert 90.
\end{proof}

\providecommand{\bysame}{\leavevmode ---\ }
\providecommand{\og}{``}
\providecommand{\fg}{''}
\providecommand{\smfandname}{\&}
\providecommand{\smfedsname}{\'eds.}
\providecommand{\smfedname}{\'ed.}
\providecommand{\smfmastersthesisname}{M\'emoire}
\providecommand{\smfphdthesisname}{Th\`ese}
\section{The Syntomic ring spectrum by Fr\'ed\'eric D\'eglise}

In this appendix, we explain why syntomic cohomology
 as defined in this paper is representable by a motivic
 ring spectrum in the sense of Morel and Voevodsky's homotopy
 theory. More precisely, we will exhibit a monoid
 object $\sss$ of the triangulated category of motives
 with $\Qp$-coefficients (see below),
 $DM$, such that for any variety $X$ and any pair of integers
 $(i,r)$,
$$
H^i_{\synt}(X_h,r)=\Hom_{DM}(M(X),\sss(r)[i]).
$$
In fact, it is possible to apply directly \cite[Th. 1.4.10]{DM1}
 to the graded commutative dg-algebra 
 $\R\Gamma_{\synt}(X,*)$ of Theorem \ref{main1} in view
 of the existence of Chern classes established in Section 5.1.
However, the use of $h$-topology in this paper
 makes the construction of $\E_{synt}$ much more straightforward
 and that is what we explain in this appendix.
 Reformulating slightly the  original definition of Voevodsky
 (see \cite{V1}),
 we introduce:
\begin{definition}
Let $\PSh(K,\Qp)$ be the category of presheaves
 of $\Qp$-modules over the category of varieties.

Let $C$ be a complex in $\PSh(K,\Qp)$. We say:
\begin{enumerate}
\item $C$ is $h$-local if for any h-hypercovering
  $\pi:Y_\bullet \rightarrow X$,
 the induced map:
$$
C(X) \rightarrow{\pi_*} \mathrm{Tot}^\oplus(C(Y_\bullet)) 
$$
is a quasi-isomorphism;
\item $C$ is $\Af$-local if for any variety $X$, the map induced
 by the projection:
$$
H^i(X_h,C) \rightarrow H^i(\Af_{X,h},C)
$$
is an isomorphism.
\end{enumerate}
We define the triangulated category $\DMe_h(K,\Qp)$
 of effective $h$-motives as the full subcategory
 of the derived category $D(\PSh(K,\Qp))$ made
 by the complexes which are $h$-local and $\Af$-local.
\end{definition}
Equivalently, we can define this category
 as the $\Af$-localization of the derived category
 of $h$-sheaves on $K$-varieties
  (see \cite{CD3}, Sec. 5.2 and more precisely Prop. 5.2.10, Ex. 5.2.17(2)).
 Recall also from \emph{loc. cit.},
 that there are derived tensor products and internal $\Hom$
 on $\DMe_h(K,\Qp)$.

For any integer $r \geq 0$,
 the \emph{syntomic sheaf} $\sss(r)$
 is both $h$-local (by definition) and $\Af$-local (Prop. \ref{homotopy}).
 Thus it defines an object of $\DMe_h(K,\Qp)$
 and for any variety $X$, one has an isomorphism:
$$
\Hom_{\DMe_h(K,\Qp)}(\Qp(X),\sss(r)[i])=
\Hom_{D(\PSh(K,\Qp))}(\Qp(X),\sss(r)[i])
=H^{i}_{\synt}(X_h,r)
$$
where $\Qp(X)$ is the presheaf of $\Qp$-vector spaces represented by $X$.
Thus, the representability assertion for syntomic cohomology
 is obvious in the effective setting.

Recall that one defines the Tate motive in $\DMe_h(K,\Qp)$
 as the object $\Qp(1):=\Qp(\mathbb P^1_K)/\Qp(\{\infty\})[-2]$.
 Given any complex object $C$ of $\DMe_h(K,\Qp)$, we put:
 $C(n):=C \otimes \Qp(1)^{\otimes,n}$.
 One should be careful that this notation is in conflict with
 that of $\sss(r)$ considered as an effective $h$-motive,
 as the natural twist on syntomic
 cohomology is unrelated to the twist of $h$-motives.
 To solve this matter, we are led to consider
 the following notion of Tate spectrum, borrowed from algebraic topology
 according to Morel and Voevodsky.
\begin{definition}
A \emph{Tate $h$-spectrum} (over $K$ with coefficients in $\Qp$),
 is a sequence $\E=(E_i,\sigma_i)_{i \in \mathbb N}$
 such that:
\begin{itemize}
\item for each $i \in \mathbb N$, $E_i$ is a complex of
 $\PSh(K,\Qp)$ equipped with an action of 
the symmetric group $\Sigma_i$ of the set with $i$-element,
\item for each $i \in \mathbb N$, $\sigma_i:E_i(1) \rightarrow E_{i+1}$
 is a morphism of complexes --
 called the \emph{suspension map} in degree $i$,
\item For any integers $i \geq 0$, $r>0$,
 the map induced by the morphisms $\sigma_i,\cdots,\sigma_{i+r}$:
$$
E_i(r) \rightarrow E_{i+r}
$$
is compatible with the action of $\Sigma_i \times \Sigma_r$,
 given on the left by the structural $\Sigma_i$-action
  on $E_i$ and the action of $\Sigma_r$ via the permutation
  isomorphism of the tensor structure on $C(\PSh(K,\Qp))$,
 and on the right via the embedding
  $\Sigma_i \times \Sigma_r \rightarrow \Sigma_{i+r}$.
\end{itemize}
A morphism of Tate $h$-spectra $f:\E \rightarrow \mathbb F$ is a sequence
 of  $\Sigma_i$-equivariant maps
  $(f_i:E_i \rightarrow F_i)_{i \in \mathbb N}$ 
 compatible with the suspension maps.
 The corresponding category will be denoted by $\Sp(K,\Qp)$.
\end{definition}
There is an adjunction of categories:
\begin{equation} \label{eq:suspension}
\Sigma^\infty:C\big(\PSh(K,\Qp)\big) \leftrightarrows \Sp(K,\Qp):\Omega^\infty
\end{equation}
such that for any complex $K$ of $h$-sheaves, $\Sigma^\infty C$
 is the Tate spectrum equal in degree $n$ to $C(n)$, equipped with the obvious
 action of $\Sigma_n$ induced by the symmetric structure on tensor product
 and with the obvious suspension maps. 
% A morphism $f$ as above is called a \emph{level weak equivalence}
%  if for any integer $i \geq 0$, the morphism of complexes $f_i$
%  is a quasi-isomorphism. We denote by $D_{Tate}(S,\Lambda)$
%  the localization of $\Sp(S,\Lambda)$ with respect to
%  level weak equivalences (See \cite[Sec. 1.4]{CD2}).
\begin{definition}
A morphism of Tate spectra   $(f_i:E_i \rightarrow F_i)_{i \in \mathbb N}$ 
 is a level quasi-isomorphism if for any $i$,
 $f_i$ is a quasi-isomorphism.

A Tate spectrum $\E$ is called a $\Omega$-spectrum if for any $i$,
 $E_i$ is $h$-local and $\Af$-local and the map of complexes
$$
E_i \rightarrow \underline{\Hom}(\Qp(1),E_{i+1})
$$
is a  quasi-isomorphism.

We define the triangulated category $\DM_h(K,\Qp)$
 of $h$-motives over $K$ with coefficients in $\Qp$ as the  category of
 Tate $\Omega$-spectra localized by the level quasi-isomorphisms.
\end{definition}
The category of $h$-motives notably enjoys the following properties:
\begin{enumerate}
\item[(DM1)] The adjunction of categories \eqref{eq:suspension}
 induces an adjunction of triangulated categories:
$$
\Sigma^\infty:\DMe_h(K,\Qp) \leftrightarrows \DM_h(K,\Qp):\Omega^\infty
$$
such that for a Tate $\Omega$-spectrum $\E$, and any integer
 $r \geq 0$, $\Omega^\infty(\E(r))=E_r$
 (see \cite[Sec. 5.3.d, and Ex. 5.3.31(2)]{CD3}).

Given any variety $X$, we define the (stable)
 $h$-motive of $X$ as $M(X):=\Sigma^\infty \Qp(X)$.
\item[(DM2)] There exists a symmetric closed monoidal structure on $\DM(K,\Qp)$
 such that $\Sigma^\infty$ is monoidal and such that
 $\Sigma^\infty \Qp(1)$ admits a tensor inverse
  (see \cite[Sec. 5.3, and Ex. 5.3.31(2)]{CD3}).
 By abuse of notations, we put: $\Qp=\Sigma^\infty \Qp$.
\item[(DM3)] The triangulated monoidal category $\DM_h(K,\Qp)$ is equivalent
 to all known versions of triangulated categories of mixed motives
 over $\Spec(K)$ with coefficients in $\Qp$
  (see \cite[Sec. 16, and Th. 16.1.2]{CD3}).
In particular, it contains as a full subcategory
 the category $\DM_{gm}(K) \otimes \Qp$ obtained from the category
 of Voevodsky geometric motives (\cite[chap.5]{FSV})
 by tensoring $\Hom$-groups with $\Qp$
 (see \cite[Cor. 16.1.6, Par. 15.2.5]{CD3}).
\end{enumerate}

With that definition, the construction of a Tate spectrum
 representing syntomic cohomology is almost obvious.
 In fact,
 we consider the sequence of presheaves 
$$
\sss:=(\sss(r), r \in \mathbb N),
$$
where each $\sss(r)$ with the trivial action of $\Sigma_r$.
 According to the first paragraph of Section 5.1,
 we can consider 
 the first Chern class of the canonical invertible sheaf $\mathbb P^1$:
 $\bar c \in H^2_{\synt}(\mathbb P^1_K,1)=H^2(\mathbb P^1_{K,h},\sss(1))$.
 Take any lift
  $c:\Qp(\mathbb P^1_K) \rightarrow \sss(1)[2]$ of this
        class. By the definition of the Tate twist, it defines an element
         $\Qp(1) \rightarrow \sss(1)$ still denoted by $c$.
 We define the suspension map:
$$
\sss(r) \otimes \Qp(1) \lomapr{\id\otimes c} \sss(r) \otimes \sss(1)
 \xrightarrow{\mu} \sss(r+1)
$$
 where $\mu$ is the multiplication coming from the graded
 dg-structure on $\sss(*)$.
Because this dg-structure is commutative, we obtain
 that these suspension maps induce a structures of a Tate spectrum
 on $\sss$.
 Moreover, $\sss$ is a Tate $\Omega$-spectrum because each
 $\sss(r)$ is $h$-local and $\Af$-local, and the map obtained by
 adjunction from $\sigma_r$ is a quasi-isomorphism because
 of the projective bundle theorem for $\mathbb P^1$ (an easy case
 of Proposition \ref{projective}).

Now, by definition of $\DM_h(K,\Qp)$ and because of property (DM1) above,
 for any variety $X$, and any integers $(i,r)$,
 we get:
$$
\Hom_{\DM_h(K,\Qp)}(M(X),\sss(r)[i])
 =\Hom_{\DMe_h(K,\Qp)}(\Qp(X),\Omega^\infty(\sss(r))[i])
 =H^i_{\synt}(X_h,r).
$$
Moreover, the commutative dg-structure on the complex $\sss(*)$
 induces a monoid structure on the associated Tate spectrum.
 In other words, $\sss$ is a ring spectrum (strict and commutative).
 This construction is completely analogous to the proof of
 \cite[Prop. 1.4.10]{DM1}. In particular, we can apply all the constructions
 of \cite[Sec. 3]{DM1} to the ring spectrum $\sss$. Let us summarize
 this briefly:
\begin{proposition}
\label{Bloch-Ogus}
\begin{enumerate}
\item Syntomic cohomology is covariant with respect
 to projective morphisms of smooth varieties
  (Gysin morphisms in the terminology of \cite{DM1}). More precisely, to  a projective morphism of smooth $K$-varieties $f: Y\to X$ one can associate a Gysin morphism in syntomic cohomology
  $$f_*: H^n_{\synt}(Y_h,i)\to H^{n-2d}_{\synt}(X_h,i-d), $$
  where $d$ is the dimension of $f$. 
\item The syntomic regulator over $\Qp$ is induced by the unit
 $\eta:\Qp \rightarrow \sss$
 of the ring spectrum $\sss$:
\begin{align*}
r_{\synt}:H^{r,i}_M(X) \otimes \Qp=&\Hom_{\DM_h(K,\Qp)}(M(X),\Qp(r)[i]) \\
& \longrightarrow \Hom_{\DM_h(K,\Qp)}(M(X),\sss(r)[i])=H^i_{\synt}(X_h,r).
\end{align*}
It is compatible with product, pullbacks and pushforwards.
\item The syntomic cohomology has a natural extension to 
 $h$-motives\footnote{and in particular to the usual Voevodsky geometrical
 motives by (DM3) above.}:
$$
\DM_h(K,\Qp)^{op} \rightarrow D(\Qp), \quad M \mapsto \Hom_{\DM_h(K,\Qp)}(M,\sss)
$$
and the syntomic regulator $r_{\synt}$ can be extended to motives.
\item There exists a canonical syntomic Borel-Moore homology
 $H^{\synt}_*(-,*)$ such that the pair of functors
 $(H_{\synt}^*(-,*),H^{\synt}_*(-,*))$ defines a Bloch-Ogus theory.
\item To the ring spectrum $\sss$ there is associated a cohomology with
 compact support satisfying the usual properties.
\end{enumerate}
\end{proposition}
For points (1) and (2), we refer the reader to \cite[Sec. 3.1]{DM1}
 and for the remaining ones to \cite[Sec. 3.2]{DM1}.

\begin{remark}
Note that the construction of the syntomic ring spectrum
 $\sss$ in $\DM_h(K,\Qp)$ automatically yields the general projective
 bundle theorem (already obtained in Prop. \ref{projective}).
 More generally, the ring spectrum $\sss$ is \emph{oriented}
 in the terminology of motivic homotopy theory.
 Thus, besides the theory of Gysin morphisms,
 this gives various constructions
  -- symbols, residue morphisms -- and yields various formulas 
 -- excess intersection formula, blow-up formulas
 (see \cite{Deg8} for more details).
\end{remark}
\printindex

\end{document}